%% file: Master-arXiv.tex
\documentclass[oneside,a4paper,english, 11pt]{amsart}

\input{artmymacros.sty}
\usepackage{makecell, multirow, tikz-cd}

\title{Gelfand--Kirillov dimension and mod~$p$ cohomology for inner forms of~$\GL_2$}

\author{Andrea Dotto}
\address{
    Department of Mathematics,
    King's College London,
    Strand, London, WC2R 2LS, United Kingdom
}
\email{andrea.dotto@kcl.ac.uk}

\author{Bao V.~Le Hung}
\address{Department of Mathematics,
Northwestern University, 
2033 Sheridan Road, 
Evanston, Illinois 60208, USA}
\email{lhvietbao@googlemail.com}

\let\det\relax
\DeclareMathOperator{\det}{det}
\newcommand{\fourmatrix}[4]{\begin{pmatrix} #1 & #2 \\ #3 & #4 \end{pmatrix}}
\newcommand{\diag}{\operatorname{diag}}

\newcommand{\mJ}{\mathcal{J}}
\newcommand{\Syl}{\operatorname{Syl}}
\newcommand{\rec}{\operatorname{rec}}
\newcommand{\cInd}{\text{c-Ind}}

\begin{document}

\maketitle

\begin{abstract}
    Under standard assumptions, we compute the GK-dimension of Hecke eigenspaces in the mod~$p$ cohomology of an inner form~$D^\times$ of~$\GL_2$ over a totally real field unramified at~$p$, 
    allowing~$D$ to be a division algebra at~$p$.
    Our arguments also apply when~$D$ is a matrix algebra at~$p$, in which case they give a simplified proof of a theorem of Breuil--Herzig--Hu--Morra--Schraen.    
\end{abstract}

\setcounter{tocdepth}{2}
\tableofcontents

\input{introduction-arXiv}

\input{Iwasawa-arXiv}

\input{deformationrings-arXiv}

\input{multiplicityone-arXiv}

\input{global-arXiv}

\appendix

\input{formulas-arXiv}

\bibliography{Biblio}
\bibliographystyle{amsalpha}

\end{document}

%% file: introduction-arXiv.tex
\section{Introduction.}

Fix a prime number~$p$, a totally real field $F/\bQ$ in which $p$ is unramified, and a place~$v$ of~$F$ over~$p$.
Let~$D/F$ be a totally definite quaternion algebra, and let~$E$ be a sufficiently large finite extension of~$\bQ_p$, with ring of integers~$\cO$ and residue field~$\bF$. 
The space of algebraic modular forms
\[
\pi := \varinjlim_{K_v \subset D_v^\times} H^0(D^\times \backslash (D \otimes_F \bA_F^{\infty})^\times/K^vK_v, \bF)
\]
of tame level $K^v \subset (D \otimes_F \bA_F^{\infty, v})^\times$ has commuting actions of the spherical Hecke algebra
\[
\bT := \bF[T_w, S_w^{\pm 1}: \text{$D_w$ is split and $K_w \subset D_w^\times$ is maximal}]
\]
and of the $p$-adic Lie group~$D_v^\times$.
The eigenspaces~$\pi[\fm]$ corresponding to maximal ideals of~$\bT$ have been extensively studied from the point of view of local-global compatibility in the mod~$p$ Langlands correspondence:
when~$\pi[\fm] \ne 0$, one has a semisimple Galois representation
\[
\lbar r_\fm : \Gal_F \to \GL_2(\bF)
\]
associated to~$\fm$, and when~$\lbar r_\fm$ is absolutely irreducible, one expects $\pi[\fm]$ to be determined by~$\rhobar_v := \lbar r_\fm |_{\Gal_{F_v}}$ under the mod~$p$ local Langlands correspondence for~$D_v^\times$.
The breakthrough works~\cite{BHHMS, HuWangGKD, BHHMSII} have established several properties of~$\pi[\fm]$ in the case that~$D_v^\times \cong \GL_2(F_v)$, and~$\fm$ satisfies certain technical assumptions
(most notably, the Taylor--Wiles assumption, and a genericity assumption at~$p$).
In this paper, we remove the assumption that~$D_v^\times \cong \GL_2(F_v)$ is split from the main results of~\cite{BHHMS, HuWangGKD}.

\begin{thm}\label{main theorem}
Assume that~$\pi[\fm] \ne 0$, and that the following are true:
\begin{enumerate}
    \item $\lbar r_\fm|_{\Gal_{F_w}}$ is $1$-generic for all~$w \mid p$, and $9$-generic for $w = v$ (see Section~\ref{subsec: single type deformation rings} for our genericity conditions);
    \item $\lbar r_\fm|_{\Gal_{F(\zeta_p)}}$ is absolutely irreducible;
    \item if $w \nmid p$, and~$\rhobar_w$ or~$D_w$ is ramified, then the universal lifting ring of~$\rhobar_w$ with $\cO$-coefficients is formally smooth over~$\cO$; and
    \item $K^v$ is the minimal level for~$\lbar r_\fm$.
\end{enumerate}
Then $\dim_{D_v^\times} \pi[\fm] = [F_v:\bQ_p]$.
\end{thm}

The symbol $\dim_{D_v^\times} \pi[\fm]$ denotes the canonical dimension of~$\pi[\fm]$ (also referred to as Gelfand--Kirillov dimension in~\cite{BHHMS}).
Assumption~(4) is explained in detail in Section~\ref{coefficients}.
As usual, it requires us to fix types at the ramified places, and Hecke operators at an auxiliary prime~$w_1$; see Theorem~\ref{main theorem with all assumptions} for a precise statement. 

\begin{rmk}
Theorem~\ref{main theorem} was previously known when $F_v = \bQ_p$ (see~\cite[Theorem~1.1]{HuWangD*} for the non-split case) or 
when $D_v^\times \cong \GL_2(F_v)$ (see~\cite[Theorem~1.1]{BHHMS} and~\cite[Theorem~1.1]{HuWangGKD}).
Our arguments work uniformly for split and non-split inner forms of~$\GL_2(F_v)$, and in the split case they give a simplified proof of the main results of~\cite{BHHMS, HuWangGKD}, 
under slightly less restrictive genericity conditions than~\cite{BHHMS}.
Furthermore, our arguments also apply to the case of Shimura curves arising from division algebras that are split at precisely one infinite place of~$F$.
\end{rmk}

We now sketch the proof of Theorem~\ref{main theorem}, focussing on the points where our arguments differ from those already available in the literature.
It suffices to prove the inequality
\begin{equation}\label{bound to prove}
\dim_{D_v^\times}\pi[\fm] \leq [F_v:\bQ_p],
\end{equation}
since the opposite inequality follows from the results of~\cite{GN}.
Let~$H \subset D_v^\times$ be the maximal compact subgroup (if~$D_v$ is not split) or the upper-triangular Iwahori subgroup (if~$D_v$ is split).
For each smooth character $\chi: H \to \bF^\times$, let $W_{\chi, 3}$ be the $3$-truncated projective envelope of~$\chi$ described in Definition~\ref{defn:Wchi,n}.
Then, as shown in~\cite[Theorem~1.6]{BHHMS} and~\cite[Corollary~2.12]{HuWangD*}, a sufficient condition for~\eqref{bound to prove} to hold
is that for all~$\chi$, the pullback map
\begin{equation}\label{cyclicity to prove}
\Hom_{\bF[H]}(\chi, \pi[\fm]) \to \Hom_{\bF[H]}(W_{\chi, 3}, \pi[\fm])
\end{equation}
is an isomorphism.
We will show this by proving that both sides have dimension one. 
In other words, we establish ``multiplicity one theorems'' for~$\chi$ and~$W_{\chi, 3}$.

We now describe the proof of multiplicity one for~$\chi$.
This result is new when~$D_v$ is a division algebra and~$f>1$; when $D_v$ is split, it follows from~\cite[Theorem~10.1.1]{EGS}, although our argument is different.
More specifically, for any $H$-representation $V$, a powerful technique to analyze $\Hom_{\bF[H]}(V, \pi[\fm])$ was introduced in~\cite{EGS}. 
The Taylor--Wiles--Kisin patching method produces a functor~$M_\infty$, from the category of $D_v^\times$-representations on $\cO$-modules, to modules over the universal deformation rings
of~$\rhobar_v$.
The functor~$M_\infty$ enjoys several exactness and rigidity properties, and there is an identification of $\Hom_{\bF[H]}(V, \pi[\fm])$ with 
$M_\infty(\cInd_H^{D^\times_v} V)/\fm$, the special fiber of the patched module $M_\infty(\cInd_H^{D^\times_v} V)$.

When~$D_v^\times$ is split and~$H$ is the Iwahori subgroup, the functor~$\cInd_{H}^{D_v^\times}(-)$ factors additionally through $\cInd_H^{\GL_2(\cO_{F_v})}(-)$.
There are many congruences between lattices in tame irreducible $E[\GL_2(\cO_{F_v})]$-modules, which are reflected in the intersection properties of tamely crystalline deformation spaces 
inside the spectrum of the universal deformation ring of~$\rhobar_v$.
Building on this observation, \cite{EGS} proved that the left-hand side of~\eqref{cyclicity to prove} has dimension 1 when $D_v$ is split. 
This observation cannot be replicated when~$D_v$ is not split, since the tame $E$-representations of $H = \cO_{D_v}^\times$ 
are one-dimensional, and so we need a new argument.
In the case of Shimura curves, one possibility is to exploit the geometry of a semistable model at level~$1+\fm_{D_v}$
to reduce to the split case~\cite{BKmultiplicity}, and in the totally definite case over~$\bQ$ one can exploit the uniformization of the supersingular locus on the modular curve
(see~\cite{Chengyang}).
In this paper, we develop a different technique, which works uniformly for Shimura curves and totally definite quaternion algebras, 
building on an analogue with $p$-adic coefficients of the methods of~\cite{Manning}. 
This is the reason we impose Assumption~(4) in Theorem~\ref{main theorem}.

We now turn to the proof of multiplicity one for~$W_{\chi, 3}$.
When $D_v^\times$ is split, this was established in~\cite{BHHMS}, using the following strategy:
\begin{itemize}
\item Construct a proper quotient~$L$ of $\cInd_H^{\GL_2(\cO_{F_v})}(W_{\chi, 3})$, and reduce the problem to proving multiplicity one for~$L$.  
\item Lift~$L$ to an $\cO$-lattice $L^\circ$ in a direct sum of tame locally algebraic representations of $\GL_2(\cO_{F_v})$.
\item Compute the dimension of $M_\infty(\cInd_{\GL_2(\cO_{F_v})}^{D_v^\times} L^\circ)/\fm$, via a finer study of intersections of tamely crystalline deformation spaces. 
\end{itemize} 
More precisely, the representation~$L$ is a $2$-truncated projective envelope of the irreducible $\F[\GL_2(\cO_{F_v})]$-module with highest weight~$\chi$.
The implementation in~\cite{BHHMS} of the first two steps of this outline involves some fairly technical arguments about the representation theory of $\GL_2(\cO_{F_v})$ in mixed characteristic. 
By contrast, a key simplification in our approach is to work directly with~$H$, whose representation theory is simpler, as its irreducible representations are one-dimensional.
This comes at an expense for the last step: the patched modules that we need to analyze become larger than those in ~\cite{BHHMS}. 
Fortunately, this can be handled by techniques introduced in~\cite{LLHMK1}. 
The statements about deformation rings that ultimately solve the problem are given in 
Propositions~\ref{higherweight}, \ref{ideals for arm cyclicity} and~\ref{ideals for Wchi3}, which we prove by explicit calculation, 
making use of the local model theory of~\cite{LLHLMmodels,LLHMK1}.
In fact, we build upon these results, and we compute in Proposition~\ref{multitype}
a local model for a union of deformation spaces where the weights and inertial types are allowed to vary simultaneously.

\begin{rmk} 
When~$D_v$ is not split, it seems unavoidable to work directly with~$W_{\chi, 3}$, since in this case, $H$ coincides with the maximal compact subgroup of~$D_v^\times$,
and so the induction functor~$\cInd_{H}^{D_v^\times}(-)$ cannot be further simplified.
As previously mentioned, the case of Theorem~\ref{main theorem} with~$D_v$ nonsplit and~$f = 1$ is the main result of~\cite{HuWangD*}.
The paper~\cite{HuWangD*} also works directly with~$W_{\chi, 3}$, but takes a different approach to the problem of lifting it to characteristic zero, and proving multiplicity one.
Most notably, the requisite computations in Galois deformation theory are performed by reducing to the split case and applying the $p$-adic local Langlands correspondence
for~$\GL_2(\bQ_p)$, which is not available when~$f > 1$. 
\end{rmk}

\subsection{Acknowledgments.}
AD is grateful to Yongquan Hu for some preliminary exchanges on the subject of this paper.
AD was supported by a Royal Society University Research Fellowship.

BLH ~acknowledges support from the National Science Foundation under grants Nos.~DMS-1952678 and DMS-2302619 and the Alfred P.~Sloan Foundation.

The authors express their gratitude to the special trimester ``The Arithmetic of the Langlands Program'', held at
the Hausdorff Institute of Mathematics, where parts of this paper were written.

\subsection{Notation and preliminaries.}\label{subsect: preliminaries}
Fix a prime number~$p > 5$ and an algebraic closure $\lbar \bQ_p/\bQ_p$.
Let~$E \subset \overline{\bQ}_p$ be a finite unramified extension of $\bQ_p$ with ring of integers~$\cO$, uniformizer~$p$ 
and residue field~$\bF$.
Let~$k$ be a finite extension of~$\bF_p$, of degree~$f$, and let~$K$ be the unramified extension of~$\bQ_p$ with residue field~$k$.
Write~$q = p^f$.
We fix an algebraic closure~$\lbar K / K$, and write~$\lbar \bF$ for the residue field of~$\lbar E = \lbar \bQ_p$, which is an algebraic closure of~$\bF$.
We write~$\Gal_K := \Gal(\lbar K/K)$, and use similar notation for the absolute Galois group of any field with a fixed algebraic closure.
For every positive integer~$n$ we write~$K_n$ for the unramified extension of~$K$ in~$\lbar K$ of degree~$n$, and~$k_n$ for its residue field.
Recall the character
\[
\omega_n: \Gal_{K_n} \to k_n^\times
\]
defined by $g(p^{1/(p^{n}-1)}) = \omega_n(g)p^{1/(p^{n}-1)}$.
We write~$\cJ$ for the set of ring homomorphisms $k \to \bF$, and we assume that $\bF$ is large enough that the quadratic extension of~$k$ embeds in~$\bF$.

We will write~$\Frob_{p^i}$ for the $p^i$-th power morphism on any commutative $\bF_p$-algebra.
We say that an ideal~$J$ in a ring~$R$ is $p$-saturated if $R/J$ is $p$-torsion free.
In general, the $p$-saturation of~$J$ is the preimage in~$R$ of $(R/J)[p^\infty]$.

Let~$D$ be a quaternion division algebra with centre~$K$, and write~$\cO_D$, resp.\ $k_D$ for its ring of integers, resp.\ residue field.
We fix a uniformizer~$\Pi_D$, and write
\[
U^n_D = 1 + \Pi_D^n \cO_D.
\]
We write~$I$ for the upper-triangular Iwahori subgroup of~$\GL_2(K)$, and $\cO_D^{\times}$ for the group of units of~$\cO_D$.
We will typically be working with a pair $(G, H) \in \{(\GL_2(K), I), (D^\times, \cO_D^\times)\}$. 
In this situation, $H$ has a unique pro-$p$ Sylow subgroup~$H_1$, which equals the pro-$p$ Iwahori subgroup~$I_1$ when~$H = I$, and equals the 
group~$U^1_D$ of $1$-units when~$H = \cO_D^\times$.
We write~$Z \cong K^\times$ for the centre of~$G$, and $Z_1 := Z \cap H_1$.

\subsubsection{Reductive groups.}
Let~$G_0 := \Res_{\cO_K/\bZ_p} \GL_2$, and define similarly the upper-triangular Borel subgroup~$B_0$ with unipotent radical~$U_0$, the lower-triangular unipotent subgroup~$U^-_0$, and 
the diagonal maximal torus~$T_0$.
Let $\uline G := G_0 \times_{\bZ_p} \cO$, and define similarly~$\uline B, \uline U, \uline U^-, \uline T$.
Let~$\uline W := W(\uline B, \uline T)$ be the Weyl group, $\Phi := \Phi(\uline G, \uline T)$ be the root system 
and $\Phi^+ := \Phi(\uline B, \uline T)$ be the set of $\uline B$-positive roots.
For all~$j \in \cJ$ we let~$\alpha_j \in \Phi^+$ be the positive root that is nonzero precisely at the embedding~$j$.
Then~$\Phi^+ = \{\alpha_j: j \in \cJ\}$.

If~$a, b \in \bZ$, we write $(a, b) \in X^*(\uline T)$ for the character 
\[
\sum_{j \in \cJ}(a, b)_j \in (\bZ^2)^{\cJ} \cong X^*(\uline T),
\]
i.e.\ we identify~$\bZ^2$ with its image under the diagonal embedding $\bZ^2 \to X^*(\uline T)$.
We let $\eta := (1, 0)$, and we write~$\uline C_0$ for the lowest $\eta$-shifted $p$-alcove in $X^*(\uline T)$.

Let $\tld{\uline W} := \uline W \ltimes X^*(\uline T)$ be the extended Weyl group of~$\uline T$ in~$\uline G$.
It is partially ordered by the Bruhat order induced by~$\uline {B}$ (extended to~$\tld{\uline W}$ from the affine Weyl group $\uline W_a := \uline W \ltimes \bZ\Phi$ in the usual way).
We write~$\tld {\uline W}^\vee$ for the same group with the partial order induced by the opposite Borel subgroup to~$\uline B$.
For~$\mu \in X^*(\uline T)$ we write $t_\mu \in \uline{\tld W}^\vee$ for the corresponding translation, 
and $\Adm^\vee(\mu)$ for the $\mu$-admissible set in~$\tld{\uline W}^\vee$, defined by
\[
\Adm^\vee(\mu) := \{\tld w \in \tld{\uline W}^\vee : \tld w \leq t_{w(\mu)} \text{ for some } w \in \uline W\}.
\]
We then have
\begin{gather*}
\Adm^\vee\sweight = \prod_{j \in \cJ}\{ t_{(2, 1)_j}, t_{(1, 2)_j}, t_{(1, 2)_j}s_j \}\\
\Adm^\vee\lweight = \prod_{j \in \cJ}\{t_{(2, 1)_j}, t_{(1, 2)_j}, t_{(1, 2)_j}s_j, t_{(3, 0)_j}, t_{(0, 3)_j}s_j, t_{(0, 3)_j}, t_{(2, 1)_j}s_j\},
\end{gather*}
where~$s_j$ is the element of~$\uline W$ which is nontrivial precisely in the embedding~$j$.

\subsubsection{Loop groups.}\label{loop groups}
We make use of notation from~\cite{LLHLMmodels} about loop groups, except that we will always work over the base $\Spec \cO$, with~$t := -p$.
If~$R$ is a Noetherian $\cO$-algebra, we thus write
\begin{gather*}
L^+\cM(R) := \{A \in M_2(R[\![v+p]\!]): A_{21} \in vR[\![v+p]\!]\}\\
L\cG(R) := \{g \in \GL_2(R(\!(v+p)\!)): g_{21} \in v R(\!(v+p)\!)\}\\
L^{[0, h]}\GL_2(R) := \{g \in \GL_2(R[\![v]\!]): v^h g^{-1} \in \GL_2(R[\![v]\!])\}\\
L^+\cG(R) := \{g \in \GL_2(R[\![v+p]\!]): g_{21} \in vR[\![v+p]\!]\}\\
L^{--}\cG(R) := \{ g \in \GL_2(R[(v+p)^{-1}]): g_{11}, g_{22} \in 1+(v+p)^{-1}R[(v+p)^{-1}],\\ 
g_{12} \in (v+p)^{-1}R[(v+p)^{-1}], g_{21} \in v(v+p)^{-1}R[(v+p)^{-1}] \}
\end{gather*}
We also write~$\cI$ for the standard upper-triangular Iwahori group scheme over~$\bF$, and $\cI_1$ for its pro-unipotent radical, so that $\cI = T \ltimes \cI_1$.
Note that if~$R$ is a Noetherian $\bF$-algebra then
\begin{gather*}
L^+\cG(R) = \cI(R)\\
L^{--}\cG(R) = \{g \in \GL_2(R[v^{-1}]): g_{11}, g_{22} \in 1 + v^{-1}R[v^{-1}], g_{12} \in v^{-1}R[v^{-1}]\}.
\end{gather*}

\subsubsection{Characters.}
The group $\uline G(\lbar \bF)$ has a geometric Frobenius endomorphism, 
which coincides with the map induced by the $p$-th power map on $\lbar \bF$. 
We denote it by~$F$.
There is an isomorphism $X_*(\uline T) \otimes_{\bZ} \lbar \bF^\times \to \uline T(\lbar \bF)$, equivariant for~$\uline W$ and~$F$, which induces for all~$w \in \uline W$ an isomorphism
\begin{equation}\label{characters of finite torus}
X^*(\uline T)/(wF-1)X^*(\uline T) \isom \Hom(T(\lbar \bF)^{wF}, \lbar \bF^\times).
\end{equation}
We fix throughout the paper an isomorphism $j_0: \lbar K \to \lbar E = \lbar \bQ_p$.
Its restriction $j_0: k \to \bF$ is an element of~$\cJ$.
This choice of~$j_0$ then allows us to replace the target of~\eqref{characters of finite torus} with $\Hom(H, \lbar \bF^\times)$, for some $H \in \{I, \cO_D^\times\}$ depending on~$w$, 
as we now recall.

Write $\pr : \uline W \cong S_2^{\cJ} \to S_2$ for the product map.
For all~$w \in \uline W$, the choice of $j_0 \in \cJ$ sets up an isomorphism
\begin{align*}
\Gm(\lbar \bF^{\Frob_q}) \times \Gm(\lbar \bF^{\Frob_q}) \to \;\; & \uline T(\lbar \bF)^{wF} \text { if~$\pr(w) = 1$}\\ 
\Gm(\lbar \bF^{\Frob_{q^2}}) \to \;\; & \uline T(\lbar \bF)^{wF} \text{ if~$\pr(w) \ne 1$.}
\end{align*}
Since $I/I_1 \cong k^\times \times k^\times$, and~$\cO_D^\times/(1+\fm_D) \cong k_D^\times$,
the choice of~$j_0$ also induces an isomorphism $\uline T(\lbar \bF)^{wF} \to I/I_1$ if~$\pr(w) = 1$, 
and a $\Frob_q$-conjugacy class of isomorphisms $\uline T(\lbar \bF)^{wF} \to \cO_D^\times/(1+\fm_D)$ if~$\pr(w) \ne~1$.
In the latter case, we make an arbitrary choice of representative of this conjugacy class (which has two elements).
Since we are ultimately interested in $N_G(H)$-conjugacy classes of characters $H \to \bF^\times$, our results will be unaffected by this choice.
In summary, for all~$w \in \uline W$ we now obtain from~\eqref{characters of finite torus} an isomorphism
\begin{equation}\label{definition of chi_w}
\chi_w : X^*(\uline T)/(wF-1)X^*(\uline T) \to \Hom(H, \bF^\times),
\end{equation}
where~$H = I$ if~$\pr(w) = 1$ and~$H = \cO_D^\times$ if~$\pr(w) \ne 1$.

For definiteness, throughout the paper we will work with the maps~$\chi_{w_H}$ 
associated to $w_H \in \uline W$
defined as follows: when~$H = I$, we let~$w_H = 1$, and when~$H = \cO_D^\times$, we let~$w_H \ne 1$ precisely
at the embedding~$j_0$.
We will often omit~$\chi_{w_H}$ from the notation, and for all~$\mu \in X^*(\uline T)$ we will write~$\mu$ for the $H$-character~$\chi_{w_H}(\mu)$. 
A consideration of $p$-adic expansions shows that
the maps~$\chi_{w_H}$ have the following elementary property.
\begin{lemma}\label{products of embeddings}
For all~$\alpha \in \Phi^+$, choose~$n_\alpha \in \{0, \pm 1, \pm 2\}$, and let $\mu := \sum_{\alpha \in \Phi^+}n_\alpha \alpha$.
Assume that $p>3$ and $\chi_{w_H}(\mu) = 0$.
Then $\mu = 0$.
\end{lemma}

\subsubsection{Inertial types.}
Recall that an inertial $\cO$-type is a representation $I_K \to \GL_2(\cO)$, with open kernel, and which can be extended to a representation of~$W_K$.
We associate to every pair $(s, \mu) \in \uline W \times X^*(\uline T)$ a ${\uline T}^\vee(\cO)$-valued character $\tau(s, \mu)$ by interpreting~$\mu$ as a cocharacter of~$\uline T^\vee$ 
and using the formula
\begin{equation}\label{definition of tau(s, mu)}
\tau(s, \mu) := \left (\sum_{i=0}^{d-1}(F^* \circ s^{-1})^i(\mu)\circ \omega_d \right ) : I_K \to \uline T^\vee(\cO).
\end{equation}
where~$F^*$ is the endomorphism of~$X_*(\uline T^\vee)$ corresponding to~$F$ under the identification~$X_*(\uline T^\vee) = X^*(\uline T)$,
and $d$ is the smallest positive integer such that $(F^* \circ s^{-1})^d = p^d$.
We obtain a tame inertial $\cO$-type, which we also denote~$\tau(s, \mu)$, by composing~$\tau(s, \mu)$ with $\uline T^\vee(\cO) \xrightarrow{j_0} \GL_2(\cO)$. 
Using the Teichm\"uller lift, and reduction modulo the maximal ideal of~$\mO$, we can identify the isomorphism classes of tame inertial $\cO$-types and~$\bF$-types.
We will sometimes denote the reduction of $\tau(s, \mu)$ by~$\lbar \tau(s, \mu)$.

Recall that a \emph{lowest alcove presentation} of a tame inertial $\cO$-type~$\tau$ is a pair $(w, \nu) \in \uline W \times X^*(\uline T)$ such that $\tau \cong \tau(w, \nu + \eta)$
and~$\nu \in \uline C_0$.
Given a lowest alcove presentation of~$\tau$, we define
\begin{equation}\label{group element associated to presentation of type}
\tld w^*(\tau) = w^{-1}t_{\nu+\eta} \in \tld{\uline W}.
\end{equation}
As usual, we say that~$\tau$ is \emph{$m$-deep} if it admits a lowest alcove presentation~$(w, \nu)$ such that
\[
m < \langle \nu + \eta, \alpha^\vee \rangle < p-m
\]
for all~$\alpha \in \Phi^+$.

\subsubsection{Inertial local Langlands correspondence.}
Let~$(G, H) \in \{(\GL_2(K), I), (D^\times, \cO_D^\times)\}$.
If~$\chi : H \to \bF^\times$ is a smooth character, and~$\chi$ is regular (i.e.\ $\chi$ is not inflated through the determinant or the reduced norm), then $(H, \chi)$ is a type
for a Bernstein component of the category $\mathrm{sm.}\,\lbar E[G]$ of smooth $\lbar E[G]$-modules.
The inertial local Langlands correspondence associates to this component an isomorphism class $\tau(\chi)$ of tame inertial $E$-types. 
We record the following compatibility between this correspondence and the morphism~\eqref{definition of chi_w}.

\begin{lemma}\label{explicit inertial type}
Let~$w_H \in \uline W$ be the identity if~$H = I$, and the nontrivial element in embedding~$j_0$ if~$H = \cO_D^\times$.
Let~$\chi: H \to \bF^\times$ be a regular smooth character, and let~$\mu$ be a preimage of~$\chi$ under the map
\[\chi_{w_H}: X^*(\uline T) \to \Hom(H, \bF^\times)\]
defined in~\eqref{definition of chi_w}.
Then $\tau(\chi) \cong \tau(w_H, \mu)$.
\end{lemma}
\begin{proof}
Recall from e.g.~\cite[Section~4.1]{herzig-duke} that Jantzen has constructed a map from $\uline W \times X^*(\uline T)$ to the set of virtual $E$-characters 
of $\uline G(\bF_p)$, denoted
\[
(s, \mu) \mapsto R_s(\mu).
\]
If $\tau = \tau(w, \nu)$ is a 0-deep 
inertial $E$-type, then the virtual character~$R_w(\nu)$ is the character of an absolutely irreducible $E[\GL_2(\bF_q)]$-module, 
and $(\GL_2(\cO_K), R_w(\nu))$ is a type for the Bernstein component corresponding to~$\tau$ under
the inertial local Langlands correspondence, see for example~\cite[Proposition~2.5.5]{LLHLMmodels}.
So it suffices to prove that  
\begin{enumerate}
\item If~$H = I$, then $(H, \chi)$ is a type for the principal series Bernstein component with type $(\GL_2(\cO_K), R_{w_H}(\mu))$.
\item If~$H = \cO_D^\times$, then $(H, \chi)$ is a type for the Jacquet--Langlands transfer of the supercuspidal Bernstein component with type $(\GL_2(\cO_K), R_{w_H}(\mu))$.
\end{enumerate}
The first statement follows because if~$H = I$ then
\[
R_{w_H}(\mu) = \Ind_{I}^{\GL_2(\cO_K)}(\chi).
\]
For the second statement, by the results of~\cite{BHlevelzero},
it suffices to prove that if~$z \in \GL_2(k)$ is a semisimple element with irreducible minimal polynomial, 
and~$z_D \in k_D$ has the same minimal polynomial over~$k$ as~$z$, then
\[
\operatorname{trace}(z|R_{w_H}(\mu)) = (-1)^{2/\deg(z)}\frac{|\GL_{2/\deg(z)}(k_{\deg(z)})|_{p'}}{q^2-1}\sum_{\gamma \in \Gal(k_{\deg(z)}/k)}\chi(\gamma z_D).
\] 
This is true by definition of the map $(w, \nu) \mapsto R_w(\nu)$ and the map~\eqref{definition of chi_w} sending~$\mu$ to~$\chi$.
\end{proof}

\begin{lemma}\label{inertial JL}
Let $\chi : H \to \bF^\times$ be a regular smooth character, and let $\alpha \in \Phi$.
Let $(w, \nu)$ be a lowest alcove presentation of~$\tau(\chi)$.
Then there exists~$\varepsilon \in \{\pm 1\}$ such that $(w, \nu +\varepsilon\alpha)$
is a lowest alcove presentation of $\tau(\chi\alpha)$.
\end{lemma}
\begin{proof}
Let~$\mu$ be a preimage of~$\chi$ under the map~$\chi_{w_H}$.
Then $\tau(\chi) = \tau(w_H, \mu)$ and~$\tau(\chi\alpha) = \tau(w_H, \mu + \alpha)$, by Lemma~\ref{explicit inertial type}.
Since $\tau(w_H, \mu + \alpha) = \tau(w_H, \mu)\tau(w_H, \alpha)$ as characters $I_K \to T^\vee(E)$, and similarly 
$\tau(w, \nu + \alpha) = \tau(w, \nu)\tau(w, \alpha)$,
it suffices to verify that
\[\tau(w_H, \alpha) = \tau(w, \alpha)^{\pm 1}.\]
This is immediate from~\eqref{definition of tau(s, mu)}.
(The same formula can be used to compute the sign, but we will not need this.)
\end{proof}

\subsubsection{Algebraic representations.}
We will use the following notation for algebraic representations of~$D^\times$.
Since~$E$ contains the quadratic unramified extension of~$K$, for every embedding $K \to E$ there exists 
a unique $\GL_2(E)$-conjugacy class of $K$-linear ring homomorphisms $D \to M_2(E)$.
This determines $\uline G(E)$-conjugacy classes of embeddings
\[
D \to \prod_{j \in \cJ} M_2(E)
\]
and
\[
D^\times \to \prod_{j\in \cJ} \GL_2(E) = \uline G(E).
\]
If~$\lambda \in X^*(\underline T)^+$ is a dominant character, we let~$V(\lambda)$ be the restriction to~$D^\times$, via any embedding in the conjugacy class just defined, 
of the corresponding algebraic $E$-representation of $\underline G(E)$.

\subsubsection{Galois deformation rings.}
Let $\rhobar: \Gal_K \to \GL_2(\bF)$ be a continuous representation.
We write~$R^\square_{\rhobar}$ for the universal lifting ring of~$\rhobar$ to complete Noetherian local $\cO$-algebras with residue field~$\bF$.
If~$\lambda \in X^*(\uline T)^+$ and~$\tau$ is a two-dimensional tame inertial $\cO$-type, we write~$R^{\lambda, \tau}_{\rhobar}$, resp.\ $R^{\leq \lambda, \tau}_{\rhobar}$, for the
reduced $\cO$-flat quotients of~$R^\square_{\rhobar}$ parametrizing potentially crystalline representations with Hodge--Tate weights~$\lambda$, resp.\ bounded by~$\lambda$,
and inertial type~$\tau$.
We write~$I^{\lambda, \tau}$ and~$I^{\leq\lambda, \tau}$ for their ideals in $R^\square_{\rhobar}$.

\subsubsection{Schur indices.}\label{subsec:Schur indices}
If~$\Gamma$ is a finite group, and $V$ is a finite-dimensional $\lbar E[\Gamma]$-module,
we write~$E(V)$ for the subfield of~$\lbar E$ generated by $\operatorname{tr}(\gamma: V \to V)$ for all~$\gamma \in \Gamma$.
Equivalently, $E(V)$ is the fixed field of $\{g \in \Gal(\lbar E/E): g^*V \cong V\}$.

\begin{lemma}\label{lem:Schur indices}
Let~$\Gamma$ be a finite group, and let $V$ be an irreducible $\lbar E[\Gamma]$-module.
Then $E(V)/E$ is unramified if and only if there exists an unramified extension~$L$ of~$E$ in~$\lbar E$ such that~$V$ descends to an $L[\Gamma]$-module.
\end{lemma}

\begin{proof}
By the theory of Schur indices, there exists a central division algebra~$D(V)$ over~$E(V)$ such that, for every subfield $E(V) \subset L \subset \lbar E$,
the representation~$V$ descends to~$L$ if and only if $D(V) \otimes_{E(V)} L$ is split.
Then the lemma is a consequence of the fact that every central division algebra over a local field splits over an unramified extension.
\end{proof}

%% file: Iwasawa-arXiv.tex
\section{Representations of compact subgroups of~$G$.}\label{D* representations}

\subsection{Truncated projective envelopes.}
Choose $(G, H) \in \{(\GL_2(K), I), (D^\times, \cO_D^\times)\}$, let~$Z$ be the centre of~$G$, 
and let $H_1 := \Syl_p(H)$ and~$Z_1 := Z \cap H_1$.

\begin{defn}\label{defn:Wchi,n}
Let~$\chi : H \to \bF^\times$ be a smooth character.
For every positive integer~$n > 0$, let~$W_{\chi, n}$ be the quotient
\[
W_{\chi, n} := \Proj_{\bF[\![H/Z_1]\!]}(\chi)/\fm^n,
\]
of the projective cover of $\chi$ in the category of $\bF[\![H/Z_1]\!]$-modules, where~$\fm$ is the maximal ideal of the Iwasawa algebra $\bF[\![H_1/Z_1]\!]$.
\end{defn}

Note that~$W_{\chi, n}$ is a smooth $H/Z_1$-representation, and that we have an isomorphism 
\[
W_{\chi, n} \cong \chi \otimes_\bF W_{1_H, n}
\]
where~$1_H$ denotes the trivial character.
The associated graded of the $H$-cosocle (equivalently, $\fm$-adic) filtration on~$W_{1_H, 3}$ is
\begin{gather*}
\gr^0_\fm W_{1_H, 3} = \bF\\
\gr^1_\fm W_{1_H, 3} = \bigoplus_{\alpha \in \Phi} \bF \alpha = \bigoplus_{i \in \cJ}\bF \alpha_i \oplus \bF \alpha_i^{-1}\\
\gr^2_\fm W_{1_H, 3} = \bF^{2f} \oplus \bigoplus_{\alpha \in \Phi}\bF\alpha^2 \oplus 
\bigoplus_{\substack{\{\alpha, \beta\} \subset \Phi\\\alpha \ne \pm \beta}}\bF \alpha\beta
\end{gather*}
(This is a consequence of the Poincar\'e--Birkhoff--Witt theorem for $\gr_\fm \bF[\![H/Z_1]\!]$, which is isomorphic to an enveloping algebra, compare \cite[Corollary~2.11]{HuWangD*} and~\cite[Formula~(44)]{BHHMS}.)
\begin{lemma}\label{multiplicity of characters in Wchi3}\leavevmode Recall our assumption that~$p>3$.
\begin{enumerate}
\item Let~$\chi^\perp$ be the complement in~$\gr^2_\fm W_{\chi, 3}$ of the $\chi$-isotypic component of $\gr^2_\fm W_{\chi, 3}$.
Then
\[
\gr^2_\fm W_{\chi, 3} = \chi^\perp \oplus \bF^{2f}\chi
\] 
and~$\chi^\perp$ is multiplicity-free.
\item For all~$\alpha \in \Phi$, the character $\chi\alpha : H \to \bF^\times$ has multiplicity one in~$W_{\chi, 3}^{\mathrm{ss}}$.
\end{enumerate}
\end{lemma}
\begin{proof}
After a twist, we can assume without loss of generality that~$\chi$ is the trivial character.
The description of $\gr_\fm^2 W_{\chi, 3}$ is then a direct consequence of the description of~$W_{1_H, 3}$ given in the paragraph above.
To prove that~$\chi^\perp$ is multiplicity-free, we need to prove that if $\alpha_1 \ne -\alpha_2 \in \Phi$ and~$\beta_1 \ne -\beta_2 \in \Phi$ then
\[
\alpha_1\alpha_2 = \beta_1\beta_2 \in \Hom(H, \bF^\times) \implies \{\alpha_1, \alpha_2\} = \{\beta_1, \beta_2\}. 
\]
If~$\alpha_1\alpha = \beta_1\beta_2$, then Lemma~\ref{products of embeddings} implies that $\alpha_1+\alpha_2 = \beta_1+\beta_2$. 
Since~$\Phi^+$ is $\bQ$-linearly independent in~$X^*(\uline T)_\bQ$, this implies that
$\{\alpha_1, \alpha_2\} = \{\beta_1, \beta_2\}$ and concludes the proof of part~(1).

Part~(2) is a direct consequence of Lemma~\ref{products of embeddings}.\qedhere
\end{proof}

\begin{defn}
We write~$\lbar W_{\chi, 3}$ for the quotient of~$W_{\chi, 3}$ by $\chi^\perp$.
\end{defn}
It follows from the discussion above that 
\begin{gather*}
\gr^0_\fm \lbar W_{\chi, 3} = \chi\\
\gr^1_\fm \lbar W_{\chi, 3} = \bigoplus_{\alpha \in \Phi} \bF \chi\alpha\\
\gr^2_\fm \lbar W_{\chi, 3}= \bF^{2f}\chi. 
\end{gather*}
Note that~$W_{\chi, 2}$ is a quotient of~$\lbar W_{\chi, 3}$.
The next lemma will be useful in Section~\ref{multiplicity}.

\begin{lemma}\label{W2inW3}
Let~$\alpha \in \Phi$.
Then every nonzero $H$-equivariant morphism $W_{\chi\alpha, 2} \to W_{\chi, 3}$ is injective.
Hence $\soc_H W_{\chi, 3} = \gr^2_\fm W_{\chi, 3}$.
\end{lemma}
\begin{proof}
This is a consequence of the fact that $\gr_{\fm}\bF[\![H_1/Z_1]\!]$ has no zero-divisors, see \cite[Lemma~6.1.2]{BHHMS} in the case $H = I$. 
Given~\cite[Proposition~2.5]{HuWangD*} and Lemma~\ref{multiplicity of characters in Wchi3}(2), the same argument goes through when~$H = \cO_D^\times$.
\end{proof}

\subsection{A presentation of~$\lbar W_{\chi,3}$.}
Our next results give a presentation for~$\lbar W_{\chi, 3}$, which we will use in Section~\ref{multiplicity} to prove a cyclicity theorem for its patched module.

\begin{lemma}\label{arm definition}
Let~$\chi : H \to \bF^\times$ be a smooth character, and let~$\alpha \in \Phi$.
There exist unique isomorphism classes of smooth $\bF[\![H/Z_1]\!]$-representations $E_{\chi, \alpha}$ and~$A_{\chi, \alpha}$ such that
\begin{gather*}
\gr^0_\fm(E_{\chi, \alpha}) = \chi, \;\;\; \gr^1_\fm(E_{\chi, \alpha}) = \chi\alpha\\
\gr^0_\fm(A_{\chi, \alpha}) = \chi, \;\;\; \gr^1_\fm(A_{\chi, \alpha}) = \chi\alpha, \;\;\; \gr^2_\fm(A_{\chi, \alpha}) = \chi, \;\;\;
\end{gather*}
i.e. $E_{\chi, \alpha} = (\chi\alpha - \chi)$ and~$A_{\chi, \alpha} = (\chi-\chi\alpha-\chi)$.
\end{lemma}
\begin{proof}
It suffices to prove that $\Ext^1_{\mathrm{sm.}\,\bF[H/Z_1]}(\chi, \chi \alpha)$ is one-dimensional, and that 
\[
\Ext^1_{\mathrm{sm.}\,\bF[H/Z_1]}(\chi, \chi) = \Ext^2_{\mathrm{sm.}\,\bF[H/Z_1]}(\chi, \chi) = 0.
\]
The statement about~$\Ext^1$ is a direct consequence of the description of $\gr^1_\fm W_{\chi, 3}$.
On the other hand, we have
\[
\Ext^i_{\mathrm{sm.}\,\bF[H/Z_1]}(\chi, \chi) \cong H^i(H/Z_1, \bF).
\] 
Since $H \cong Z_1 \times H/Z_1$, we have
\[
H^*(H, \bF) = H^*(Z_1, \bF) \otimes_{\bF} H^*(H/Z_1, \bF)
\]
hence the vanishing of $\Ext^1$ and~$\Ext^2$ follows if we prove that there exists a graded $\bF$-algebra~$\Lambda$ such that $\Lambda^0 = \bF, \Lambda^1 = \Lambda^2 = 0$, and
$H^*(H, \bF) = H^*(Z_1, \bF) \otimes_{\bF} \Lambda$.
This is a consequence of~\cite[Proposition~5.5]{DLHcohomology} when~$H = I$, and~\cite[Proposition~5.12]{DLHcohomology} when~$H = \cO_D^\times$.
\end{proof}

\begin{prop}\label{presenting Wchi3}
Let~$\chi : H \to \bF^\times$ be a smooth character, and write $\Delta(\chi) \subset \chi^{2f}$ for the diagonal subspace. 
Then there exists an exact sequence
\[
0 \to \lbar W_{\chi, 3} \to \bigoplus_{\alpha \in \Phi}A_{\chi, \alpha} \to \chi^{\oplus \Phi}/\Delta(\chi) \to 0
\]
such that the composite $\lbar W_{\chi, 3} \to A_{\chi, \alpha}$ is surjective for every $\alpha \in \Phi$.
\end{prop}
\begin{proof}
By Lemma~\ref{multiplicity of characters in Wchi3}(2) we know that for every~$\alpha \in \Phi$ we have
\[
\dim_\bF \Hom_{\bF[H/Z_1]}(W_{\chi \alpha, 2}, \lbar W_{\chi, 3}) = 1.
\]
The image of a nonzero element is a quotient of~$W_{\chi\alpha, 2}$ whose socle is a direct sum of copies of~$\chi$, by Lemma~\ref{W2inW3}. 
Hence the image is a nonsplit extension of~$\chi \alpha$ by~$\chi$, and so it is isomorphic to~$E_{\chi\alpha, -\alpha}$, by Lemma~\ref{arm definition}. 
It follows that there is a map
\[
\bigoplus_{\alpha \in \Phi} E_{\chi\alpha, -\alpha} \to \fm\lbar W_{\chi, 3}
\]
which is furthermore surjective (because it is surjective on $H$-cosocles).
Comparing lengths, we see that it is an isomorphism. 
So we have an exact sequence
\[
0 \to \bigoplus_{\alpha \in \Phi} E_{\chi\alpha, -\alpha} \to \lbar W_{\chi, 3} \to \chi \to 0.
\]
Since $\Hom(E_{\chi\alpha, -\alpha}, E_{\chi\beta, -\beta}) = 0$ if~$\alpha \ne \beta \in \Phi$, and it is one-dimensional if~$\alpha = \beta$,
we can choose surjections $\lbar W_{\chi, 3} \to A_{\chi, \alpha}$ such that the diagram
\[
\begin{tikzcd}
0 \arrow[r] & \bigoplus_{\alpha \in \Phi} E_{\chi\alpha, -\alpha} \arrow[d, "\id"] \arrow[r] & \lbar W_{\chi, 3} \arrow[d] \arrow[r] & \chi \arrow[d, "\Delta"] \arrow[r]& 0\\
0 \arrow[r] & \bigoplus_{\alpha \in \Phi} E_{\chi\alpha, -\alpha} \arrow[r] & \bigoplus_{\alpha \in \Phi} A_{\chi, \alpha} \arrow[r] & \bigoplus_{\alpha \in \Phi} \chi \arrow[r] & 0
\end{tikzcd}
\]
commutes.
The snake lemma then produces an isomorphism
\[
\coker\left (\lbar W_{\chi, 3} \to \bigoplus_{\alpha \in \Phi}A_{\chi, \alpha}\right ) \to \coker (\chi \to \chi^{\oplus \Phi})
\]
and shows that $\lbar W_{\chi, 3} \to \bigoplus_{\alpha \in \Phi}A_{\chi, \alpha}$ is injective,
as desired.
\end{proof}

\begin{prop}\label{arm lifting}
Let~$j \in \cJ$, and choose $\alpha \in \{\pm \alpha_j\} \subset \Phi$.
Then:
\begin{enumerate}
\item $[\chi] \oplus ([\chi\alpha] \otimes_E V(1, -1)_j)$ contains a unique $E^\times \times E^\times$-homothety class 
of $\mO[H/Z_1]$-lattices $\tld A_{\chi, \alpha}$ with irreducible $H$-cosocle isomorphic to~$\chi$. 
\item There is an exact sequence
\begin{equation}\label{arm lifting sequence}
0 \to \chi \alpha^2 \to \tld A_{\chi, \alpha} \otimes_\mO \bF \to A_{\chi, \alpha} \to 0.
\end{equation}
\item Let $\cL_\chi, \cL_{\chi\alpha}$ be the images of~$\tld A_{\chi, \alpha}$ in $[\chi]$, resp.\ $[\chi\alpha] \otimes_E V(1, -1)_j$.
Then there is an exact sequence
\begin{equation}\label{arm exact sequence}
0 \to \tld A_{\chi, \alpha} \to \cL_\chi \oplus \cL_{\chi\alpha} \to \chi \to 0.
\end{equation}
\end{enumerate}
\end{prop}
\begin{proof}
By construction, $V(1, -1)_j = (\Sym^2 E^2 \otimes_E \det^{-1})_{j}$, and so the semisimplified mod~$\pi_E$ reduction of~$V(1, -1)_j$ is 
$\alpha \oplus 1 \oplus \alpha^{-1}$.
Part~(1) then follows from the fact 
that the semisimplified mod~$p$ reductions of $[\chi]$ and $[\chi\alpha] \otimes_E V(1, -1)_j$ contain~$\chi$ with multiplicity one.
Indeed, it suffices to let~$\tld A_{\chi, \alpha}$ be the image of a morphism 
\[\Proj_{\cO[\![H/Z_1]\!]}(\chi) \to [\chi] \oplus ([\chi\alpha] \otimes_E V(1, -1)_j)\]
with nonzero projection on both summands, and to note that such a morphism is unique up to scaling by $E^\times \times E^\times$.
It is then true by construction that $\tld A_{\chi, \alpha} \otimes_\mO \bF$ has irreducible cosocle isomorphic to~$\chi$, 
that~$Z_1$ acts trivially on $\tld A_{\chi, \alpha}$, and that
\[
(\tld A_{\chi, \alpha} \otimes_\mO \bF)^{\mathrm{ss}} = \chi^{\oplus 2} \oplus \chi\alpha \oplus \chi \alpha^2.
\]
Since $\Ext^1_{\bF[\![H/Z_1]\!]}(\chi, \chi) \cong \Ext^1_{\bF[\![H/Z_1]\!]}(\chi, \chi\alpha^{\pm 2}) = 0$, it now follows that
\[
\gr^0_{\fm}(\tld A_{\chi, \alpha} \otimes_{\cO} \bF) = \chi,\; \gr^1_\fm(\tld A_{\chi, \alpha} \otimes_{\cO} \bF) =  \chi\alpha, \; \gr^2_\fm(\tld A_{\chi, \alpha} \otimes_{\cO} \bF) 
= \chi \oplus \chi\alpha^2.
\]
The uniqueness part of Lemma~\ref{arm definition} now implies that the cokernel of $\chi\alpha^2 \to \tld A_{\chi, \alpha} \otimes_\cO \bF$ is isomorphic to~$A_{\chi, \alpha}$, 
and part~(2) follows.

To prove part~(3), observe that the cokernel of $\tld A_{\chi, \alpha} \to \cL_\chi \oplus \cL_{\chi\alpha}$ is an $H$-stable quotient of~$\cL_{\chi\alpha}$
on which $H$ acts via~$[\chi]$.
So it suffices to prove that if $\cL \subset (\Sym^2 E^2 \otimes \det^{-1})_j$ is an $H$-stable $\cO$-lattice, and $\cL' \subset \cL$ is an $H$-stable sublattice
such that $H$ acts on $\cL/\cL'$ via $[\alpha_j]$ or~$[\alpha_j^{-1}]$, then $\cL/\cL'$ has $\cO$-length at most one.

Let~$\mu = \diag([k^\times], [k^\times])$ when~$H = I$, and~$\mu = [k_D^\times]$ when~$H = \cO_D^\times$.
Then~$\mu$ acting on $(\Sym^2 E^2 \otimes \det^{-1})_j$ 
has eigenvectors~$x^2, xy, y^2$ with eigencharacter~$[\alpha_j], 1, [\alpha_j^{-1}]$ respectively.
Assume first that the $H$-action on~$\cL/\cL'$ is via~$[\alpha_j]$.
Rescaling~$\cL$, we can assume that $\cL \cap Ex^2 = \cO x^2$. 
Since~$\mu$ acts on $\cL/\cL'$ via~$[\alpha_j]$, we see that $\cO x^2 \to \cL/\cL'$ is surjective. 
So it suffices to prove that $p\cO x^2 \subset \cL' \cap Ex^2$.
Note that
\[
\fourmatrix{1}{1}{p}{1} \in \operatorname{image}(H \to \Aut_E(V(1, -1)_j)),
\]
since it is the image of $1+\Pi_D$ when~$H = \cO_D^\times$.
This implies $(x+py)^2 = x^2+2pxy+p^2y^2 \in \cL$, and so its $\mu$-isotypic components $x^2, 2pxy, p^2y^2$ are also contained in $\cL$.
Since~$\mu$ acts trivially on~$pxy$, we see that $pxy = 0 \in \cL/\cL'$, and so~$pxy \in \cL'$, 
which in turn implies $p(x+py)(x+y) = p(x^2+(1+p)xy+py^2) \in \cL'$.
Thus $px^2 \in \cL'$, as desired.
The proof for~$[\alpha_j^{-1}]$ is similar, using that $y^2 \in \cL$ implies $(x+y)^2 \in \cL$, hence $xy \in \cL'$, hence $py^2 \in \cL'$.
\end{proof}

\subsection{Construction of essentially self-dual lattices.}
We end this section with some properties of types for the Bernstein components of unit groups of quaternion algebras.
They will be used in Section~\ref{global} to find coefficients for algebraic modular forms (at $\ell$-adic places for $\ell \ne p$)
which are simultaneously self-dual, minimal, and defined over the absolutely unramified $p$-adic field~$E$.
Accordingly, we let~$F_w/\bQ_\ell$ be a finite extension, we let~$B/F_w$ be a quaternion division algebra, and we write~$W_{F_w}$ for the Weil group of~$F_w$.
We write~$q_w$ for the cardinality of the residue field of~$F_w$, and~$\nu$ for the reduced norm of~$B^\times$. 
We normalize the local Langlands correspondence~$\rec_{B^\times}$ for $B^\times$ with $\lbar E$-coefficients 
so that it commutes with duals, twists and central characters; this requires us to choose
a square root of $|k_{F_w}|$ in~$\lbar E$.
Since by assumption $\bQ_{p^2} \subset E$, this square root is contained in~$E$, and so~$\rec_{B^\times}$ has the following rationality property, using notation from Section~\ref{subsec:Schur indices}:
$E_{W_{F_w}}(\rec_{B^\times} V) = E_{B^\times}(V)$ for all irreducible smooth $\lbar E[B^\times]$-modules~$V$.

\begin{lemma}\label{self-duality for D*}
Let $(\tau : W_{F_w} \to \GL_2(\lbar E), N)$ be a Langlands parameter for~$B^\times$, and let~$\xi := \det (\tau)$. 
Let~$V$ be an $\lbar E[\cO_B^\times]$-type for the Bernstein component of~$\lbar E[B^\times]$ corresponding to~$(\tau, N)$ under~$\rec_{B^\times}$.
Assume that~$E_{W_{F_w}}(\tau) = E$.
Then:
\begin{enumerate}
    \item The contragredient $V^\vee$ is $B^\times$-conjugate to $V \otimes (\xi \circ \nu)^{-1}$, where~$\nu : B^\times \to F_w^\times$ is the reduced norm.
    \item $V$ descends to an $L[\cO_B^\times]$-module~$V_0$ for some unramified extension $E \subset L \subset \lbar E$, and
    \item $V_0$ contains a unique homothety class of $\cO_L[\cO_B^\times]$-stable lattices.
\end{enumerate}
\end{lemma}

\begin{rmk}\label{self-duality for tame characters}
Similarly to Lemma~\ref{self-duality for D*}(1), note also that if $(G, H) \in \{(\GL_2, I), (D^\times, \cO_D^\times)\}$ and $\chi : H \to \cO^\times$ is a tame smooth character, 
and $\xi = \chi |_{\cO_F^\times}$,
then $\chi^{-1} \otimes (\xi \circ \nu)$ is conjugate to~$\chi$ under the normalizer of~$H$ in~$G$.
\end{rmk}

\begin{proof}
Assume first that~$N \ne 0$.
Then $\tau = \chi \oplus \chi \nr_{q_w}$ for some character $\chi: F_w^\times \to \lbar E^\times$. 
Since~$E_{W_{F_w}}(\tau) = E$, and the trace of~$\tau$ is $(\nr_{q_w}+1)\chi$, we see that~$\chi$ takes values in~$E$.
On the other hand, $(\tau, N)$ is the Langlands parameter of an unramified twist of $\chi \circ \nu$.
Hence $V = \chi \circ \nu |_{\cO_{B}^\times}$, and all assertions of the lemma are immediate.

Assume now that~$N = 0$, so that~$\tau$ is irreducible.
Let~$V'$ be an irreducible smooth $\lbar E[B^\times]$-module with~$\operatorname{rec}_{B^\times}(V') = \tau$.
Since $\tau^\vee \cong \tau \otimes \xi^{-1}$,
and~$\operatorname{rec}_{B^\times}$ commutes with duals and character twists, we see that $(V')^\vee \cong V' \otimes (\xi \circ \nu)^{-1}$.
Now the first statement follows because $V'|_{\cO_B^\times}$ is isomorphic to either~$V$ or the direct sum of~$V$ and a $B^\times$-conjugate of~$V$.

We now prove the second statement.
By Lemma~\ref{lem:Schur indices}
it suffices to prove that $E_{\cO_B^\times}(V)/E$ is unramified.
The rationality properties of the local Langlands correspondence imply that $E_{B^\times}(V') = E_{W_{F_w}}(\rec_{B^\times}V') = E$.
Since~$\Gal(\lbar E/E_{B^\times}(V'))$ is the stabilizer of the $\lbar E[B^\times]$-isomorphism class of~$V'$, it acts on $\{V, (\ad\,\Pi_B)^*(V)\}$,
and the stabilizer of~$V$ under this action is $\Gal(\lbar E/E_{\cO_B^\times}(V))$.
Hence~$E_{\cO_B^\times}(V)$ has degree at most two over~$E_{B^\times}(V') = E$, and so
$E_{\cO_B^\times}(V)$ has absolute ramification degree~$\leq 2$.

Next, let~$W \subset V$ be an irreducible submodule for $\lbar E[1+\fm_{B}]$.
Since 
\[
\cO_B^\times \cong k_{B}^\times \ltimes (1+\fm_{B}),
\]
if~$N$ is the normalizer in $\cO_B^\times$ of the isomorphism class of~$W$, then there exists an $\lbar E[N]$-module~$\tld W$ such that
$\tld W|_{1+\fm_{B}} \cong W$ and $V = \Ind_N^{\cO_B^\times} \tld W$.
Let~$\mu \subset k_{B}^\times$ be the subgroup such that $N = \mu \ltimes (1+\fm_{B})$.
Note that $\tld W$ is a twist by a character $\chi: N/(1+\fm_{B_w}) \cong \mu \to \lbar E^\times$ 
of the unique extension $W^{\mathrm{can}}$ of~$W$ to~$N$ 
such that $\det(W^\mathrm{can})$ is trivial on~$\mu$.
By~\cite[Theorem~2(b)]{Glauberman}, the trace of~$W^{\operatorname{can}}$ is $\bZ$-valued on $\mu$, and $E_{N}(W^{\operatorname{can}}) = E_{1+\fm_B}(W)$.
Since~$1+\fm_B$ is a pro-$\ell$ group, we conclude that $E_N(W^{\operatorname{can}})$ is unramified over~$E$. 
By Lemma~\ref{lem:Schur indices}, $W^{\operatorname{can}}$ descends to the maximal unramified extension~$E^{\nr}$ of~$E$ in~$\lbar E$,
and so the same is true for $\Ind_{N}^{\cO_B^\times}W^{\operatorname{can}}$. 
We conclude that $E_{\cO_B^\times}(\Ind_{N}^{\cO_B^\times}W^{\operatorname{can}})$
is unramified over~$E$.

If we now choose an extension $\tld \chi : \cO_B^\times \to \lbar E^\times$ of~$\chi$, the projection formula implies that 
\[
V = \tld \chi \otimes_{\lbar E} \Ind_N^{\cO_B^\times}(W^{\operatorname{can}}).
\]
Hence, for all~$g \in \cO_B^\times$, there exist $x_g \in E^{\nr}$ and~$\zeta_g \in \mu_{p^\infty}(\lbar E)$ such that
\[
\operatorname{tr}(g: V \to V) = \zeta_g x_g.
\]
From this we conclude that either $\operatorname{tr}(g: V \to V) \in E^{\nr}$ for all~$g$, or $E_{\cO_B^\times}(V) E^{\nr}$ contains a nontrivial $p$-th root of unity.
Since~$p-1 >2$, the latter option contradicts the fact that the absolute ramification index of $E_{\cO_B^\times}(V)$ is~$\leq 2$.
Hence $E_{\cO_B^\times}(V) \subset E^{\nr}$, which concludes the proof of the second statement.

For the third statement, it suffices to prove that the semisimplifed mod~$p$ reduction $r_p(V_0)$ of~$V_0$ is irreducible.
Using notation from the previous paragraph, we have $V_0 = \Ind_N^{\cO_B^\times} \tld W$, where
$\tld W|_{1+\fm_B} = W$ is irreducible.
Recall now that, since $1+\fm_B$ is a pro-$\ell$ group,
$r_p: \Irr_L(1+\fm_B) \to \Irr_{k_L}(1+\fm_B)$ is an $\cO_B^\times$-equivariant bijection.
This immediately implies that $r_p(W)$ is irreducible,
and~$N$ is the normalizer in $\cO_B^\times$ of the isomorphism class of $r_p(W)$.
Hence~$r_p(\tld W)$ is irreducible,
and~$N$ is the normalizer of~$r_p(\tld W)$ in~$\cO_B^\times$. 
We conclude that $r_p(V_0)$ is irreducible, since it is isomorphic to
$\Ind_N^{\cO_B^\times}r_p(\tld W)$.
\end{proof}

%% file: deformationrings-arXiv.tex
\section{Galois deformation rings.}\label{deformation rings}

\subsection{Moduli spaces of Kisin modules.}\label{moduli of Kisin modules}
We continue to use the notations and conventions from section \ref{subsect: preliminaries}.
Let~$\tau$ be a two-dimensional tame inertial $\cO$-type with a fixed lowest alcove presentation~$(w, \nu)$.
Let~$\lambda \in X^*(\uline T)^+$, and assume that $\sweight \leq \lambda \leq \lweight$.
In this section we describe the smooth atlas for the moduli stack of Kisin modules~$Y^{\leq \lambda, \tau}$ introduced in~\cite{LLHLMmodels}.
We emphasize that some of the objects defined below depend on~$(w, \nu)$, rather than just on~$\tau$, although we will often omit this dependence from the notation.
Recall the following objects from~\cite{EGstack, LLHLMmodels}:
\begin{enumerate}
    \item The Emerton--Gee stacks $\cX^{\lambda, \tau}$, resp.\ $\cX^{\leq \lambda, \tau}$, parametrizing potentially crystalline $\Gal_K$-representations on $p$-complete $\cO$-modules, 
    with Hodge--Tate weights~$\lambda$, resp.\ $\leq \lambda$,
    and inertial type~$\tau$.
    \item The moduli stack~$\PhiMod_K^{\et, 2}$ of \'etale $\varphi$-modules of rank two for $K_\infty := \bigcup_{n \geq 0} K(\sqrt[p^n]{-p})$.
    \item The moduli stack~$Y^{\leq \lambda, \tau}$ of Breuil--Kisin modules of rank two for~$K_\infty$, with descent datum~$\tau$, and
    elementary divisors bounded by~$\lambda$.
\end{enumerate} 
If~$\tld z \in \uline{\tld W}$ there is an open substack $Y^{\leq \lambda, \tau}(\tld z)$, parametrizing Kisin modules that admit a $\tld z$-gauge basis 
in the sense of~\cite[Section~5.2]{LLHLMmodels}.
As $\tld z$ varies, the substacks $Y^{\leq \lambda, \tau}(\tld z)$ assemble into an open cover of $Y^{\leq \lambda, \tau}$.
Furthermore, if~$\tau$ is 4-deep in~$\uline{C_0}$, 
then $Y^{\leq \lambda, \tau}(\tld z)$ is not empty if and only if $\tld z \in \Adm^\vee\lambda$, by~\cite[Corollary~5.3.4]{LLHLMmodels}.

Next, recall that there is a ${\uline T}^\vee$-torsor 
\begin{equation}\label{torus cover}
\tld U(\tld z, \leq \lambda) \to Y^{\leq \lambda, \tau}(\tld z) \subset Y^{\leq \lambda, \tau},
\end{equation}
defined as follows.
For all $\tld z_j = z_j t_{\nu_j} \in \tld W$ and~$\lambda_j \in X^*(T)$, let~$\tld U(\tld z_j, \leq \lambda_j)$ be the $p$-flat affine $p$-adic formal $\cO$-scheme 
whose points in an $p$-flat, $p$-complete $\cO$-algebra~$R$ 
are the matrices $A^{(j)} \in M_2(R[v+p])$ such that
\begin{enumerate}
\item $A_{21}^{(j)} \in vM_2(R[v+p])$.
\item $\deg A_{ik}^{(j)} \leq \nu_{j, k} - \delta_{i < z_j(k)}$, with equality when $i = z_j(k)$.
\item The leading coefficient of~$A_{z_j(k)k}^{(j)}$ is a unit in~$R$.
\item $\det A^{(j)} \in R^\times (v+p)^3$.
\item All entries of~$A^{(j)}$ are divisible by $(v+p)^{\lambda_{j, 2}}$.
\end{enumerate}
We then put 
\[
\tld U(\tld z, \leq \lambda) := \prod_{j \in \cJ} \tld U(\tld z_j, \leq \lambda_j).
\]
Note that $\tld U (\tld z, \leq \lambda)$ is independent of~$\tau$; by contrast, 
the torus action on $\tld U (\tld z, \leq\lambda)$ whose quotient is~\eqref{torus cover} is the $(w, \nu)$-twisted conjugation action~\cite[Theorem~5.3.1]{LLHLMmodels} 
and so depends on~$(w, \nu)$.
We introduce the notation $R^{\leq \lambda_j, \tld z_j} := \cO (\tld U(\tld z_j \leq \lambda_j))$ 
and $R^{\leq \lambda, \tld z} := \bigotimes_{\cO, j}^{\wedge} R^{\leq \lambda_j, \tld z_j}$;
we also write $A^{\tld z_j} \in M_2(R^{\leq \lambda_j, \tld z_j}[v+p])$ for the universal object.
Then, by definition, there is a presentation $R^{\leq \lambda_j, \tld z_j} = T^{\tld z_j}/I^{\leq \lambda_j, \tld z_j}$, where~$T^{\tld z_j}$ is the $p$-completion of a polynomial ring
as in Section~\ref{subsec: gauges}, 
the universal matrices $A^{\tld z_j}$ are as written in Section~\ref{subsec: gauges}, 
and~$I^{\leq \lambda_j, \tld z_j}$ is the $p$-saturation of the ideal of~$T^{\tld z_j}$ generated by
\begin{gather*}
\det A^{\tld z_j} \in (T^{\tld z_j})^\times(v+p)^3, \text{ and }\\
A^{\tld z_j} \in (v+p)M_2(T^{\tld z_j}[\![v+p]\!]) \text{ if  $\lambda_j = (2, 1)_j$}.
\end{gather*}

\subsection{Monodromy condition.}\label{monodromy condition}
We now introduce the $\tau$-monodromy condition on~$A^{\tld z_j}$, and use it to construct a smooth atlas of~$\cX^{\leq \lambda, \tau}$ from the atlas of~$Y^{\leq \lambda, \tau}$
defined in Section~\ref{moduli of Kisin modules}.
We also recall its truncated version, which we will compute explicitly in Section~\ref{truncated monodromy condition}. 

\begin{defn}\label{defn: monodromy ideals}
Let $(\fM^{\tau}, \beta^{\tau})$ be the Kisin module with $R^{\leq \lambda, \tld z}$-coefficients and $\tld z$-gauge basis classified by~\eqref{torus cover}.
We define the following ideals of~$R^{\leq \lambda, \tld z}$:
\begin{enumerate}
\item $I_{\fM^{\tau}, \nabla_\infty}$ is the ideal described in the statement of~\cite[Proposition~7.1.6]{LLHLMmodels}. By definition, it is $p$-saturated.
\item $I_{\fM^{\tau}, \beta^{\tau}, \nabla_1}$ 
is the ideal defined in~\cite[Definition~7.1.8]{LLHLMmodels}. It need not be $p$-saturated.
\item $I^{\sat}_{\fM^{\tau}, \beta^{\tau}, \nabla_1}$ is the $p$-saturation of~$I_{\fM^{\tau}, \beta^{\tau}, \nabla_1}$.
\end{enumerate}
\end{defn}

By construction,
there exist ideals $I_{\fM^{\tau}, \beta^{\tau}, \nabla_1}^{(j)} \subset R^{\leq \lambda_j, \tld z_j}$ such that
\begin{equation}\label{factorization of truncated monodromy}
I_{\fM^{\tau}, \beta^{\tau}, \nabla_1} = \sum_{j \in \cJ}I_{\fM^{\tau}, \beta^{\tau}, \nabla_1}^{(j)}R^{\leq \lambda, \tld z}.
\end{equation}
Note that~\eqref{factorization of truncated monodromy} implies that
\[
I^{\sat}_{\fM^{\tau}, \beta^{\tau}, \nabla_1} = \sum_{j \in \cJ}I^{(j), \sat}_{\fM^{\tau}, \beta^{\tau}, \nabla_1}R^{\leq \lambda, \tld z}.
\]
where $I^{(j), \sat}_{\fM^{\tau}, \beta^{\tau}, \nabla_1}$ is the $p$-saturation of $I^{(j)}_{\fM^{\tau}, \beta^{\tau}, \nabla_1}$ in~$R^{\leq \lambda_j, \tld z_j}$.

\begin{rmk}
In contrast with $I_{\fM^{\tau}, \beta^{\tau}, \nabla_1}$, the ideal $I_{\fM^{\tau}, \nabla_\infty}$ 
cannot in general be written as a sum over~$j \in \cJ$
of ideals of~$R^{\lambda_j, \tld z_j}$.
In a similar vein, $I_{\fM^{\tau}, \beta^{\tau}, \nabla_1}^{(j)}$ depends on~$(w, \nu)$, and not just on $(w_j, \nu_j)$.
\end{rmk}

By~\cite[Remark~7.1.9]{LLHLMmodels}, the ideal~$I_{\fM^{\tau}, \beta^{\tau}, \nabla_1}^{(j)}$ is generated by the condition
\[
\left ( v\frac d{dv} A^{\tld z_j} - A^{\tld z_j}\diag(\ba_{\tau, j}) \right )(v+p)^3(A^{\tld z_j})^{-1} \in (v+p)^2 L^+\cM(R^{\leq \lambda_j, \tld z_j}),
\]
for a quantity $\ba_{\tau, j} \in \cO^{\oplus 2}$ defined in~\emph{loc.\ cit}.
We will need the following property of~$\ba_{\tau, j}$, which we will often apply implicitly, 
in situations in which~$\nu$ is assumed to be at least $5$-deep in~$\uline C_0$.

\begin{lemma}\label{structure constants}
Let~$\kappa^{\tau}_j := \langle \ba_{\tau, j}, (1, -1)_j\rangle$, and let~$(w, \nu)$ be our fixed lowest alcove presentation of~$\tau$.
Then
\[
\kappa^{\tau}_j \equiv \langle w^{-1}(\nu+\eta), (1, -1)_j \rangle \mod p\cO.
\]
Hence, if~$\nu$ is $m$-deep in~$\uline C_0$, then
$\kappa_j^\tau$ is not congruent to $\pm n \text{ mod } p\cO$ for all~$0 \leq n \leq m$.
\end{lemma}
\begin{proof}
See~\cite[Lemma~7.3.1]{LLHLMmodels}.
\end{proof}

As explained in~\cite[Appendix~3]{FengLeHung}, the ring
\[
(R^{\leq \lambda_j, \tld z_j}/I^{(j)}_{\fM^\tau, \beta^\tau, \nabla_1}) \otimes_\cO E
\]
decomposes as a direct product indexed by the dominant weights~$\lambda_j' \leq \lambda_j$.
The factor indexed by~$\lambda_j'$ corresponds to the locus on which the elementary divisors of~$A^{\tld z_j}$ are given by~$\lambda'_j$.
We momentarily introduce the notation $I^{(j), \sat, \lambda'_j}_{\fM^\tau, \beta^\tau, \nabla_1}$ for the $p$-saturated ideal of 
$R^{\leq \lambda_j, \tld z_j}/I^{(j)}_{\fM^\tau, \beta^\tau, \nabla_1}$
corresponding to the factor indexed by~$\lambda_j'$.
Similarly, we write
\[
I^{\sat, \lambda'}_{\fM^\tau, \beta^\tau, \nabla_1} := \sum_{j \in \mJ}I^{(j), \sat, \lambda'_j}_{\fM^\tau, \beta^\tau, \nabla_1}R^{\leq \lambda, \tld z}.
\]
In computations, it will be useful to pull back these ideals through the map $T^{\tld z} \to R^{\leq \lambda, \tld z}$,
and so we introduce the following definition.

\begin{defn}\label{defn: pulled back monodromy ideals}
Recall the surjections $T^{\tld z_j} \to R^{\leq \lambda_j, \tld z_j}$ with kernel~$I^{\leq \lambda_j, \tld z_j}$, and their tensor product $T^{\tld z} \to R^{\leq \lambda, \tld z}$. 
We introduce the following ideals:
\begin{enumerate}
\item $I_{\nabla_\infty}^{\leq \lambda, \tau}$ is the preimage in $T^{\tld z}$ of $I_{\fM^\tau, \nabla_\infty}$.
\item $I_{\nabla_1}^{\leq \lambda, \tau}$, resp.\ $I_{\nabla_1}^{\lambda, \tau}$, 
is the preimage in $T^{\tld z}$ of $I_{\fM^\tau, \beta^\tau, \nabla_1}^{\sat}$, resp.\
$I_{\fM^\tau, \beta^\tau, \nabla_1}^{\sat, \lambda}$.
\item $I_{\nabla_1}^{\leq \lambda_j, \tau}$, resp.\ $I_{\nabla_1}^{\lambda_j, \tau}$, 
is the preimage in $T^{\tld z_j}$ of $I_{\fM^\tau, \beta^\tau, \nabla_1}^{(j),\sat}$, resp.\
$I_{\fM^\tau, \beta^\tau, \nabla_1}^{(j),\sat, \lambda_j}$. 
\end{enumerate}
\end{defn}
These ideals can be used to construct a smooth atlas for the stack~$\cX^{\lambda, \tau}$, in the following way.
Assume that~$\nu$ is $5$-deep in~$\uline C_0$.
By~\cite[Proposition~7.2.3, Proposition~5.4.2]{LLHLMmodels}, if $\tld z \in \Adm^\vee\lambda$ then there is a commutative diagram of $\Ind$-algebraic stacks over~$\Spf\,\cO$:
\begin{equation}\label{presentation of stacks}
\begin{tikzcd}
\tld U(\tld z, \leq \lambda) \arrow[r] & Y^{\leq \lambda, \tau}(\tld z) \arrow[r] & Y^{\leq \lambda, \tau} \arrow[r] & \PhiMod_{K, 2}^{\et}\\
\tld \cX^{\leq \lambda, \tau}(\tld z) \arrow[r] \arrow[u] & \cX^{\leq \lambda, \tau}(\tld z) \arrow[r] \arrow[u] & \cX^{\leq \lambda, \tau} \arrow[u] \arrow[ur] &
\end{tikzcd}
\end{equation}
with the following properties:
\begin{enumerate}
    \item The squares are Cartesian, and define the stacks $\cX^{\leq \lambda, \tau}(\tld z)$ and~$\tld \cX^{\leq \lambda, \tau}(\tld z)$. 
    \item The vertical arrows are closed immersions, and the ideal of 
    $\tld \cX^{\leq \lambda, \tau}(\tld z) \to \tld U(\tld z, \leq \lambda)$ is $I^{\leq \lambda, \tau}_{\nabla_\infty}$.
    \item The composite of the top row classifies the \'etale $\varphi$-module whose matrices of~$\varphi$ are the~$A^{\tld z_j}\tld w^*(\tau)_j$,
    where~$\tld w^*(\tau)$ is defined in~\eqref{group element associated to presentation of type}.
\end{enumerate}

If~$\lambda' \leq \lambda$ is a dominant weight, we can pull back $\tld \cX^{\leq \lambda, \tau}(\tld z)$ through the closed immersion $\cX^{\lambda', \tau} \to \cX^{\leq \lambda, \tau}$.
We write $\tld \cX^{\lambda', \tau}(\tld z)$ for this pullback, and
$I_{\nabla_\infty}^{\lambda', \tau} \subset T^{\tld z}$ for the pullback of the ideal of $R^{\leq \lambda, \tld z}$ 
corresponding to~$\tld\cX^{\lambda', \tau}(\tld z)$ under~\eqref{presentation of stacks}.
Note that~$I_{\nabla_\infty}^{\lambda', \tau}$ and $I_{\nabla_1}^{\lambda', \tau}$ are independent of the choice of~$\lambda$ such that $\lambda' \leq \lambda$.

\subsection{Truncated monodromy condition.}\label{truncated monodromy condition}
The truncated monodromy condition is an approximation of the monodromy condition, in the following sense:
by~\cite[Theorem~C.3.1, Remark~C.3.4]{FengLeHung}, 
if $\tau = \tau(w, \nu)$, $\nu$ is $m$-deep in~$\uline{C_0}$, and~$m > 4$,
then we have equalities of ideals 
\begin{equation}\label{approximated local models}
I^{\leq \lambda, \tau}_{\nabla_1} + p^{m-4} T^{\tld z} = I^{\leq \lambda, \tau}_{\nabla_\infty} + p^{m-4}T^{\tld z}
\end{equation}
and
\begin{equation}\label{fixed weight approximated local models}
I^{\lambda, \tau}_{\nabla_1} + p^{m-4} T^{\tld z} = I^{\lambda, \tau}_{\nabla_\infty} + p^{m-4}T^{\tld z}.
\end{equation}
We now compute $I_{\nabla_1}^{\leq \lambda, \tau}$ and $I_{\nabla_1}^{\lambda, \tau}$ explicitly.

\begin{thm}\label{weight at most (3, 0)}
Let~$\tau$ be a two-dimensional tame inertial $\cO$-type with lowest alcove presentation~$(w, \nu)$, let~$\sweight \leq \lambda \leq \lweight$, and let $\tld z \in \Adm^\vee \lambda$.
Assume that~$\nu$ is $5$-deep in~$\uline C_0$.
Then $I^{\leq \lambda_j, \tau}_{\nabla_1} \subset T^{\tld z_j}$ and $I^{\lambda_j, \tau}_{\nabla_1} \subset T^{\tld z_j}$ coincide with the ideals defined in Section~\ref{subsec: single-type approximation}.
\end{thm}
\begin{proof}
Throughout this proof, we will write $I_{\nabla_1, \mathrm{ex}}^{\leq \lambda_j, \tau}$ and $I_{\nabla_1, \mathrm{ex}}^{\lambda_j, \tau}$ for the ideals 
defined in Section~\ref{subsec: single-type approximation}, which by inspection are $p$-saturated.
We begin by proving the equality $I_{\nabla_1}^{\leq \lambda_j, \tau} = I_{\nabla_1, \mathrm{ex}}^{\leq \lambda_j, \tau}$. 
If $\tld z_j \in \Adm^\vee(2, 1)_j$, this follows from calculations
in~\cite{BHHMS}: we refer to~\cite[Tables~1--3]{BHHMS} and to the proof 
of~\cite[Proposition~4.2.1]{BHHMS}.
We give details for the other cases.

Assume that $\tld z_j \not \in \Adm^\vee(2, 1)_j$ and $\lambda_j = (2, 1)_j$. 
Then $I^{\leq \lambda_j, \tau}_{\nabla_1, \mathrm{ex}} = T^{\tld z_j}$, by definition.
On the other hand, the ideal $I_{\nabla_1}^{\leq \lambda_j, \tau}$ is the $p$-saturation in $T^{\tld z_j}$ of an ideal that contains~$p$: 
in fact, the matrix $A^{\tld z_j}$ has an entry which is either~$1$ or~$v$, and
so the condition $A^{\tld z_j} \in (v+p)M_2(T^{\tld z_j}[\![v+p]\!])$ 
generates an ideal of~$T^{\tld z_j}$ that contains~$p$.
Hence~$I_{\nabla_1}^{\leq \lambda_j, \tau} = I_{\nabla_1, \mathrm{ex}}^{\leq \lambda_j, \tau}$ is true in this case.

There remains to consider the case that~$\tld z_j \not\in \Adm^\vee(2, 1)_j$ and $\lambda_j = (3, 0)_j$.
Because of the condition on the determinant, 
the ideal $I_{\fM^{\tau}, \beta^{\tau}, \nabla_1} \subset R^{\leq (3, 0)_j, \tld z_j}$ is generated by
\[
\left ( v\frac{d}{dv}A^{\tld z_j}-A^{\tld z_j}\fourmatrix{\kappa^{\tau}_j}{0}{0}{0} \right )(A^{\tld z_j})^{\mathrm{cof}} \in (v+p)^2L^+\cM\left (R^{\leq (3, 0)_j, \tld z_j} \right ).
\]
where~$A^{\mathrm{cof}}$ is the transpose of the cofactor matrix of~$A$.

Assume first that $\tld z_j = t_{(3, 0)_j}$.
Then the ideal generated by the condition
\[
\det A^{\tld z_j} \in (T^{\tld z_j})^\times(v+p)^3
\]
is generated by~$a_0, a_1, a_2$.
This implies that
\begin{equation*}
\left ( v\frac{d}{dv} A^{\tld z_j} - A^{\tld z_j}\fourmatrix{\kappa^{\tau}_j}{0}{0}{0} \right ) (A^{\tld z_j})^{\mathrm{cof}}\equiv 
\diag(a_3, d_0)\fourmatrix{0}{0}{vX}{0} \text{ mod } (v+p)^2 L^+\cM\left (R^{\leq (3, 0)_j, \tld z_j} \right ).
\end{equation*}
where $X := (v+p)((2-\kappa_j^\tau)c_1-2pc_2)+((1-\kappa_j^\tau)c_0-pc_1)$.
Hence by definition $I_{\fM^\tau, \beta^\tau, \nabla_1}$ is generated by the conditions $X|_{v=-p}=\frac{d}{dv}|_{v=-p}X=0$
and so
$I_{\fM^\tau, \beta^\tau, \nabla_1}= ((\kappa_j^\tau-2)c_1+2pc_2, (\kappa_j^\tau-1)c_0+pc_1)$.
Hence $I_{\nabla_1}^{\leq (3, 0)_j, \tau}$ coincides with 
\[
I^{\leq (3, 0)_j, \tau}_{\nabla_1, \mathrm{ex}} := (a_0, a_1, a_2, (\kappa_j^\tau-2)c_1+2pc_2, (\kappa_j^\tau-1)c_0+pc_1),
\]
as desired.

Assume now that~$\tld z_j = t_{(0, 3)_j}s$.
The ideal generated by the condition
\[
\det A^{\tld z_j} \in (T^{\tld z_j})^\times(v+p)^3
\]
is generated by
\[(d_0a_2-(c_1-p), d_0a_1-(c_0-pc_1), d_0a_0+pc_0).\]
Define
\[
\fourmatrix X 0 Y 0 := v\frac d {dv} A^{\tld z_j} - A^{\tld z_j} \fourmatrix{\kappa_j^\tau} 0 0 0,
\]
noting that the second column of the right-hand side is indeed zero.
Then
\[
\left ( v\frac d {dv} A^{\tld z_j} - A^{\tld z_j} \fourmatrix{\kappa_j^\tau} 0 0 0 \right ) (A^{\tld z_j})^{\mathrm{cof}} = \fourmatrix {c_2d_0 X} {-b_0X} {c_2d_0 Y} {-b_0Y}.
\]
Since~$b_0, c_2$ are units, we see that~$I_{\fM^\tau, \beta^\tau, \nabla_1}$ is generated by
\[
X \in (v+p)^2R^{\leq (3, 0)_j, \tau}[\![v+p]\!], Y \in (v+p)^2R^{\leq (3, 0)_j, \tau}[\![v+p]\!], d_0Y \in v(v+p)^2R^{\leq (3, 0)_j, \tau}[\![v+p]\!].
\]
We compute that
\begin{align*}
X &= 2a_2 v(v+p)+a_1 v - \kappa_j^\tau(a_2(v+p)^2+a_1(v+p)+a_0)\\
&\equiv (v+p)(-2a_2p+(1-\kappa_j^\tau)a_1) + (-pa_1-\kappa_j^\tau a_0) \text{ mod } (v+p)^2R^{\leq (3, 0)_j, \tau}[\![v+p]\!]
\end{align*}
and 
\begin{align*}
Y &= v((1-\kappa_j^\tau)((v+p)^2+c_1(v+p)+c_0)+2v(v+p)+c_1v)\\
&\equiv v((v+p)((2-\kappa_j^\tau)c_1-2p)+((1-\kappa_j^\tau)c_0-pc_1)) \text{ mod } v(v+p)^2R^{\leq (3, 0)_j, \tau}[\![v+p]\!].
\end{align*}
Hence $I_{\fM^\tau, \beta^\tau, \nabla_1}$ is generated by %
\[
((\kappa_j^\tau-1)a_1 + 2pa_2, \kappa_j^\tau a_0+pa_1, (\kappa_j^\tau-2)c_1+2p, (\kappa_j^\tau-1)c_0+pc_1)
\]
and so $I_{\nabla_1}^{\leq (3, 0)_j, \tau}$ is the $p$-saturation in $T^{\tld z_j}$ of
\[
(d_0a_2-(c_1-p), d_0a_1-(c_0-pc_1), d_0a_0+pc_0, (\kappa_j^\tau-1)a_1 + 2pa_2, \kappa_j^\tau a_0+pa_1, (\kappa_j^\tau-2)c_1+2p, (\kappa_j^\tau-1)c_0+pc_1),
\]
which is readily seen to coincide with $I_{\nabla_1, \mathrm{ex}}^{\leq (3, 0)_j, \tau}$, as desired.

The remaining two cases $\tld z_j \in \{t_{(0, 3)_j}, st_{(1, 2)_j}\}$ follow by conjugation by the Iwahori-normalizing matrix $st_{(1, 0)_j}$ (or can be proved directly in the same way).

We now prove that $I_{\nabla_1, \mathrm{ex}}^{\lambda_j, \tau} = I_{\nabla_1}^{\lambda_j, \tau}$.
Since these are $p$-saturated ideals of $T^{\tld z_j}/I_{\nabla_1}^{\leq (3, 0)_j, \tau}$, it suffices to prove that they have the same extension to an ideal of 
$S := \bigl( T^{\tld z_j}/I_{\nabla_1}^{\leq (3, 0)_j, \tau} \bigr) \otimes_{\cO} E$.
This can be verified by inspection from the formulas in Section~A.2, and is immediate whenever $\tld z_j \not \in \Adm^\vee(2, 1)$.
We give details in the case $\tld z_j = t_{(2, 1)_j}$.
By definition, the extension of $I_{\nabla_1}^{\lambda_j, \tau}$ to $S$
is generated by the condition that the elementary divisors of any specialization of $A^{\tld z_j}$ at a closed point of~$S$ are given by~$\lambda_j$.
When $\lambda_j = (2, 1)_j$, because of the determinant condition, this means that all entries of $A^{\tld z_j}$ vanish at~$v = -p$. 
This shows that
$I_{\nabla_1, \mathrm{ex}}^{(2, 1)_j, \tau}S = I_{\nabla_1}^{(2, 1)_j, \tau}S$.
When $\lambda = (3, 0)_j$, the first four relations in $I_{\nabla_1}^{\leq (3, 0)_j, \tau}$
imply that $b_0$ is invertible in $S/I_{\nabla_1}^{(3, 0)_j, \tau}S$. 
Then the fifth relation 
\[
b_0((\kappa_j^\tau-1)(\kappa_j^\tau-2)b_0c_1-2p) = 0
\]
implies that
$(\kappa_j^\tau-1)(\kappa_j^\tau-2)b_0c_1-2p \in I_{\nabla_1}^{(3, 0)_j, \tau}S$, and so
$I_{\nabla_1, \mathrm{ex}}^{(3, 0)_j, \tau}S \subset I_{\nabla_1}^{(3, 0)_j, \tau}S$.
The reverse inclusion follows from the fact that every specialization of $A^{\tld z_j}$ to a closed point of $S/I_{\nabla_1, \mathrm{ex}}^{(3, 0)_j, \tau}S$ has elementary divisors~$(3, 0)$, 
which is because 
\[
(\kappa_j^\tau-1)(\kappa_j^\tau-2)b_0c_1-2p \in I_{\nabla_1, \mathrm{ex}}^{(3, 0)_j, \tau}
\]
implies that $b_0$ is a unit in $S/I_{\nabla_1, \mathrm{ex}}^{(3, 0)_j, \tau}S$
(because so is~$2p$).
\end{proof}

\subsection{Single-type Galois deformation rings.}\label{subsec: single type deformation rings}
Let $\rhobar: G_{K} \to \GL_2(\bF)$ be a continuous representation.
Choose a lowest alcove presentation~$(\mathfrak{w}, \mu)$ of the tame inertial $\bF$-type $\rhobar^{\mss}|_{I_{K}}$, 
so that $\rhobar^{\mss}|_{I_{K}} = \overline\tau(\mathfrak{w}, \mu+\eta)$.
We assume that~$\mathfrak{w} = 1$ when~$\rhobar$ is reducible, and we define
\[\tld w^*(\rhobar) := \mathfrak{w}^{-1}t_{\mu+\eta},\]
and we say that~$\rhobar$ is $m$-generic if~$\mu$ is $m$-deep in~$\uline C_0$.
We assume from now on that~$\rhobar$ is $9$-generic, although weaker genericity conditions are sometimes possible.
Looking ahead, the reason for $9$-genericity is that we will sometimes require the exponent in~\eqref{approximated local models} and~\eqref{fixed weight approximated local models} 
to be at least~$2$. 
We thus need that $R_{\rhobar}^{\lambda, \tau} \ne 0$ implies
that~$\tau$ is $6$-deep, which is guaranteed when~$\rhobar$ is $9$-generic.
This also turns out to be enough genericity for the rest of our results.

In this section we explain how to use the diagram~\eqref{presentation of stacks} to compute presentations of potentially crystalline deformation rings of~$\rhobar$ by passing to 
versal rings of $\cX^{\lambda, \tau}$.
We assume that $(2, 1) \leq \lambda \leq (3, 0)$, and we begin by recalling the classification of tame inertial $\cO$-types such that $R_{\rhobar}^{\lambda, \tau} \ne 0$,
or equivalently $\rhobar \in \cX^{\lambda, \tau}(\bF)$.
We do this in more detail than strictly necessary, since we will need some of this material in later sections of the paper 
(for example, in the ``covering" part of the proof of Proposition~\ref{Wchi3}).

Recall from e.g.~\cite[Proposition~3.1]{Dan}
that there exist unique $N \in \uline U^-(\bF)$, $\Delta \in \uline T^\vee(\bF)$ such that 
the \'etale $\varphi$-module of~$\rhobar|_{\Gal_{K_\infty}}$ over~$\cO_{\cE, \bF}$ has matrices
\begin{equation}\label{matrices of Phi on rhobar}
\Phi_{\rhobar}^{(j)} := \Delta_j N_j\tld w^*(\rhobar)_j.
\end{equation}
Furthermore, $N_j = 1$ for all~$j$ if and only if~$\rhobar$ is semisimple.
If~$\tau$ is a tame inertial $\cO$-type, it admits at most one lowest alcove presentation~$(w, \nu)$ such that
\[
\tld w^*(\rhobar, \tau) := \tld w^*(\rhobar)\tld w^*(\tau)^{-1} \in \Adm^\vee(3, 0),
\]
where~$\tld w^*(\tau)$ is defined in~\eqref{group element associated to presentation of type}.
In the rest of the paper, we will restrict to inertial types that admit such a presentation, and we will implicitly work with this presentation.
It will be useful to introduce the sets
\begin{gather*}
\Adm^\vee_{\rhobar}(2, 1)_j := \{\tld w_j \in \Adm^\vee(2, 1)_j: N_j \ne 1 \implies \tld w_j \ne t_{(1, 2)_j} \}\\
\Adm^\vee_{\rhobar}(3, 0)_j := \{\tld w_j \in \Adm^\vee(3, 0)_j: N_j \ne 1 \implies \tld w_j \ne t_{(0, 3)_j} \}.
\end{gather*}
and write $\Adm^\vee_{\rhobar}(\lambda) := \prod_{j \in \cJ} \Adm^\vee_{\rhobar}(\lambda_j)$ for any~$\lambda$ such that $(2, 1) \leq \lambda \leq (3, 0)$.

\begin{lemma}\label{types lifting rhobar}
Assume that~$\rhobar$ is $9$-generic.
Let~$\tau$ be a $5$-deep tame inertial $\cO$-type, and let~$(2, 1) \leq \lambda \leq (3, 0)$ be a dominant weight.
The following are equivalent:
\begin{enumerate}
    \item $\rhobar \in \cX^{\lambda, \tau}(\bF)$.
    \item $\rhobar \in \cX^{\lambda, \tau}(\tld w^*(\rhobar, \tau))$.
    \item for all~$j \in \cJ$, $\tld w^*(\rhobar, \tau)_j \in \Adm^\vee(\lambda_j)$ and $N_j \tld w^*(\rhobar, \tau)_j \in \cI(\bF) \Adm^\vee(\lambda_j) \cI(\bF)$.
    \item for all~$j \in \cJ$, $\tld w^*(\rhobar, \tau)_j \in \Adm^\vee_{\rhobar}(\lambda_j)$.
\end{enumerate}
\end{lemma}
\begin{proof}
We begin by showing the equivalence of the first three statements. Note that~(2) implies~(1) by definition.
We shorten notation to $\tld w := \tld w^*(\rhobar, \tau)$.

(1) implies~(3): We have $\tld w \in \Adm^\vee(\lambda)$, by~\cite[Corollary~5.5.8]{LLHLMmodels}, and the fact that 
$\rhobar \in \cX^{\lambda, \tau}(\bF)$ implies $\rhobar^{\mathrm{ss}} \in \cX^{\lambda, \tau}(\bF)$.
Furthermore, by assumption, there is a Kisin module~$\fM \in Y^{\lambda, \tau}(\bF)$ such that $\fM \otimes_{\fS_\bF} \cO_{\cE, \bF}$ is isomorphic to the $\varphi$-module of $\rhobar|_{\Gal_{K_\infty}}$.
Choosing an eigenbasis of~$\fM$, and writing~$A^{(j)}$ for the corresponding matrices of Frobenius,
there thus exist $X_j \in L\GL_2(\bF)$ such that
\[
X_j\Phi_{\rhobar}^{(j)} = A^{(j)}\tld w^*(\tau)\varphi(X_{j \circ \Frob_p}) \text{ for all } j \in \cJ.
\] 
Now $\Phi_{\rhobar}^{(j)} = (\Delta^{(j)}N_j\tld w_j)\tld w^*(\tau)_j$, by definition, and $\Delta^{(j)}N_j\tld w_j, A^{(j)} \in L^{[0, 3]}\GL_2(\bF)$.
Since~$\tau$ is $4$-deep, we conclude from~\cite[Lemma~5.4.4]{LLHLMmodels} that $X_j \in \cI(\bF)$ for all~$j$.
In turn, \cite[Lemma~5.2.2]{LLHLMmodels} then implies that 
\[
N_j \tld w_j \in T(\bF)(\cI(\bF)A^{(j)})T(\bF).
\]
Since $A^{(j)} \subset \cI(\bF) \Adm^\vee(\lambda_j) \cI(\bF)$ by~\cite[Corollary~5.3.5]{LLHLMmodels}, this implies~(3).

(3) implies~(2): Assuming~(3), one checks directly from the formulas in Section~\ref{subsec: gauges} that $\Delta^{(j)}N_j\tld w_j$ defines an $\F$-point of
$T^{\tld w_j}/(p, I_{\nabla_1}^{\lambda_j, \tau})$.
Now \eqref{fixed weight approximated local models} implies that 
\[T^{\tld w_j}/(p, I_{\nabla_1}^{\lambda_j, \tau}) = T^{\tld w_j}/(p, I_{\nabla_\infty}^{\lambda_j, \tau}),\]
and so we obtain a Galois representation
$\rhobar_1 \in \cX^{\lambda, \tau}(\tld w)(\bF)$ such that $\rhobar|_{\Gal_{K_\infty}} \cong \rhobar_1|_{\Gal_{K_\infty}}$.
By~\cite[Lemma~7.2.10~(2), (3)]{LLHLMmodels}, this implies that $\rhobar \cong \rhobar_1$ as $\Gal_K$-representations, which implies~(2).

Finally, one sees by inspection that~(3) and~(4) are equivalent: the content of this statement is that, for all $\tld w_j \in \Adm^\vee(\lambda_j)$,
we have $N_j \tld w_j \not \in \cI(\bF) \Adm^\vee(\lambda_j)\cI(\bF)$ if and only if $N_j \ne 1$ and 
\[
(\tld w_j, \lambda_j) \in \{(t_{(1, 2)_j}, (2, 1)_j), (t_{(0, 3)_j}, (3, 0)_j)\}.\qedhere
\]
\end{proof}

The following property of $\Adm^\vee_{\rhobar}\sweight$ is a direct consequence of its definition.

\begin{lemma}\label{hypercube}
Let~$\tld u \in \Adm^\vee_{\rhobar}\sweight$.
\begin{enumerate}
\item If $\alpha \in \Phi$, then $\tld u t_{2 \alpha} \not \in \Adm^\vee_{\rhobar}(2, 1)$.
\item If $\alpha_1, \alpha_2 \in \Phi$ are linearly independent roots, then $\tld u t_{\alpha_1 + \alpha_2} \in \Adm^\vee_{\rhobar}(2, 1)$ implies $\tld u t_{\alpha_1}, \tld u t_{\alpha_2} \in \Adm^\vee_{\rhobar}\sweight$.
\end{enumerate}
\end{lemma}
\begin{proof}
The first statement is true by inspection. 
The second statement is a consequence of the componentwise nature of
the definition of~$\Adm^\vee_{\rhobar}\sweight$.
\end{proof}

If $\tld w \in \Adm_{\rhobar}^\vee\lweight$, we define $\tau_{\tld w}$ to be the unique inertial $\cO$-type which admits a lowest alcove presentation 
such that $\tld w^*(\rhobar, \tau_{\tld w}) = \tld w$.
Since $\rhobar$ is $9$-generic, the type~$\tau_{\tld w}$ is $6$-deep, and so
Lemma~\ref{types lifting rhobar} imply that there is an $\F$-point $A_{\rhobar, \tld w} = (A_{\rhobar, \tld w}^{(j)})_{j \in \cJ}$ of $T^{\tld w}/I_{\nabla_\infty}^{\lambda, \tau_{\tld w}}$ 
whose image in $\PhiMod^{\et}_{K, 2}(\bF)$
is $\rhobar|_{\Gal_{K_\infty}}$.
We thus see from~\eqref{presentation of stacks} that the
completion of $\tld \cX^{\lambda, \tau_{\tld w}}$ at~$\rhobar$ is isomorphic to
\[
\Spf \left ( T^{\tld w}/I_{\nabla_\infty}^{\lambda, \tau_{\tld w}} \right )^\wedge_{A_{\rhobar, \tld w}}.
\]
Since~$R^{\lambda, \tau_{\tld w}}_{\rhobar}$ is formally smooth of dimension~4 over a minimal versal ring to~$\cX^{\lambda, \tau_{\tld w}}$ at~$\rhobar$
(compare the proof of~\cite[Theorem~4.8.14]{EGstack})
we thus obtain an isomorphism
\begin{equation}\label{local models fixed weight}
R^{\lambda, \tau_{\tld w}}_{\rhobar}[\![X_1, \ldots, X_{2f}]\!] \cong \left (T^{\tld w}/I_{\nabla_\infty}^{\lambda, \tau_{\tld w}} \right )^\wedge_{A_{\rhobar, \tld w}}[\![Y_1, Y_2, Y_3, Y_4]\!].
\end{equation}

\subsection{Class groups of potentially Barsotti--Tate deformation rings.}
In this section we prove that the Weil class groups of the rings~$R^{\eta, \tau}_{\rhobar}$ are $2$-torsion free. 
We will only need to assume that~$\rhobar$ is $1$-generic.
After a twist, it suffices to prove that the class group of $R^{\sweight, \tau}_{\rhobar}$ is $2$-torsion free, which is what we will do.

\begin{prop}\label{Gorenstein}
Assume that $\rhobar$ is $1$-generic and $\tld w \in \Adm^\vee_{\rhobar}\sweight$.
Then~$R^{\sweight, \tau_{\tld w}}_{\rhobar}$ is equisingular to
\[
\widehat\bigotimes_{\substack{j \in \cJ:\;N_j = \id,\\\tld w_j = t_{(1, 2)_j}s}} \cO[\![X_j, Y_j]\!]/(X_jY_j-p).
\]
Thus $R_{\rhobar}^{\sweight, \tau_{\tld w}}$ is a complete intersection, hence a Gorenstein local ring.
\end{prop}
\begin{proof}
This is well-known, and goes back to~\cite[Theorem~7.2.1]{EGS}.
From the point of view of the previous section,
it is a consequence of the fact that the equality $I_{\nabla_1}^{\sweight, \tau_{\tld w}}+p^{m-4} T_{\rhobar}^{\tld w} = I^{\sweight, \tau_{\tld w}}_{\nabla_\infty} + p^{m-4} T^{\tld w}_{\rhobar}$ 
in~\eqref{fixed weight approximated local models} can be improved to $I_{\nabla_1}^{\sweight, \tau_{\tld w}} = I^{\sweight, \tau_{\tld w}}_{\nabla_\infty}$, 
since there is no monodromy condition in weight $\leq \sweight$.
\end{proof}

\begin{prop}\label{2torsion}
Let~$\tau$ be a tame inertial $\cO$-type.
Assume that $\rhobar$ is $1$-generic, and that $R:= R_{\rhobar}^{\eta, \tau}$ is not zero.
Then the Weil class group~$\mCl(R)$ is $2$-torsion free.
\end{prop}
\begin{proof} 
By Proposition~\ref{Gorenstein} there is a presentation
\[
R=\cO[\![x_j,y_j,z_k: 1\leq j\leq a, 1\leq k \leq b]\!]/(x_jy_j-p)
\] 
for some integers~$a, b \geq 0$.
Let $U$ be the smooth locus of the map $\Spec R \rightarrow \Spec \cO$.
Note that $U$ is exactly the complement of the locus in which $x_j=y_j=0$ for some $j$, and hence the complement $\Spec R \setminus U$ has codimension at least 2. 
Thus $\mCl(R) \cong \mCl(U) = \Pic(U)$, and so it suffices to prove that $\mathrm{Pic}(U)[2]=0$.
We have the Kummer short exact sequence
\[0\to \cO(U)^{\times}/(\cO(U)^{\times})^2\to H^1_{\et}(U,\mu_2 )\to \Pic(U)[2]\to 0.\]
Since $R$ is normal, we have $\cO(U) = R$, and since $p$ is odd, the natural map $R^\times \otimes_\bZ \bZ/2 \to \bF^\times \otimes_\bZ \bZ/2 \cong \bZ/2$
is an isomorphism.
Thus we need to show that $H^1_{\et}(U,\mu_2 ) = \bZ/2$.

Let $j:U[\frac{1}{p}]\into U$, $i:U/p\into U$ are the natural (open, respectively closed) immersions. 
The generic fibre $U[\frac 1 p]$ is isomorphic to $\Spec R[\frac 1 p]$. 
The irreducible components of the special fiber $U/p$ are geometrically irreducible, and they are in bijection with subsets $S\subset \{1,2, \ldots, f\}$, 
where the component corresponding to $S$ is given by the locus $x_j=y_i=0$ where $j\in S, i\notin S$.
The exact triangle
\[   i_*i^!\mu_2\to \mu_2\to Rj_*\mu_2\to\]
and absolute cohomological purity
$i^!\mu_2=\bF_2[-2]$ (see~\cite[Th\'eor\`eme~3.1.1]{Rioupurity}) give the exact sequence
\[
\xymatrix{0\ar[r] &H^1_{\et}(U,\mu_2 )\ar[r] &H^1_{\et}(U[\frac{1}{p}],\mu_2 )\ar[r]^{\res}& \sum_{\mathrm{Irr}(U_{\lbar \bF})} \F_2\ar[r] &\cdots}
\]
This sequence also admits the following geometric interpretation: Let $\xi \in H^1_{\et}(U[\frac{1}{p}],\mu_2 )$ 
correspond to a degree $2$ \'{e}tale cover $V\to U[\frac{1}{p}]$. 
If~$\fp \subset R$ is a minimal prime above~$pR$,
the normalization $\tld{U}$ of $U$ in $V$ gives a degree $2$ extension~$R'_\fp$ of the localization $R_{\mathfrak{p}}$. 
Then the projection of~$\res(\xi)$ at the irreducible component of~$U_{\lbar \bF}$ 
corresponding to $\mathfrak{p}$ is exactly the degree of the residue field extension of~$R'_\fp/R_\fp$. 
In particular, we see that $\res(\xi)$ is trivial at the component corresponding to $\mathfrak{p}$ 
exactly when $\tld{U}\to U$ is unramified at $\mathfrak{p}$.

Since~$R$ is $p$-complete, \cite[Corollary~6.6.3, Proposition~4.3.3]{Fujiwara} implies that
\[
H^1_{\et}(\Spec R[\frac 1 p], \mu_2) \cong H^1_{\et}((\Spf R)^{\rig}, \mu_2).
\]
Now, since the rigid analytic fiber~$(\Spf R)^{\rig}$ is a product of open balls and open annuli, we see that 
\[H^1_{\et}(U[\frac{1}{p}],\mu_2 ) = H^1_{\et}(\Spec R[\frac 1 p], \mu_2) \cong \left (\bigoplus_{1\leq j\leq a} \F_2\xi_j \right ) \oplus \F_2\kappa \oplus \F_2\kappa^{\nr}\] 
where $\xi_j$, resp.\ $\kappa$, resp.\ $\kappa^{\nr}$ is represented by the double \'{e}tale cover of $U[\frac{1}{p}]$ extracting $\sqrt{X_j}$, 
resp.\ $\sqrt{p}$, resp.\ $\bF^\times/(\bF^\times)^2$.

Recalling our parametrization of $\mathrm{Irr}(U/p)$ by subsets $S\subset \{1,\ldots, f\}$, we compute
\begin{enumerate}
\item $\res(\xi_j)$ is non-trivial exactly at the components corresponding to $S$ containing $j$: If $\mathfrak{p}$ is the prime corresponding to the component $S$, a local uniformizer of $R_{\mathfrak{p}}$ is given by $X_j$ if $j\in S$, and is given by $Y_j=pX_j^{-1}$ otherwise. This shows $\tld{U}$ is ramified at $\mathfrak{p}$ exactly when $j\in S$.
\item $\res(\kappa)$ is non-trivial at each component of $U/p$: this is because $p$ is a uniformizer of $R_{\mathfrak{p}}$ for any minimal prime $\mathfrak{p}$ containing $p$. 
\item $\res(\kappa^{\nr}) = 0$, because~$\kappa^{\nr}$ is everywhere unramified.
\end{enumerate}
We now claim that $\ker(\res) = \bF_2 \kappa^{\nr}$, which finishes the proof.
Indeed suppose 
\[
\res(\alpha\kappa+\sum_{j=1}^a a_j\xi_j)=0.
\] 
Let $T:=\{j: a_j\neq 0\}$, then the component of $\res(\alpha\kappa+\sum a_j\xi_j)$ at $S$ is given by $\alpha+|S\cap T|$. 
Thus we learn that $|S\cap T|=\alpha$ mod $2$ for any $S\subset \{1,\ldots, f\}$, which implies that~$T$ is empty and~$\alpha = 0$, as desired.
\end{proof}

\subsection{Multi-type Galois deformation rings.}
We continue to assume that~$\rhobar : \Gal_K \to \GL_2(\bF)$ is a $9$-generic representation.
Fix $\tld u \in \Adm_{\rhobar}^\vee\sweight$, and choose
\begin{equation}\label{definition of T}
T \subset \{\lambda \in X^*(\uline T): \sweight \leq \lambda \leq \lweight\} \times \{\tld u, \tld u t_{\pm \alpha_j}: j \in \cJ\}.
\end{equation}
Define
\[
I^T := \bigcap_{(\lambda, \tld w) \in T} I^{\lambda, \tau_{\tld w}} \subset R^\square_{\rhobar}.
\]
In this section we will compute a model for the multi-type deformation ring $R^{T}_{\rhobar} := R^\square_{\rhobar}/I^{T}$.
Recall from Section~\ref{subsec: gauges} the ring $T^{\tld w} = \bigotimes^{\wedge}_{\cO, j} T^{\tld w_j}$, with universal matrices
$A^{\tld w_j} \in M_2(T^{\tld w_j}[v])$. 
In Section~\ref{subsec: charts for PhiMod} we have introduced a $p$-complete $\cO$-algebra~$S^{\tld u} = \bigotimes^\wedge_{\cO, j} S^{\tld u_j}$, and for all
$\tld w_j \in \{\tld u_j, \tld u_j t_{\pm \alpha_j} \}$, 
a surjection 
\[
\pr_{\tld w_j}: S^{\tld u_j} \to T^{\tld w_j}.
\]
We write $\pr_{\tld w} := \bigotimes^\wedge_{\cO, j} \pr_{\tld w_j}$.
We have also introduced matrices $\Psi^{\tld u_j} \in M_2(S^{\tld u_j}[v^{\pm 1}])$ 
such that
\begin{equation}\label{specialization of Psi}
\pr_{\tld w_j}(\Psi^{\tld u_j}) = A^{\tld w_j}\tld w_j^{-1}.
\end{equation}

\begin{defn}\label{K-ideals}
    Let~$(\lambda, \tld w) \in T$. Then:
    \begin{enumerate}
    \item $K^{\det} \subset S^{\tld u}$ is the ideal generated by the conditions
\begin{equation}\label{determinant condition on S}
\det(\Psi^{\tld u_j}) \in (S^{\tld u_j})^\times v^{-3}(v+p)^3.
\end{equation}
    \item $K_{\nabla_\infty}^{\lambda, \tau_{\tld w}} := \pr_{\tld w}^{-1}(I_{\nabla_\infty}^{\lambda, \tau_{\tld w}}) \subset S^{\tld u}$.
    \item $K^{\lambda_j, \tau_{\tld w}}_{\nabla_1} := \pr_{\tld w_j}^{-1}(I_{\nabla_1}^{\lambda_j, \tau_{\tld w}}) \subset S^{\tld u_j}$.
    \item $K^{\lambda, \tau_{\tld w}}_{\nabla_1} := \pr_{\tld w}^{-1}(I_{\nabla_1}^{\lambda, \tau_{\tld w}}) = \sum_{j \in \cJ} K^{\lambda_j, \tau_{\tld w}}_{\nabla_1}S^{\tld u} \subset S^{\tld u}$.
    \end{enumerate}
\end{defn}

Note that $K^{\det} \subset K_{\nabla_\infty}^{\lambda, \tau_{\tld w}}$ for all~$(\lambda, \tld w) \in T$.
In the next result we show that the ideals~$K_{\nabla_1}^{\lambda, \tau_{\tld w}}$ are an algebraic approximation to the ideals~$K_{\nabla_\infty}^{\lambda, \tau_{\tld w}}$, 
in the same sense as Section~\ref{truncated monodromy condition}.
These ideals can be computed from the definition of~$\pr_{\tld w_j}$ and~$I^{\lambda_j, \tau_{\tld w}}_{\nabla_1}$:
it suffices to compute preimages of the generators of~$I^{\lambda_j, \tau_{\tld w}}_{\nabla_1, j}$ under~$\pr_{\tld w_j}$, and to add generators of~$\ker \pr_{\tld w_j}$.
A choice of generators is listed in Section~\ref{subsec: multi-type approximation} in the cases we will need in the sequel.

\begin{lemma}\label{multitype local models}
Let~$(\lambda, \tld w) \in T$. Then
    \begin{equation}\label{congruence for K}
    K^{\lambda, \tau_{\tld w}}_{\nabla_1} + p^{2}S^{\tld u} = K_{\nabla_\infty}^{\lambda, \tau_{\tld w}} + p^{2}S^{\tld u}.
\end{equation}
\end{lemma}
\begin{proof}
By definition, we have $K_{\nabla_1}^{\lambda, \tau_{\tld w}} = \pr_{\tld w}^{-1}(I_{\nabla_1}^{\lambda, \tau_{\tld w}})$ and 
$K^{\lambda, \tau_{\tld w}}_{\nabla_\infty} = \pr_{\tld w}^{-1}(I^{\lambda, \tau_{\tld w}}_{\nabla_\infty})$.
Then the lemma follows from~\eqref{fixed weight approximated local models}, the fact that~$\tau_{\tld w}$ is $6$-deep (since~$\rhobar$ is $9$-generic),
and the general fact that, if $\varphi: A \to B$ is a surjective ring homomorphism, and~$I \subset B$ is an ideal, then
\[\varphi^{-1}(I +p^n B) = \varphi^{-1}(I) + p^n A\]
for all~$n \geq 0$.
\end{proof}

Because of the determinant condition on~$K^{\det}$, 
the matrices $\Phi^{\tld u_j} := \Psi^{\tld u_j}\tld w^*(\rhobar)$
are the matrices of Frobenius for an \'etale $\varphi$-module~$\cM$ over~$\cO_{\cE, S^{\tld u}/K^{\det}}$, i.e. they define an $S^{\tld u}/K^{\det}$-valued point of $\PhiMod_{K, 2}^{\et}$. 
Because of~\eqref{specialization of Psi}, $\cM$ has the property that, for all $(\lambda, \tld w) \in T$,
\begin{gather}\label{M interpolates}
\cM \otimes_{S^{\tld u}/K^{\det}} S^{\tld u}/K_{\nabla_\infty}^{\lambda, \tau_{\tld w}} \cong 
\fM^{\tau_{\tld w}} \otimes_{\fS_{T^{\tld w}}} \cO_{\cE, T^{\tld w}/I_{\nabla_\infty}^{\lambda, \tau_{\tld w}}}.
\end{gather}
Indeed, $\fM^{\tau_{\tld w}} \otimes_{\fS_{T^{\tld w}}} \cO_{\cE, T^{\tld w}/I_{\nabla_\infty}^{\lambda, \tau_{\tld w}}}$ has matrices of Frobenius given by
$A^{\tld w_j} \tld w^*(\tau_{\tld w}) = (A^{\tld w_j}\tld w_j^{-1})\tld w^*(\rhobar)$, which is the specialization of $\Phi^{\tld u_j}$ under~$\pr_{\tld w_j}$.
The \'etale $\varphi$-module~$\cM$ gives rise to a classifying map
\[\operatorname{ch}_\cM: \Spf(S^{\tld u}/K^{\det}) \to \Phi\text{-Mod}_{K, 2}^{\text{\'et}}.\]
There is a shifted conjugation action of~$\uline T^\vee$ on $\Spf(S^{\tld u}/K^{\det})$, defined by
\[
(t.\Phi^{\tld u})_j = t_j \Phi^{\tld u_j} t^{-1}_{j \circ \Frob_p},
\]
and~$\operatorname{ch}_\cM$ is constant on the orbits of this action.

\begin{prop}\label{monomorphism}
The map~$\operatorname{ch}_\cM$ induces a monomorphism
\[
[\Spec (S^{\tld u}/K^{\det} \otimes_{\cO} \bF)/\uline{T}^\vee] \to \PhiMod_{K, 2}^{\et} \times_\cO \bF.
\]
\end{prop}
\begin{proof}
Let $R$ be an $\bF$-algebra and $x_1,x_2$ be two $R$-points of $\Spec\!\!(S^{\tld u}/K^{\det})$, 
i.e.\ matrices of Frobenius $(\Phi^{(j)}_1)_{j \in \mJ},(\Phi^{(j)}_2)_{j \in \mJ}$, 
which give rise to the same $R$-point of 
$\PhiMod_{K, 2}^{\et}(R)$.
We need to prove that~$\Phi_1$ and~$\Phi_2$ are $\uline T^\vee$-conjugate.
Since the source and target of $\operatorname{ch}_\cM$ are limit-preserving, we can assume that~$R$ is Noetherian, and so we can use notation from Section~\ref{loop groups}.
Our assumption is that there exists $(X_j)_{j \in \mJ}\in L\cG_\bF(R)^{\cJ}$ such that
\[X_j\Phi^{(j)}_1=\Phi^{(j)}_2\varphi(X_{j\circ \Frob_p})\]
For~$i \in \{1, 2\}$, let $M_i^{(j)} := \Phi_i^{(j)}v^3\tld w^*(\rhobar)_j^{-1}$. 
These are specializations of the matrix $\Psi^{\tld u_j}v^3$ displayed in
Section~\ref{subsec: charts for PhiMod}, and so they have the following properties:
\begin{itemize}
\item $M_i^{(j)}\in \Mat_2(R[\![v]\!])$ and $\det M_i \in R^\times v^6$, in particular $v^6M_i^{-1}\in\Mat_2(R[\![v]\!])$ (recall that~$p$ is zero in~$R$).
\item $v^{-3}M_i^{(j)}\in T^\vee(R) \times L^{--}\cG_\bF(R)$ (see Section~\ref{loop groups} for~$L^{--}\cG$).
\item $\Phi^{(j)}_iv^3=M^{(j)}_i\tld{w}^*(\rhobar)_j$.
\end{itemize}
The first property implies that $M_i^{(j)} \in L^{[0, 6]}\GL_2(R)$, and so~\cite[Lemma~5.4.4]{LLHLMmodels} shows that~$X_j \in L^+\cG_\bF(R)$ for all~$j \in \cJ$
(using that~$\tld w^*(\rhobar)$ is $7$-deep).
Since $L^+\cG_{\bF}(R)$ coincides with the standard upper-triangular Iwahori subgroup~$\cI$ of $L\GL_2$, 
we see that there exist $t = (t_j)_{j \in \cJ} \in \uline T^\vee(R)$ and $(Y_j)_{j \in \cJ} \in \cI_1(R)^{\cJ}$ 
such that
\[
Y_j(t . \Phi_1)^{(j)} = \Phi_2^{(j)}\varphi(Y_{j \circ \Frob_p}).
\]
It thus suffices to prove that~$Y_j = 1$ for all~$j$.
By the properties of $M_i^{(j)}$, we have 
\begin{equation}\label{first inclusion in monomorphism}
\Phi_iv^3 \in \prod_{j \in \cJ} L^{[0, 6]}\GL_2(R) \tld w^*(\rhobar)_j
\end{equation}
and
\begin{equation}\label{second inclusion in monomorphism}
\Phi_i \in \prod_{j \in \cJ}(T^\vee(R) \times L^{--}\cG(R))\tld w^*(\rhobar)_j.
\end{equation}
Note that the right-hand side of~\eqref{first inclusion in monomorphism} and~\eqref{second inclusion in monomorphism} is stable under the shifted conjugation action of $\uline T^\vee(R)$.
If~$Y_j \ne 1$ for some~$j$, by~\cite[Lemma~5.2.2]{LLHLMmodels} and~\eqref{first inclusion in monomorphism} we find that there exists $Y' \ne 1 \in \cI_1(R)^{\cJ}$ such that
$(\Phi_1v^3)\tld w^*(\rhobar)^{-1} = Y'(t \cdot \Phi_1v^3)\tld w^*(\rhobar)^{-1}$,
which implies that
$\Phi_1\tld w^*(\rhobar)^{-1} = Y'(t \cdot \Phi_1)\tld w^*(\rhobar)^{-1}$.
Because of~\eqref{second inclusion in monomorphism}, this is a contradiction to the fact that the multiplication map
\[
\cI_1(R) \times T^\vee(R) \times L^{--}\cG(R) \to L\cG(R)
\]
is injective, by~\cite[Lemma~3.2.2]{LLHLMmodels}.
This concludes the proof.
\end{proof}

\begin{rmk}
By a consideration of $R$-points for any $\cO/p^a$-algebra, and an induction on~$a$ as in the proofs of~\cite[Proposition~5.4.3, Proposition~5.2.7]{LLHLMmodels},
one can show that Proposition~\ref{monomorphism} remains true without tensoring with~$\bF$.
\end{rmk}

By construction, there is a point $S^{\tld u}/K^{\det} \to \bF$ such that the corresponding specialization of $\cM$ is the \'etale $\varphi$-module of 
$\rhobar_\infty := \rhobar |_{\Gal_{K_\infty}}$.
It is defined by sending $\Psi^{\tld u_j}$ to the matrix~$\Delta_j N_j$ defined in Section~\ref{subsec: single type deformation rings}.
By Proposition~\ref{monomorphism}, it is unique up to shifted conjugation,
and we denote it by~$\cM_{\rhobar}$. 
We define
\[
S^{\tld u}_{\rhobar} := (S^{\tld u})^\wedge_{\cM_{\rhobar}}, \;\;\; 
K_{\rhobar}^{\lambda, \tau_{\tld w}} := K_{\nabla_\infty}^{\lambda, \tau_{\tld w}}\cdot S^{\tld u}_{\rhobar}. \;\;\;
\]

\begin{prop}\label{multitype}
For any~$T$ as in~\eqref{definition of T},
there exists an isomorphism 
\begin{equation}\label{multitype isomorphism}
\left ( S^{\tld u}_{\rhobar}/\bigcap_{(\lambda, \tau_{\tld w}) \in T}K_{\rhobar}^{\lambda, \tau_{\tld w}} \right )[\![Y_1, Y_2, Y_3, Y_4]\!] \cong R^T_{\rhobar}[\![X_1, \ldots, X_{2f}]\!]
\end{equation}
such that for all~$(\lambda, \tld w) \in T$, the ideal generated by $K_{\rhobar}^{\lambda, \tau_{\tld w}}$ is mapped to the ideal generated by $I^{\lambda, \tau_{\tld w}}$.    
\end{prop}
\begin{proof}
Consider the map
\[
R^\square_{\rhobar}[\![X_i]\!] \to \prod_{(\lambda, \tau_{\tld w}) \in T}R_{\rhobar}^{\lambda, \tau_{\tld w}}[\![X_i]\!] = 
\prod_{(\lambda, \tld w) \in T}(T^{\tld w}/I_{\nabla_\infty}^{\lambda, \tau_{\tld w}})^\wedge_{A_{\rhobar, \tld w}}[\![Y_j]\!]
\]
obtained as a product of~\eqref{local models fixed weight}, whose image is by definition $R^T_{\rhobar}[\![X_i]\!]$.
Now let $R_{\rhobar_\infty}^\square$ be the universal lifting ring of $\cM_{\rhobar}$, which is the ring denoted~$R^\square$ in~\cite[Section~5.4]{EGschemetheoretic}.
Passing to the $\varphi$-module of the universal lift of~$\rhobar$ yields a continuous morphism $R^\square_{\rhobar_\infty} \to R_{\rhobar}^\square[\![X_i]\!]$,
which is seen to be surjective by a computation in Galois cohomology.
It thus suffices to prove that the composition
\[
R^\square_{\rhobar_\infty} \to 
\prod_{(\lambda, \tld w) \in T}(T^{\tld w}/I_{\nabla_\infty}^{\lambda, \tau_{\tld w}})^\wedge_{A_{\rhobar, \tld w}}[\![Y_j]\!]
\]
factors as a composition
\[
R^\square_{\rhobar_\infty} \to (S^{\tld u}_{\rhobar}/K^{\det}S_{\rhobar}^{\tld u})[\![Y_j]\!] \xrightarrow{\prod \pr_{\tld w}} \prod_{(\lambda, \tld w) \in T}(T^{\tld w}/I_{\nabla_\infty}^{\lambda, \tau_{\tld w}})^\wedge_{A_{\rhobar, \tld w}}[\![Y_j]\!].
\]
in which the first arrow is surjective.
The factorization exists by~\eqref{M interpolates}. %
Since, by~\cite[Section~5.4.12]{EGschemetheoretic}, $R^\square_{\rhobar_\infty}$ is a versal ring to 
$\PhiMod_{K, 2}^{\et}$ at~$\rhobar_{\infty}$, the first arrow is surjective by Proposition~\ref{monomorphism} (note that since source and target are $p$-adically separated and complete,
it suffices to check surjectivity after quotienting out by~$p$).
\end{proof}

We now record some consequences of Proposition~\ref{multitype}.
Recall that~$H$ is either the upper-triangular Iwahori subgroup of~$\GL_2(K)$, or the maximal compact subgroup of~$D^\times$.

\begin{prop}\label{higherweight}
Let~$\chi : H \to \cO^\times$ be a regular tame character, choose $j \in \cJ$, and let~$\alpha \in \{\pm \alpha_j\} \subset \Phi$. 
Assume that $\rhobar$ is $9$-generic, that $R^{\eta, \tau(\chi)}_{\rhobar} \ne 0$, and that~$R^{\eta, \tau(\chi\alpha)}_{\rhobar} \ne 0$.
Then there exist two elements~$x, y$ of $R:= R^\square_{\rhobar}$
such that
\begin{enumerate}
    \item $xy-p \in I^{\eta+\alpha_j, \tau(\chi\alpha)}$;
    \item $I^{\eta+\alpha_j, \tau(\chi\alpha)} + xR = I^{\eta, \tau(\chi)}+pR$;
    \item $I^{\eta+\alpha_j, \tau(\chi\alpha)} + yR = I^{\eta, \tau(\chi\alpha)}+pR$.
\end{enumerate}
\end{prop}
\begin{proof}
After a twist, it suffices to prove the proposition with~$\eta$ replaced by~$\sweight$.
Let~$\lambda = \sweight + \alpha_j$.
By Lemma~\ref{inertial JL} we can choose lowest alcove presentations of~$\tau(\chi), \tau(\chi\alpha)$ such that, 
if we write $\tld u := \tld w^*(\rhobar, \tau(\chi))$ and~$\tld w = \tld w^*(\rhobar, \tau(\chi\alpha))$
(so that $\tau(\chi) = \tau_{\tld u}$ and~$\tau(\chi\alpha) = \tau_{\tld w}$)
then $\tld w = \tld u t_{\pm \alpha}$.
Furthermore, Lemma~\ref{types lifting rhobar} shows that $\tld u, \tld w \in \Adm^\vee_{\rhobar}\sweight$.
These facts imply that $\{\tld u_j, \tld w_j\} = \{t_{({2, 1})_j}, t_{({1, 2})_j}\}$.

Define 
\[
T := \{(\sweight, \tld u), (\sweight, \tld w), (\lambda, \tld w)\}.
\]
Then it suffices to prove the proposition with $R$ replaced by~$R^T_{\rhobar}$.
Hence, by Proposition~\ref{multitype}, it suffices to construct~$x, y \in S^{\tld u}$ such that
\begin{enumerate}
    \item $xy - p \in K_{\nabla_\infty}^{\lambda, \tau_{\tld w}}$;
    \item $K_{\nabla_\infty}^{\lambda, \tau_{\tld w}} + x S^{\tld u} = K_{\nabla_\infty}^{\sweight, \tau_{\tld u}} + p S^{\tld u}$;
    \item $K_{\nabla_\infty}^{\lambda, \tau_{\tld w}} + y S^{\tld u} = K_{\nabla_\infty}^{\sweight, \tau_{\tld w}} + p S^{\tld u}$.
\end{enumerate}
We will do this in the case where $\tld u_j = t_{(2, 1)_j}, \tld w_j = t_{(1, 2)_j}$; the other case is similar, and can be reduced to this after conjugation by a normalizer of Iwahori.

Recall the description of~$S^{\tld u}$ in Section~\ref{subsec: charts for PhiMod},
and the description of $K_{\nabla_1}^{\lambda, \tau_{\tld w}}$ in Section~\ref{subsec: multi-type approximation}.
Let $x = \beta_1, y = \gamma_2 \in S^{\tld u_j}$.
After rescaling~$y$ by an element of $\bZ_{(p)}^\times$, we see that 
\[
xy - p \in K^{(3, 0)_j, \tau_{\tld w}}_{\nabla_1} \subset K_{\nabla_\infty}^{\lambda, \tau_{\tld w}} + p^{2}S^{\tld u}, 
\]
where the inclusion is a consequence of Lemma~\ref{multitype local models}.
Hence, after rescaling~$y$ by an element of $(S^{\tld u})^\times$, we see that part~(1) is true.

We now prove part~(2).
Recall that~$\lambda_j = (3, 0)_j$ and $\lambda_i = (2, 1)_i$ if~$i \ne j \in \cJ$.
By inspection of the formulas in Section~\ref{subsec: multi-type approximation}, we have
\[K^{\lambda_j, \tau_{\tld w}}_{\nabla_1} + x S^{\tld u_j} = K^{(2, 1)_j, \tau_{\tld u}}_{\nabla_1}+pS^{\tld u_j}.\]
On the other hand, the fact that $\tld w = \tld u t_{\pm \alpha}$ implies that, for all~$i \ne j \in \cJ$, we have
\[K^{\lambda_i, \tau_{\tld w}}_{\nabla_1} = K^{(2, 1)_i, \tau_{\tld u}}_{\nabla_1}.\]
This justifies the first equality in
\begin{equation}\label{p-comparison higher weight}
K^{\lambda, \tau_{\tld w}}_{\nabla_1} + xS^{\tld u} = K^{\sweight, \tau_{\tld u}}_{\nabla_1} + pS^{\tld u} = K_{\nabla_\infty}^{\sweight, \tau_{\tld u}} + pS^{\tld u},
\end{equation}
whereas the second equality is a consequence of Lemma~\ref{multitype local models}.
There remains to prove that
\[
K^{\lambda, \tau_{\tld w}}_{\nabla_1} + xS^{\tld u} = K_{\nabla_\infty}^{\lambda, \tau_{\tld w}} + xS^{\tld u}.
\]
Again by Lemma~\ref{multitype local models}, it suffices to prove that~$p$ is contained in both sides of this equation.
It is contained in the left-hand side by~\eqref{p-comparison higher weight}, and in the right-hand side by part~(1).
This concludes the proof of part~(2).
The proof of part~(3) is similar.
\end{proof}

\begin{prop}\label{ideals for arm cyclicity}
Let~$\chi : H \to \cO^\times$ be a regular tame character.
Assume that $\rhobar$ is $9$-generic and $R^{\eta, \tau(\chi)}_{\rhobar} \ne 0$.
Choose $j \in \cJ$ and $\alpha \in \{\pm\alpha_j\} \subset \Phi$.
Then
\[
I^{\eta, \tau(\chi)} + I^{\eta+\alpha_j, \tau(\chi\alpha)} = I^{\eta, \tau(\chi)} + pR^\square_{\rhobar}.
\]
\end{prop}
\begin{proof}
It suffices to prove that $I^{\eta+\alpha_j, \tau(\chi\alpha)} \subset I^{\eta, \tau(\chi)} + pR^\square_{\rhobar}$ and that 
$p \in I^{\eta, \tau(\chi)} + I^{\eta+\alpha_j, \tau(\chi\alpha)}$.    
After a twist, it suffices to prove this with~$\eta$ replaced by~$\sweight$.
By Lemma~\ref{inertial JL} and Lemma~\ref{types lifting rhobar} we can find lowest alcove presentations of~$\tau(\chi), \tau(\chi\alpha)$ such that, writing $\tld u := \tld w^*(\rhobar, \tau(\chi))$ and~$\tld w = \tld w^*(\rhobar, \tau(\chi\alpha))$,
we have $\tld u \in \Adm^\vee_{\rhobar}\sweight$ and $\tld w = \tld u t_{\pm \alpha}$.
Then $\tau(\chi) = \tau_{\tld u}$ and~$\tau(\chi\alpha) = \tau_{\tld w}$.

Define $T := \{(\sweight, \tld u), (\sweight + \alpha_j, \tld w)\}$.
By Proposition~\ref{multitype isomorphism}, it suffices to prove that
\begin{enumerate}
\item $K_{\nabla_\infty}^{\sweight+\alpha_j, \tau_{\tld w}} \subset K_{\nabla_\infty}^{\sweight, \tau_{\tld u}} + pS^{\tld u}$. 
\item $p \in K_{\nabla_\infty}^{\sweight, \tau_{\tld u}} + K_{\nabla_\infty}^{\sweight+\alpha_j, \tau_{\tld w}}$.
\end{enumerate}
We begin by proving part~(1).
Because of Lemma~\ref{multitype local models}, it suffices to prove that
\[
K_{\nabla_1}^{\sweight+\alpha_j, \tau_{\tld w}} + pS^{\tld u} \subset K_{\nabla_1}^{\sweight, \tau_{\tld u}} + pS^{\tld u},
\]
which is a consequence of the formulas in Section~\ref{subsec: multi-type approximation}.

We now prove part~(2).
Note that, if~$J \subset S^{\tld u}$ is any ideal,
then $p \in J$ if and only if $p \in J + p^2 S^{\tld u}$, because $S^{\tld u}$ is $p$-complete.
So, by Lemma~\ref{multitype local models}, it suffices to prove that $p \in K_{\nabla_1}^{(2, 1), \tau_{\tld u}} + K_{\nabla_1}^{(2, 1)+\alpha_j, \tau_{\tld w}}$. 
In turn, it suffices to prove that
\begin{equation}\label{containment of p}
p \in K^{(2, 1)_j, \tau_{\tld u}}_{\nabla_1} + K^{(3, 0)_j, \tau_{\tld w}}_{\nabla_1} \subset S^{\tld u_j}.
\end{equation}
Again, this is true by inspection.
In more detail, if~$\tld u_j = t_{(2, 1)_j}$, then this follows when $\tld w_j = t_{(1, 2)_j}$ from
\[\delta_1+p \in K^{(2, 1)_j, \tau_{\tld u}}_{\nabla_1},\;\;\; \delta_1+(\kappa^{\tld w}_j -1)^{-1}2p \in K^{(3, 0)_j, \tau_{\tld w}}_{\nabla_1}
\]
and when~$\tld w_j = t_{(3, 0)_j}$ from
\[
\delta_1+p \in K^{(2, 1)_j, \tau_{\tld u}}_{\nabla_1},\;\;\; \delta_1+2p \in K^{(3, 0)_j, \tau_{\tld w}}_{\nabla_1}.
\]
The other cases are similar.
\end{proof}
For the next result, we introduce the following notation for all $\alpha \in \Phi$: we put $|\alpha| := \alpha$ when~$\alpha$ is positive 
and $|\alpha| := -\alpha$ when~$\alpha$ is negative.

\begin{prop}\label{ideals for Wchi3}
Let~$\chi : H \to \cO^\times$ be a regular tame character.
Assume that $\rhobar$ is $9$-generic and $R^{\eta, \tau(\chi)}_{\rhobar} \ne 0$.
Let~$I := I^{\eta, \tau(\chi)} + pR^\square_{\rhobar}$, and for all~$\alpha \in \Phi$, let 
\[I^{\chi, \alpha} := (I^{\eta, \tau(\chi)} \cap I^{\eta+|\alpha|, \tau(\chi\alpha)}) + pR^\square_{\rhobar}.\]
Then the sum of the inclusions $I^{\chi, \alpha} \to I$ induce a surjection
\[\bigoplus_{\alpha \in \Phi}I^{\chi, \alpha} \to I^{\oplus \Phi}/\Delta(I)\] 
where~$\Delta(I)$ denotes the diagonal embedding of~$I$.
\end{prop}
\begin{proof}
After a twist, it suffices to prove the proposition with~$\eta$ replaced by~$\sweight$.
Choose a lowest alcove presentation of~$\tau(\chi)$ such that $\tld u := \tld w^*(\rhobar, \tau) \in \Adm^\vee_{\rhobar}\sweight$.
Then Lemma~\ref{inertial JL} shows that
\[
\{\tld w^*(\rhobar, \tau(\chi\alpha)): \alpha \in \Phi\} = \{\tld u t_{\alpha}: \alpha \in \Phi\},
\]
and so 
\[\{\tau(\chi\alpha): \alpha \in \Phi\} = \{\tau_{\tld u t_{\alpha}}: \alpha \in \Phi\}.\]
An application of Nakayama's lemma then shows that 
it suffices to prove the proposition with $R^\square_{\rhobar}$ replaced by~$R^T_{\rhobar}$, where 
\[
T := \{((2, 1), \tld u)\} \cup \{((2, 1)+|\alpha|, \tld ut_{\alpha}): \alpha \in \Phi\}.
\]
To do so, we will apply Proposition~\ref{multitype}.
For all $\alpha \in \Phi$, define
\begin{gather*}
K := K^{\sweight, \tau_{\tld u}}_{\nabla_1} + p S^{\tld u} =  K_{\nabla_\infty}^{\sweight, \tau_{\tld u}} + pS^{\tld u}\\
K_{\nabla_\infty, \alpha} := (K_{\nabla_\infty}^{\sweight, \tau_{\tld u}} \cap K_{\nabla_\infty}^{\sweight +|\alpha|, \tau_{\tld ut_{\alpha}}}) + pS^{\tld u}\\
K_{\nabla_1, \alpha} := (K_{\nabla_1}^{\sweight, \tau_{\tld u}} \cap K_{\nabla_1}^{\sweight +|\alpha|, \tau_{\tld ut_{\alpha}}}) + pS^{\tld u}
\end{gather*}
where the equality in the first line is a consequence of Lemma~\ref{multitype local models}.
Then, by Proposition~\ref{multitype}, it suffices to prove that the sum of inclusions
\begin{equation}\label{to prove surjective II}
\bigoplus_{\alpha \in \Phi}K_{\nabla_\infty, \alpha} \to K^{\oplus \Phi}/\Delta(K)
\end{equation}
becomes surjective after completing at $\cM_{\rhobar}$.
Write $S^{\tld u} \to \bF_{\rhobar}$ for the residue field at the maximal ideal corresponding to~$\cM_{\rhobar}$. 
Let~$V := K \otimes_{S^{\tld u}} \bF_{\rhobar}$, and define
\begin{gather*}
V_{\nabla_\infty, \alpha} := \operatorname{image}(K_{\nabla_\infty, \alpha} \to V)\\
V_{\nabla_1, \alpha} := \operatorname{image}(K_{\nabla_1, \alpha} \to V)
\end{gather*}
By Nakayama's lemma,
surjectivity of~\eqref{to prove surjective II} after completion at $\cM_{\rhobar}$ is implied by the following statement, which is what we will prove:
for all~$\alpha \in \Phi$, we have
\begin{equation}\label{to prove for gluing}
V_{\nabla_\infty, \alpha} + \bigcap_{\beta \in \Phi: \beta \ne \alpha}V_{\nabla_\infty, \beta}  = V.
\end{equation}
We begin by proving that $V_{\nabla_\infty, \alpha} = V_{\nabla_1, \alpha}$ for all~$\alpha \in \Phi$.
For this, it suffices to prove that
\[\operatorname{image}(K_{\nabla_\infty}^{\sweight, \tau_{\tld u}} \cap K_{\nabla_\infty}^{\sweight +|\alpha|, \tau_{\tld ut_{\alpha}}} \to V) = 
\operatorname{image}(K_{\nabla_1}^{\sweight, \tau_{\tld u}} \cap K_{\nabla_1}^{\sweight +|\alpha|, \tau_{\tld ut_{\alpha}}} \to V).\]
This is a special case of~\eqref{comparing intersections} below, applied to $I_1 = K_{\nabla_1}^{\sweight, \tau_{\tld u}}, J_1 = K_{\nabla_1}^{\sweight +|\alpha|, \tau_{\tld ut_{\alpha}}}$, and
$I_\infty = K_{\nabla_\infty}^{\sweight, \tau_{\tld u}}, J_\infty = K_{\nabla_\infty}^{\sweight +|\alpha|, \tau_{\tld ut_{\alpha}}}$.
Note that the assumptions of~\eqref{comparing intersections} are met, by Lemma~\ref{multitype local models}, which shows the congruence modulo~$p^2$,
and Proposition~\ref{ideals for arm cyclicity}, which shows that $p \in I_\infty + J_\infty$ and $p \in I_1 + J_1$ (this is step~(2) in the proof of Proposition~\ref{ideals for arm cyclicity}).

It thus suffices to prove that, for all~$\alpha \in \Phi$,
\begin{equation}\label{to prove for gluing II}
V_{\nabla_1, \alpha} + \bigcap_{\beta \in \Phi: \beta \ne \alpha}V_{\nabla_1, \beta} = V.
\end{equation}
To do so, observe first that by definition
\[
V = 
\sum_{i \in \cJ}(K^{(2, 1)_i, \tau_{\tld u}}_{\nabla_1}, p) \otimes_{S^{\tld u_i}} \bF_{\rhobar}.  
\]
Next, choose~$\alpha \in \Phi$, and let~$j \in \cJ$ be the embedding in which~$\alpha$ is not trivial, so that~$|\alpha| = \alpha_j$.
We next show that
\begin{equation}\label{intermediate step for gluing}
\sum_{i \ne j}(K^{(2, 1)_i, \tau_{\tld u}}_{\nabla_1},p) \otimes_{S^{\tld u_i}} \bF_{\rhobar}  \subset V_{\nabla_1, \alpha}.
\end{equation}
Observe that $\sweight$ and $\sweight + |\alpha|$ only differ at~$j$, and the same is true for $\tld u$ and $\tld u t_{\alpha}$.
Hence
\[
K^{(2, 1)_i, \tau_{\tld u}}_{\nabla_1} = K^{((2, 1)+|\alpha|)_i, \tau_{\tld u t_\alpha}}_{\nabla_1}
\]
as ideals of $S^{\tld u_i}$, for all~$i \ne j$, which immediately implies~\eqref{intermediate step for gluing}.
To conclude the proof of~\eqref{to prove for gluing II}, it thus suffices to prove that
\[
(K^{(2, 1)_j, \tau_{\tld u}}_{\nabla_1}, p) \otimes_{S^{\tld u_j}} \bF_{\rhobar} \;\; \subset V_{\nabla_1, \alpha} + \bigcap_{\beta \in \Phi: \beta \ne \alpha}V_{\nabla_1, \beta}. 
\]
To do this, it suffices to prove that the inclusion
\[
(K^{(2, 1)_j, \tau_{\tld u}}_{\nabla_1}\cap K^{(3, 0)_j, \tau_{\tld u t_{\alpha}}}_{\nabla_1}, p) \otimes_{S^{\tld u_j}} \bF_{\rhobar} + 
(K^{(2, 1)_j, \tau_{\tld u}}_{\nabla_1}\cap K^{(3, 0)_j, \tau_{\tld u t_{-\alpha}}}_{\nabla_1}, p) \otimes_{S^{\tld u_j}} \bF_{\rhobar}
\subset
(K^{(2, 1)_j, \tau_{\tld u}}_{\nabla_1}, p) \otimes_{S^{\tld u_j}} \bF_{\rhobar},
\]
is an equality, because for all~$\beta \in \Phi \setminus \{\alpha\}$ we have 
\[
(K^{(2, 1)_j, \tau_{\tld u}}_{\nabla_1}\cap K^{(3, 0)_j, \tau_{\tld u t_{-\alpha}}}_{\nabla_1}, p) \otimes_{S^{\tld u_j}} \bF_{\rhobar} \subset V_{\nabla_1, \beta}.
\]
We have therefore reduced the proof of the proposition to showing that the natural map
\begin{multline}\label{to prove surjective}
(K^{(2, 1)_j, \tau_{\tld u}}_{\nabla_1}\cap K^{(3, 0)_j, \tau_{\tld u t_{\alpha}}}_{\nabla_1}, p) \otimes_{S^{\tld u_j}} \bF_{\rhobar} + 
(K^{(2, 1)_j, \tau_{\tld u}}_{\nabla_1}\cap K^{(3, 0)_j, \tau_{\tld u t_{-\alpha}}}_{\nabla_1}, p) \otimes_{S^{\tld u_j}} \bF_{\rhobar}
\to\\
(K^{(2, 1)_j, \tau_{\tld u}}_{\nabla_1}, p) \otimes_{S^{\tld u_j}} \bF_{\rhobar}
\end{multline}
is surjective for any~$j \in \cJ$ and any~$\alpha \in \{\pm \alpha_j\}$.
Note that we can assume without loss of generality that~$\alpha = \alpha_j$.

Fix~$j \in \cJ$, let~$\fm$ be the kernel of $S^{\tld u_j} \to \bF_{\rhobar}$, and let $S_{\rhobar}^{\tld u_j}$ be the completion of~$S^{\tld u_j}$ at~$\fm$.
Since~$S_{\rhobar}^{\tld u_j}$ is flat over~$S^{\tld u_j}$, 
it suffices to prove surjectivity of~\eqref{to prove surjective}
after completion at~$\fm$
(i.e.\ after replacing~$S^{\tld u_j}$ in~\eqref{to prove surjective} with~$S^{\tld u_j}_{\rhobar}$).
This is true by inspection of the formulas in Section~\ref{subsec: multi-type approximation}, 
and repeated applications of Lemma~\ref{intersection distortion} and Lemma~\ref{linear distortion}.

In more detail, let~$\mathfrak{I}$ be the image of~\eqref{to prove surjective}. 
Assume first that $\tld u_j = t_{(2, 1)_j}$.
Then
\begin{equation}\label{formula for K modulo p I}
(K^{(2, 1)_j, \tau_{\tld u}}_{\nabla_1}, p) \otimes_{S^{\tld u_j}_{\rhobar}} \bF_{\rhobar} = \langle p, \alpha_0, \beta_0, \gamma_0, \delta_0, \beta_1, \alpha_2, \delta_1, \gamma_1, \alpha_1 \rangle.
\end{equation}
Note that, since $\Phi^{\tld u_j} := \Psi^{\tld u_j}\tld w^*(\rhobar)$ is congruent modulo~$\fm$ to the matrix~$\Phi^{(j)}_{\rhobar}$ of Frobenius on the $\varphi$-module of~$\rhobar$, 
the formula~\eqref{matrices of Phi on rhobar} shows that~\eqref{formula for K modulo p I} is indeed contained in~$\fm$. However, $\gamma_2$ may or may not be in~$\fm$.

Since $\alpha_1, \beta_1\in K^{(2, 1)_j, \tau_{\tld u}}_{\nabla_1} \cap K^{(3, 0)_j, \tau_{\tld u t_\alpha}}_{\nabla_1}$, we see that $\alpha_1, \beta_1 \in \mathfrak{I}$.
Since 
\[
\gamma_0 + p\gamma_1 + p^2\gamma_2 \in K^{(2, 1)_j, \tau_{\tld u}}_{\nabla_1} \cap K^{(3, 0)_j, \tau_{\tld u t_{-\alpha}}}_{\nabla_1}
\]
we see that~$\gamma_0 \in \mathfrak{I}$.
We now prove that $\alpha_0, \beta_0, \delta_0 \in \mathfrak{I}$.
By Lemma~\ref{intersection distortion} it suffices to prove that
\begin{gather*}
\alpha_0 \in (K^{(2, 1)_j, \tau_{\tld u}}_{\nabla_1}+p\fm) \cap (K^{(3, 0)_j, \tau_{\tld u t_\alpha}}_{\nabla_1}+p\fm)\\
\beta_0 \in (K^{(2, 1)_j, \tau_{\tld u}}_{\nabla_1}+p\fm) \cap (K^{(3, 0)_j, \tau_{\tld u t_\alpha}}_{\nabla_1}+p\fm)\\
\delta_0 \in (K^{(2, 1)_j, \tau_{\tld u}}_{\nabla_1}+p\fm) \cap (K^{(3, 0)_j, \tau_{\tld ut_\alpha}}_{\nabla_1}+p\fm).
\end{gather*}
These are because
\begin{gather*}
\alpha_0+p\alpha_1+p^2\alpha_2+p^3 \in K^{(2, 1)_j, \tau_{\tld u}}_{\nabla_1} \text{ and } \alpha_0 \in K^{(3, 0)_j, \tau_{\tld u t_\alpha}}_{\nabla_1}\\
\beta_0+p\beta_1 \in K^{(2, 1)_j, \tau_{\tld u}}_{\nabla_1} \text{ and } \beta_0 \in K^{(3, 0)_j, \tau_{\tld u t_\alpha}}_{\nabla_1}\\ 
\delta_0+p\delta_1+p^2 \in K^{(2, 1)_j, \tau_{\tld u}}_{\nabla_1} \text{ and } \delta_0-p^2 \in K^{(3, 0)_j, \tau_{\tld u t_\alpha}}_{\nabla_1}.
\end{gather*}
When $\gamma_2 \in \fm$, a similar argument also works for~$\gamma_1$, using
\[
\text{$\gamma_1 + p\gamma_2 \in K^{(2, 1)_j, \tau_{\tld u}}_{\nabla_1}$ and $\gamma_1+2p\gamma_2 \in K^{(3, 0)_j, \tau_{\tld u t_{-\alpha}}}_{\nabla_1}$}.
\]
However, when~$\rhobar$ is not split, it may be the case that $\gamma_2 \not \in \fm$; we will treat this case after proving
that~$\alpha_2, \delta_1 \in \mathfrak{I}$, which we do next.
Since 
\[
\alpha_2+\delta_1+2p-\beta_1\gamma_2 \in K^{(2, 1)_j, \tau_{\tld u}}_{\nabla_1} \cap K^{(3, 0)_j, \tau_{\tld ut_{-\alpha}}}_{\nabla_1}, 
\]
we see that $\alpha_2+\delta_1 \in \mathfrak{I}$ (recall we already know that~$\beta_1 \in \mathfrak{I}$).
On the other hand, for all~$s \in \bZ_p$ we have $\alpha_2-\delta_1 + s\beta_1\gamma_2 \in K^{(2, 1)_j, \tau_{\tld u}}_{\nabla_1}$, and there exists
$s \in \bZ_p$ such that $\alpha_2-\delta_1 + s\beta_1\gamma_2 \in K^{(3, 0)_j, \tau_{\tld ut_{-\alpha}}}_{\nabla_1}$.
Hence $\alpha_2-\delta_1 \in \mathfrak{I}$.
Since~$p \ne 2$, this concludes the proof that $\alpha_2, \delta_1 \in \mathfrak{I}$.
There remains to prove that $\gamma_1 \in \mathfrak{I}$.
Lemma~\ref{linear distortion} applied to $\delta_1 + p, \gamma_1 + p\gamma_2 \in K_{\nabla_1}^{(2, 1)_j, \tau_{\tld u}}$ and 
$\delta_1+2(\kappa_j^{\tau_{\tld u t_{-\alpha}}}-1)^{-1}p, \gamma_1 + 2p\gamma_2 \in K_{\nabla_1}^{(3, 0)_j, \tau_{\tld u t_{-\alpha}}}$ shows that
\[
\gamma_2\delta_1 + (1-2(\kappa_j^{\tau_{\tld u t_{-\alpha}}}-1)^{-1})\gamma_1 \in (K_{\nabla_1}^{(2, 1)_j, \tau_{\tld u}} \cap K_{\nabla_1}^{(3, 0)_j, \tau_{\tld u t_{-\alpha}}}, p).
\]
Since Lemma~\ref{structure constants} implies that the coefficient of~$\gamma_1$ is a unit in~$\bZ_{(p)}$, and we already know that $\delta_1 \in \mathfrak{I}$, we conclude that
$\gamma_1 \in \mathfrak{I}$.
This concludes the proof of the proposition in the case $\tld u_j = t_{(2, 1)_j}$.

The case $\tld u_j = t_{(1, 2)_j}$ is similar.

Assume now that~$\tld u_j = t_{(1, 2)_j}s$.
Then
\[
(K^{(2, 1)_j, \tau_{\tld u}}_{\nabla_1}, p) \otimes_{S^{\tld u_j}} \bF_{\rhobar} = \langle p, \alpha_0, \beta_0, \beta_1, \gamma_0, \gamma_1, \delta_0, \delta_1, \alpha_1, \beta_2\gamma_2 \rangle.
\]
Since 
\[
\alpha_0+p\alpha_1+p^2 \in K^{(2, 1)_j, \tau_{\tld u}}_{\nabla_1} \cap K^{(3, 0)_j, \tau_{\tld ut_\alpha}}_{\nabla_1}
\]
we see that $\alpha_0 \in \mathfrak{I}$.
A similar argument shows that~$\gamma_0, \delta_0, \beta_0 \in \mathfrak{I}$.
Since $\beta_1+p\beta_2 \in K^{(2, 1)_j, \tau_{\tld u}}_{\nabla_1}$ and $\beta_1 + 2(\kappa_j^{\tau_{\tld ut_\alpha}}-1)^{-1}p\beta_2 \in K^{(3, 0)_j, \tau_{\tld ut_\alpha}}_{\nabla_1}$, 
an application of Lemma~\ref{intersection distortion}
shows that $\beta_1 \in \mathfrak{I}$.
When~$\gamma_2 \in \fm$, a similar argument shows that~$\gamma_1 \in \mathfrak{I}$; 
we treat the case where~$\gamma_2$ is invertible after proving
that $\alpha_1, \delta_1 \in \mathfrak{I}$, which we do next.
Observe that $\alpha_1+2p, \delta_1 + 2(\kappa_j^{\tau_{\tld u t_\alpha}}-2)^{-1} p \in K^{(3, 0)_j, \tau_{\tld ut_\alpha}}_{\nabla_1}$ 
and $\alpha_1+p, \delta_1+p \in K^{(2, 1)_j, \tau_{\tld u}}_{\nabla_1}$.
Hence Lemma~\ref{linear distortion} shows that
\[
(1-2(\kappa_j^{\tau_{\tld u t_\alpha}}-2)^{-1})\alpha_1+\delta_1 \in \mathfrak{I}.
\]
Similarly, 
\[
\alpha_1+(1-2(\kappa_j^{\tau_{\tld u t_{-\alpha}}}-2)^{-1})\delta_1 \in \mathfrak{I}.\] 
These equations are linearly independent, and so $\alpha_1, \delta_1 \in \mathfrak{I}$. 
In fact, if they were linearly dependent then $(1-2(\kappa_j^{\tau_{\tld u t_\alpha}}-2)^{-1})(1-2(\kappa_j^{\tau_{\tld u t_{-\alpha}}}-2)^{-1}) \equiv 1 \text{ mod } p\cO$, which implies
\[
(\kappa_j^{\tau_{\tld u t_\alpha}}-2) +
(\kappa_j^{\tau_{\tld u t_{-\alpha}}}-2)-2 \equiv 0. 
\]
If~$(\mathfrak{w}, \mu)$ is our fixed lowest alcove presentation of $\rhobar^{\mathrm{ss}}|_{I_{\bQ_p}}$, then
$\tld w^*(\tau_{\tld u t_{\pm \alpha}}) = t_{\mp \alpha}\tld u^{-1} \mathfrak{w}^{-1}t_{\mu+\eta}$, hence
Lemma~\ref{structure constants} implies that
\[
\kappa_j^{\tau_{\tld u t_\alpha}} + \kappa_j^{\tau_{\tld u t_{-\alpha}}} \equiv \pm 2\langle \mu + \eta, \alpha_j \rangle.
\] 
Since $\rhobar$ is at least $4$-generic, this cannot be congruent to~$6$, which is a contradiction.

Finally, a similar argument shows that 
\[
\gamma_1(2(\kappa_j^{\tld u t_{\alpha}}-2)^{-1}-1) + \delta_1(-\gamma_2) \in \mathfrak{I}
\]
and
\[
(1-2(\kappa_j^{\tau_{\tld u t_\alpha}}-2)^{-1})\beta_2\gamma_2 + 2(\kappa_j^{\tau_{\tld u t_\alpha}}-2)^{-1}\delta_1 \in \mathfrak{I}.
\]
By Lemma~\ref{structure constants},
these imply that $\gamma_1, \beta_2\gamma_2 \in \mathfrak{I}$.
This concludes the proof.
\end{proof}

\begin{lemma}\label{intersection distortion}
Let~$S$ be a local $\mO$-algebra with residue field~$\bF$ and maximal ideal~$\fm_S$, and let~$K \subset S$ be an ideal.
\begin{enumerate}
\item Let~$I, J \subset K$ be ideals of~$S$. Assume that~$I, J$ are $p$-saturated, and that $p \in I+J$.
Let~$x \in K$, and assume that $x \in (I+p\fm_S) \cap (J+p\fm_S)$.
Then
\begin{equation}\label{element of intersection}
x + \fm_S K \in \operatorname{image}(I \cap J \to K/\fm_S K).
\end{equation}
\item Let~$I_1, I_\infty, J_1, J_\infty \subset K$ be ideals of~$S$.
Assume that~$I_1, I_\infty, J_1, J_\infty$ are $p$-saturated, that~$p \in I_1 + J_1$, $p \in I_\infty + J_\infty$, and that $I_1+p^2S = I_\infty + p^2S$, $J_1 + p^2 S = J_\infty + p^2 S$.
Then
\begin{equation}\label{comparing intersections}
\operatorname{image}(I_1 \cap J_1 \to K/\fm_S K) = \operatorname{image}(I_\infty \cap J_\infty \to K/\fm_S K).
\end{equation}
\end{enumerate}
\end{lemma}
\begin{proof}
Proof of~(1):
By~\cite[Lemma~3.17]{LLHMK1} there exists $e \in \fm_S\left ( (I+pS) \cap (J+pS) \right )$ such that $x + e \in I \cap J$.
By assumption $p \in I + J \subset K$, and so $I + pS, J + pS \subset K$.
Hence $e \in \fm_S K$.
This concludes the proof.

Proof of~(2):
Let~$x \in I_1 \cap J_1$.
Since $I_1 + p^2 S = I_\infty + p^2 S$, and similarly for~$J_1$ and~$J_\infty$, we have $x \in (I_\infty + p\fm_S) \cap (J_\infty + p\fm_S)$
Hence part~(1) implies that $x + \fm_SK$ is contained in the right-hand side of~\eqref{comparing intersections}.
Hence the left-hand side of~\eqref{comparing intersections} is contained in the right-hand side.
The reverse containment is proved in the same way.
\end{proof}

\begin{lemma}\label{linear distortion}
Let~$I_1, I_2 \subset K$ be ideals of a local $\mO$-algebra~$S$ with residue field~$\bF$.
Let~$x, y \in S$, and assume that there exist $\alpha, \beta, \gamma, \delta \in S$ such that
\[(x+p\alpha, y+p\beta) \subset I_1, \;\;\; (x+p\gamma, y +p\delta) \subset I_2.\]
Then 
\[
(\delta-\beta)x + (\alpha-\gamma)y \in (I_1 \cap I_2, p).
\]
\end{lemma}
\begin{proof}
Since $(\delta-\beta)\alpha + (\alpha-\gamma)\beta = (\delta-\beta)\gamma + (\alpha-\gamma)\delta = \alpha\delta-\beta\gamma$, we have
\[
(\delta-\beta)x + (\alpha-\gamma)y + p(\alpha\delta-\beta\gamma) \in I_1 \cap I_2,
\]
which implies the lemma.
\end{proof}

\subsection{Fixing the determinant.}
In applications, it will be important to work with fixed-determinant Galois deformation rings.
We continue to assume that $\rhobar: \Gal_K \to \GL_2(\bF)$ is a $9$-generic Galois representation
and that~$\lambda$ is a dominant weight such that $(2, 1) \leq \lambda \leq (3, 0)$.
Let $\zeta: \Gal_K \to \cO^\times$ be a continuous lift of $\det \rhobar$.
Write $R^{\square}_{\rhobar} \to R^{\lambda, \tau, \zeta}_{\rhobar}$ for the reduced $p$-flat quotient corresponding to potentially crystalline lifts of~$\rhobar$
with Hodge--Tate weights~$\lambda$, tame inertial type~$\tau$, and determinant~$\zeta$, and write~$I^{\lambda, \tau, \zeta}$ for the kernel of this map.
Then
there exists~$f_\zeta \in R^\square_{\rhobar}$, independent of~$\tau$ and~$\lambda$, such that
\[
I^{\lambda, \tau, \zeta} = I^{\lambda, \tau} + f_\zeta R^\square_{\rhobar}
\]
for all~$\lambda, \tau$.
Hence Propositions~\ref{higherweight}, \ref{ideals for arm cyclicity} and~\ref{ideals for Wchi3} remain true if we replace~$I^{\lambda, \tau}$ with~$I^{\lambda, \tau, \zeta}$.
In Section~\ref{global} we will also need the following slight generalization.
Let~$A$ be a complete Noetherian local $\cO$-algebra, with residue field~$\bF$, which is furthermore $\cO$-torsion free.
Then~$A$ with the $\fm_A$-adic topology is isomorphic as a topological $\cO$-module to a direct product of copies of~$\cO$. 
Hence the completed tensor product $(-) \widehat \otimes_{\cO} A$ is exact on pseudocompact topological $\cO$-modules.
Let~$R_\infty := R^{\square}_{\rhobar} \widehat \otimes_{\cO} A$, and let~$I_\infty^{\lambda, \tau, \zeta}$ be the kernel of
\[
R_\infty \to R^{\lambda, \tau, \zeta}_{\rhobar} \widehat \otimes_{\cO} A.
\]
Then 
\[
I_\infty^{\lambda, \tau, \zeta} = I^{\lambda, \tau, \zeta}R_\infty = I^{\lambda, \tau, \zeta} \widehat \otimes_{\cO} A.
\]
The first equality shows that Propositions~\ref{higherweight} and~\ref{ideals for arm cyclicity} remain true with~$I^{\lambda, \tau}$ replaced by~$I_\infty^{\lambda, \tau, \zeta}$,
and the second equality together with the exactness of $(-) \widehat \otimes_{\cO} A$ shows that the same is true for Proposition~\ref{ideals for Wchi3}.

%% file: multiplicityone-arXiv.tex
\section{Multiplicity one theorems.}\label{multiplicity}

Let $\rhobar: \Gal_K \to \GL_2(\bF)$ be a $9$-generic Galois representation.
Let~$\lbar \psi := \lbar\epsilon^{-1} \det \rhobar$ and let $\psi: \Gal_K \to \cO^\times$ be the Teichm\"uller lift of~$\lbar \psi$.
We write~$R^{\square, \psi}_{\rhobar}$ is for the universal deformation ring with determinant~$\psi\epsilon$ of~$\rhobar$ to complete Noetherian local $\mO$-algebras with residue field~$\bF$.
Fix integers~$n, m \geq 0$ and define 
\begin{gather*}
R_\infty := R^{\square, \psi}_{\rhobar}[\![x_i, y_i, z_j: 1 \leq i \leq n, 1 \leq j \leq m]\!]/(x_iy_i-p)\\ 
R_\infty^{\lambda, \tau} := R_\infty \widehat\otimes_{R^\square_{\rhobar}} R^{\lambda, \tau}_{\rhobar}.
\end{gather*}

\begin{defn}
Let~$(G, H) \in \{(\GL_2, I), (D^\times, \cO_{D}^\times)\}$.
A $G$--module with an arithmetic action of~$R_\infty$ is an $R_\infty[G]$-module~$M_\infty$ such that
\benum
\item $M_\infty$ is a projective object of the category of pseudocompact $\cO[\![H]\!]$-modules.
\item $M_\infty$ is finitely generated as an $R_\infty[\![H]\!]$-module.
\item if~$\lambda \in X^*(\uline T)^+$, $\chi : H \to \cO^\times$ is a regular tame character, and~$\cL$ is an $H$-stable $\cO$-lattice in $\chi \otimes V(\lambda)$, 
then $M_\infty(\cL)$ is a maximal Cohen--Macaulay module over $R_\infty^{\lambda+\eta, \tau(\chi)}$.
\item the centre $Z(G) \cong K^\times$ acts on~$M_\infty$ via $\psi \circ \operatorname{Art}_K$.
\eenum 
\end{defn}
In~(2), and more generally if~$\cL$ is a finitely generated $\mO[\![H]\!]$-module, we use the notation
\[
M_\infty(\cL) := M_\infty \otimes_{\cO[\![H]\!]} \cL.
\]
Note that~(4) implies that~$Z_1$ acts trivially on~$M_\infty$.

\begin{defn}
A $G$-module~$M_\infty$ with arithmetic action of~$R_\infty$ is \emph{minimal} if $M_\infty(\chi)[1/\pi_E]$ is locally free of rank one 
over~$R_\infty^{\eta, \tau(\chi)}[1/\pi_E]$, 
for every tame regular character $\chi: H \to \mO^\times$. 
It is \emph{self-dual} if
\[
M_\infty(\chi) \cong \Hom_{R_\infty^{\eta, \tau(\chi)}}(M_\infty(\chi), \omega_R)
\]
as $R_\infty^{\eta, \tau(\chi)}$-modules, for every regular tame character $\chi: H \to \mO^\times$.
Here~$\omega_R$ denotes the canonical module of~$R_\infty^{\eta, \tau(\chi)}$.
\end{defn}

In the rest of this section, we will fix a $G$-module $M_\infty$ with a minimal, self-dual arithmetic action of~$R_\infty$.
In Section~\ref{global} we will construct examples by applying the Taylor--Wiles--Kisin patching method to the cohomology 
of totally definite quaternion algebras and Shimura curves over totally real fields.
The reader familiar with this construction might have expected $R_\infty$ to be formally smooth over $R^\square_{\rhobar}$ (which in our notation corresponds to the case~$m = 0$).
The need for this slightly more general choice of~$R_\infty$ comes from the minimality condition at $p$-adic places, and is explained further in Remark~\ref{explaining R_infty}.

\subsection{Multiplicity one for characters and extensions.}
\begin{prop}\label{chi}
Let~$\chi : H \to \bF^\times$ be a regular tame character.
Then the $R_\infty$-module $M_\infty(\chi)$ is cyclic.
\end{prop}
\begin{proof}
Without loss of generality, $M_\infty(\chi) \ne 0$, and so $R_{\rhobar}^{\eta, \tau(\chi)} \ne 0$.
Hence there exists a lowest alcove presentation of~$\tau(\chi)$ such that $\tld w^*(\rhobar, \tau(\chi)) \in \Adm^\vee_{\rhobar}(\eta)$, and so~$R_\infty^{\eta, \tau(\chi)}$ admits a presentation as in
Proposition~\ref{Gorenstein}.
The self-duality condition, together with~\cite[\href{https://stacks.math.columbia.edu/tag/0AV3}{Tag~0AV3}]{stacks-project}, implies that~$M_\infty(\chi)$ is a reflexive $R_\infty^{\eta, \tau(\chi)}$-module.
The minimality condition implies that~$M_\infty(\chi)$ has generic rank one.
By~\cite[\href{https://stacks.math.columbia.edu/tag/0EBM}{Tag~0EBM}]{stacks-project} we deduce that $M_\infty(\chi)$ represents an element of the class group~$\mCl(R_\infty^{\eta, \tau(\chi)})$, 
and the self-duality condition implies the relation
\[
2[M_\infty(\chi)] = [\omega_R] \in \mCl(R_\infty^{\eta, \tau(\chi)}).
\]
By Proposition~\ref{Gorenstein} this implies that $2[M_\infty(\chi)] = 0$.
By Proposition~\ref{2torsion} we deduce that $2[M_\infty(\chi)] = 0$, i.e.\ $M_\infty(\chi)$ is free of rank one over~$R_\infty^{\eta, \tau(\chi)}$, which completes the proof.
\end{proof}

\begin{prop}\label{chiext}
Let~$\chi, \chi' : H \to \F^\times$ be regular tame characters, and let
\[
0 \to \chi' \to \lbar V \to \chi \to 0 
\]
be a nonsplit extension of $\bF[\![H/Z_1]\!]$-modules.
Then the $R_\infty$-module $M_\infty(\lbar V)$ is cyclic.
\end{prop}
\begin{proof}
Since~$\lbar V$ is not split, by Lemma~\ref{arm definition} there exists $\alpha \in \Phi$ such that $\chi' = \chi \alpha$ and $\lbar V \cong E_{\chi, \alpha}$. 
Let~$j$ be the embedding in which~$\alpha$ is not trivial.
If~$R^{\eta, \tau(\chi)}_\infty = 0$ or~$R_\infty^{\eta, \tau(\chi\alpha)} = 0$ then
the proposition is an immediate consequence of Proposition~\ref{chi}. 
We can therefore assume that these deformation rings do not vanish. Note that this implies that the hypotheses of Proposition~\ref{higherweight} are met.

The locally algebraic $E[H]$-representation $[\chi \alpha] \otimes V(1, -1)_j$ has semisimplified mod $\pi_E$ reduction
\[
\chi \oplus \chi\alpha \oplus \chi \alpha^2,
\]
and so it contains an $H$-stable $\mO_E$-lattice~$\cL$ such that
\[
\cosoc_H(\cL \otimes_{\cO} \bF) = \gr_{\fm}^0(\cL \otimes_{\mO} \bF) = \chi, \; \gr_{\fm}^1(\cL \otimes_{\mO} \bF) = \chi \alpha,\; 
\gr_{\fm}^2(\cL \otimes_{\mO} \bF) = \chi \alpha^2.
\]
So it suffices to prove that $M_\infty(\cL)$ is cyclic.
Let~$\cL^*$ be the $H$-stable $\mO_E$-lattice such that
\[
p\cL \subset p\cL^* \subset \cL \subset \cL^*
\]
and $\cL^*/\cL$ has length two.
We have exact sequences
\begin{equation}\label{latticeext1}
0 \to \cL \to \cL^* \to E_{\chi \alpha, \alpha} \to 0
\end{equation}
and 
\begin{equation}\label{latticeext2}
0 \to p\cL^* \to \cL \to \chi \to 0.
\end{equation}
(Recall that $\gr_\fm^0 E_{\chi \alpha, \alpha} = \chi \alpha,\; \gr_\fm^1 E_{\chi \alpha, \alpha} = \chi\alpha^2$.)

By part~(3) of the definition of an arithmetic action on~$M_\infty$, the $R_\infty$-action on $M_\infty(\cL^*)$ factors through $R := R^{\eta+\alpha_j, \tau(\chi\alpha)}_\infty$.
By the exact sequence~\eqref{latticeext2} and the cyclicity of~$M_\infty(\chi)$ (which is Proposition~\ref{chi}),
we therefore see that~$M_\infty(\cL)$ is cyclic if and only if
\[
M_\infty(p\cL^*) \subset \fm_{R_\infty} M_\infty(\cL),
\]
which we now show to be true.
By Lemma~\ref{hypercube}, we know that $M_\infty(\chi \alpha^2) =0$.
This shows that
\[
M_\infty(E_{\chi \alpha, \alpha}) \isom M_\infty(\chi \alpha)
\]
is cyclic, and
isomorphic to $R^{\eta, \tau(\chi\alpha)}_\infty/pR^{\eta, \tau(\chi\alpha)}_\infty$.

We now apply Proposition~\ref{higherweight}. 
We find that there exist $x, y \in R$ such that $xy = p$ and
\[
R/xR = R_\infty^{\eta, \tau(\chi)}/p R_\infty^{\eta, \tau(\chi)} \text{ and } R/yR = R_\infty^{\eta, \tau(\chi\alpha)}/p R_\infty^{\eta, \tau(\chi\alpha)}.
\]
Since, by assumption, these rings do not vanish, we see that $x, y \in \fm_R$.
Since we have shown that $M_\infty(E_{\chi \alpha, \alpha})$ is isomorphic to $R_\infty^{\eta, \tau(\chi\alpha)}/p R_\infty^{\eta, \tau(\chi\alpha)}$, 
we now see from~\eqref{latticeext1} that
\[
y M_\infty(\cL^*) \subset M_\infty(\cL).
\]
Multiplying this by~$x$ and using~$xy = p$, we conclude that
\[
p M_\infty(\cL^*) \subset x M_\infty(\cL).
\]
On the other hand, it follows from~\eqref{latticeext2} that
\[
x M_\infty(\cL) \subset p M_\infty(\cL^*).
\]
Hence $pM_\infty(\cL^*) = xM_\infty(\cL) \subset \fm_R M_\infty(\cL)$, which concludes the proof.
\end{proof}

\subsection{Multiplicity one for~$A_{\chi, \alpha}$.}
\begin{prop}\label{arm cyclicity}
Let~$\chi: H \to \F^\times$ be a regular tame character, let~$\alpha \in \Phi$, and recall from Lemma~\ref{arm definition} the $\F[H]$-modules~$A_{\chi, \alpha}$ of length~$3$.
Assume that $M_\infty(\chi) \ne 0$. 
Then the $R_\infty$-module $M_\infty(A_{\chi, \alpha})$ is cyclic.
\end{prop}
\begin{proof}
It follows from Lemma~\ref{hypercube} that~$M_\infty(\chi\alpha^2) = 0$, and so Proposition~\ref{arm lifting}(2)
implies that it suffices to prove the cyclicity of $M_\infty(\tld A_{\chi, \alpha})$, where $\tld A_{\chi, \alpha}$ is the image of any morphism
\[
\Proj_{\cO[\![H/Z_1]\!]}(\chi) \to [\chi] \oplus \left ([\chi\alpha] \otimes_E V(1, -1)_j \right )
\]
with nonzero projection on both summands.
Let~$\cL_\chi, \cL_{\chi\alpha}$ be the images of these projections.
By Proposition~\ref{arm lifting}(3), there is an exact sequence
\[
0 \to \tld A_{\chi, \alpha} \to \cL_\chi \oplus \cL_{\chi\alpha} \to \chi \to 0.
\]
By Proposition~\ref{chi}, $M_\infty(\cL_\chi)$ is cyclic with annihilator $I_\infty^{\eta, \tau(\chi)}$.
By Proposition~\ref{chiext}, $M_\infty(\cL_{\chi\alpha})$ is cyclic with annihilator $I_\infty^{\eta+\alpha_j, \tau(\chi\alpha)}$.
Hence the cyclicity of $M_\infty(\tld A_{\chi, \alpha})$ is equivalent to
\[
I^{\eta, \tau(\chi)}_\infty + pR_\infty = I_\infty^{\eta, \tau(\chi)} + I_\infty^{\eta+\alpha_j, \tau(\chi\alpha)},
\]
which is Proposition~\ref{ideals for arm cyclicity}. 
\end{proof}

\subsection{Multiplicity one for~$W_{\chi, 3}$.}
\begin{prop}\label{Wchi3bar}
Let $\chi : H \to \F^\times$ be a regular tame character.
Assume that $M_\infty(\chi) \ne 0$.
Then the $R_\infty$-module $M_\infty(\ovl W_{\chi, 3})$ is cyclic.
\end{prop}
\begin{proof}
Applying $M_\infty(-)$ to the exact sequence in Proposition~\ref{presenting Wchi3} we obtain an exact sequence
\[
0 \to M_\infty(\lbar W_{\chi, 3}) \to \bigoplus_{\alpha \in \Phi} M_{\infty}(A_{\chi, \alpha}) \to M_\infty(\chi^{\oplus \Phi}/\Delta(\chi)) \to 0.
\]
By Proposition~\ref{arm cyclicity}, the module~$M_\infty(A_{\chi, \alpha})$ is cyclic with annihilator
\[
I_\infty^{\chi, \alpha} := (I_\infty^{\eta, \tau(\chi)} \cap I_\infty^{\eta+\alpha_j, \tau(\chi\alpha)})+pR_\infty,
\] 
and by Proposition~\ref{chi}, the module $M_\infty(\chi)$ is cyclic with annihilator $I_\infty := I_\infty^{\eta, \tau(\chi)} + pR_\infty$.
Note that $I_\infty^{\chi, \alpha} \subset I_\infty$.
Then, by the long exact sequence of~$\operatorname{Tor}$, the proposition is equivalent to the statement that the inclusions $I_\infty^{\chi, \alpha} \to I_\infty$ induce a surjection
\[\bigoplus_{\alpha \in \Phi} I_\infty^{\chi, \alpha} \to I_\infty^{\oplus \Phi}/\Delta(I_\infty),\]
which is Proposition~\ref{ideals for Wchi3}.
\end{proof}

\begin{cor}\label{Wchi2}
Let $\chi : H \to \F^\times$ be a regular tame character.
Assume that~$M_\infty(\chi) \ne 0$.
Then the $R_\infty$-module $M_\infty(W_{\chi, 2})$ is cyclic.
\end{cor}
\begin{proof}
This follows from Proposition~\ref{Wchi3bar}, since~$W_{\chi, 2}$ is a quotient of~$\overline W_{\chi, 3}$.
\end{proof}

\begin{prop}\label{Wchi3}
Let $\chi : H \to \F^\times$ be a regular tame character.
Assume that $M_\infty(\chi) \ne 0$.    
Then the $R_\infty$-module $M_\infty(W_{\chi, 3})$ is cyclic.
\end{prop}
\begin{proof}
Let~$V$ be a nonzero subspace of $\chi^\perp$ (as defined in Lemma~\ref{multiplicity of characters in Wchi3}) such that $M_\infty(W_{\chi, 3}/V)$ is cyclic.
If~$V$ does not contain~$\chi \alpha_1\alpha_2$ for some linearly independent $\alpha_1, \alpha_2 \in \Phi$, then $M_\infty(V) = 0$, 
since $M_\infty(\chi\alpha^2) = 0$ for all~$\alpha \in \Phi$, by Lemma~\ref{hypercube}~(1).
So we are done in this case.
Otherwise, let~$U$ be the complement of $\chi \alpha_1\alpha_2$ in~$V$.
We claim that $M_\infty(W_{\chi, 3}/U)$ is also cyclic.
This claim implies the proposition by descending induction on $\dim V$, with the base case being Proposition~\ref{Wchi3bar}, which is the case $V = \chi^\perp$ 
(since $\lbar W_{\chi, 3} = W_{\chi, 3}/\chi^\perp$ by definition).

We now prove the claim.
If~$M_\infty(\chi \alpha_1\alpha_2) = 0$ then there is nothing to prove, 
since $M_\infty(V/U) = 0$ and so $M_\infty(W_{\chi, 3}/V) = M_\infty(W_{\chi, 3}/U)$.
On the other hand, if $M_\infty(\chi \alpha_1\alpha_2) \ne 0$ then Lemma~\ref{hypercube}~(2) implies that $M_\infty(\chi\alpha_1) \ne 0$ and $M_\infty(\chi\alpha_2) \ne 0$.
By Lemma~\ref{W2inW3} there is a submodule $N \subset W_{\chi, 3}$ isomorphic to~$W_{\chi\alpha_1, 2}$, hence
\[
\chi\alpha_1\alpha_2 \subset N \subset W_{\chi, 3}.
\]
Let~$\lbar N$ be the image of $N$ in~$W_{\chi, 3}/U$.
Then the kernel of the surjection $\lbar N \to \chi\alpha_1$ contains $\chi \alpha_1\alpha_2$, and $M_\infty(\lbar N)$ is cyclic, by Corollary~\ref{Wchi2}.
So $M_\infty(\lbar N)$ is a cyclic submodule of $M_\infty(W_{\chi, 3}/U)$ 
that properly contains $M_\infty(\chi\alpha_1\alpha_2)$. 
(The containment is proper because $M_\infty(\chi \alpha_1) \ne 0$.)

Now~\cite[Lemma~10.1.13]{EGS} implies that $M_\infty(W_{\chi, 3}/U)$ is cyclic if and only if 
$M_\infty(W_{\chi, 3}/U)/M_\infty(\chi\alpha_1\alpha_2)$ is cyclic.
Since 
\[
M_\infty(W_{\chi, 3}/U)/M_\infty(\chi\alpha_1\alpha_2) \cong M_\infty(W_{\chi, 3}/V), 
\]
which is cyclic by assumption, this concludes the proof.
\end{proof}

%% file: global-arXiv.tex
\section{Global applications.}\label{global}
Recall from the introduction that~$F/\bQ$ is a totally real number field in which~$p$ is unramified,
and~$D$ is a totally definite quaternion algebra with centre~$F$.
Let $\lbar r: \Gal_F \to \GL_2(\bF)$ be a continuous Galois representation. 
Fix a finite order character $\psi: \Gal_F \to \cO^\times$ such that $\det \lbar r = \lbar \psi \lbar \epsilon$.
If~$w$ is a finite place of~$F$, we write $\rhobar_w := \lbar r|_{\Gal_{F_w}}$.
We choose a place~$v|p$, and also write~$\rhobar$ for~$\rhobar_v$.
We make the following assumptions on~$\lbar r$:
\begin{enumerate}
    \item $\lbar r|_{\Gal_{F(\zeta_p)}}$ is absolutely irreducible;
    \item if~$w \mid p$, then $\rhobar_w$ is $1$-generic if $w \ne v$, and $9$-generic if $w = v$;
    \item if~$w \nmid p$, and $\rhobar_w$ or~$D_w$ is ramified, 
    then the universal lifting ring of~$\rhobar_w$ with $\cO$-coefficients and determinant~$\psi_w$ is formally smooth over~$\cO$.
\end{enumerate}
We choose an auxiliary place~$w_1$ of~$F$ as in~\cite[Section~6.2]{EGS}, a uniformizer $\pi_{w_1} \in \cO_{F_{w_1}}$, 
and an eigenvalue $\beta_{w_1} \in \bF$ of~$\rhobar_{w_1}(\Frob_{w_1})$ 
(it is one of the defining properties of~$w_1$ that~$\rhobar_{w_1}$ is unramified with distinct Frobenius eigenvalues).
We write~$S$ for the union of the sets of ramified finite places of~$\lbar r$,
ramified finite places of~$D$, and $p$-adic places of~$F$.
We choose a maximal order~$\cO_D \subset D$, and 
for each place~$w$ where~$D$ is split we choose an isomorphism $D_w \cong M_2(F_w)$ sending~$\cO_D \otimes_{\bZ} \cO_{F_w}$ to~$M_2(\cO_{F_w})$.
Let~$\bT^S := \cO[T_w, S_w^{\pm 1}: w \not \in S \cup \{w_1\}]$,
and let
\[
\fm_{\lbar r} := (T_w - (\bN w)\operatorname{tr}(\Frob_w),  S_w - (\bN w)\det(\Frob_w)) \subset \bT^S.
\]
In Section~\ref{algebraic modular forms} we will attach to~$\lbar r$ a compact open subgroup $K^v \subset (D \otimes_F \bA_F^{\infty, v})^\times$, a smooth $\cO[K^v]$-module~$\cL^v$, and for every compact open
subgroup $K_v \subset D_v^\times$ an $\cO$-module of algebraic modular forms $S(K^vK_v, \cL^v)$.
The space $S(K^vK_v, \cL^v)$ has a standard action of $\bT^S[T_{w_1}]$, for a Hecke operator~$T_{w_1}$ defined in Section~\ref{algebraic modular forms}, 
and using notation from Section~\ref{Hecke algebras}, we define a maximal ideal 
\[
\fm_{Q_0} := (\fm_{\lbar r}, T_{w_1} - \beta_1)\] 
of~$\bT^S[T_{w_1}]$.
We then define
\[
\pi := \varinjlim_{K_v \subset D_v^\times} S(K^vK_v, \cL^v \otimes_{\cO} \bF).
\]
In this section we prove the following more precise version of Theorem~\ref{main theorem}.

\begin{thm}\label{main theorem with all assumptions}
    Under Assumptions~(1), (2) and~(3) above, we have $\dim_{D_v^\times} \pi[\fm_{Q_0}] = [F_v:\bQ_p]$.
\end{thm}

As explained in the introduction, to prove Theorem~\ref{main theorem with all assumptions} it suffices to prove that
\[
\dim_\bF \Hom_{\bF[\![H/Z_1]\!]}(W_{\chi, 3}, \pi[\fm_{Q_0}]) = \dim_\bF \Hom_{\bF[\![H/Z_1]\!]}(\chi, \pi[\fm_{Q_0}]) = 1
\]
for all smooth characters $\chi: H \to \bF^\times$.
By Propositions~\ref{chi} and~\ref{Wchi3}, it therefore suffices to construct a $G$-module~$M_\infty$ with a minimal, self-dual arithmetic action of~$R_\infty$, such that
\[
M_\infty/\fm_{R_\infty}M_\infty \cong \pi[\fm_{Q_0}]^\vee.
\]
By now, this is a standard construction, which goes back to~\cite{EGS,CEGGPS}.
The self-duality property has been introduced first in~\cite{Manning}, which works with $\ell$-adic coefficients with~$\ell \ne p$, and does not fix inertial types;
for this reason, it is not directly available in the literature in the form we need, and so in this section we describe the main points in the construction of~$M_\infty$,
following~\cite[Section~6]{DoLe} and~\cite{GN}.
Although we will focus on the case that~$D$ is totally definite,
the arguments also apply, with minor modifications, when~$D$ splits at precisely one infinite place of~$F$.

\subsection{Algebraic modular forms.}\label{algebraic modular forms}
Let~$K = \prod_{w} K_w \subset (D \otimes_F \bA_F^\infty)^\times$ be a compact open subgroup, and let $\cL$ be a finite free $\cO$-module with a continuous action of~$K$.
We assume that this action extends to a continuous action of~$K(\bA_F^\infty)^\times$ with~$(\bA_F^\infty)^\times$ acting by~$\psi$. 
We define $S(K, \cL)$ as the space of functions
\[
f: D^\times \backslash (D \otimes_F \bA_F^{\infty})^\times \to \cL^*:= \Hom_\cO(\cL, \cO)
\]
such that $f(gu) = u^{-1}f(g)$ for all~$u \in K(\bA_F^\infty)^\times$.

We now describe the pairing on $S(K, \cL)$ that will be the source of the self-duality property of~$M_\infty$.
If~$K_w$ acts trivially on~$\cL$, the Hecke algebra $\cO[K_w \backslash D_w^\times / K_w]$ acts on~$S(K, \cL)$ in a standard way.
The algebra $\cO[K_w \backslash D_w^\times / K_w]$ has an anti-involution~$i_w$, given by inversion on~$D_w^\times$.
Furthermore, if $\psi_w \circ \det$ is trivial on~$K_w$, there is an $\cO$-algebra automorphism of $\cO[K_w \backslash D_w^\times / K_w]$
sending the double coset operator $[K_w g K_w]$ to~$\psi_w(\det g)[K_w g K_w]$.
We denote this automorphism by $T \mapsto \psi_w(\det T)T$. 
A standard construction~\cite[Section~1]{Taylormeromorphic}
then shows that there exists a perfect pairing
\begin{equation}\label{duality pairing}
S(K, \cL) \times S(K, \cL^* \otimes (\psi \circ \det)) \to \cO
\end{equation}
such that $(T x, y) = (x, \iota_w(\psi_w(\det T)T) y)$ for all~$T \in \cO[K_w \backslash D_w^\times / K_w]$.

\subsection{Coefficients.}\label{coefficients}
For every finite place~$w \ne v$ of~$F$, we define a compact open subgroup $K_w \subset D_w^\times$ and an $\cO[K_w]$-module~$\cL_w$, in the following way:
\begin{enumerate}
\item If~$w \mid p$, we let~$K_w$ be the maximal compact subgroup of~$D_w^\times$ if~$D_w$ is not split, and the upper-triangular Iwahori subgroup of
$(\cO_D \otimes_\bZ \cO_{F_w})^\times \cong \GL_2(\cO_{F_w})$ otherwise.
In both cases, we choose a regular tame character $\chi_w : K_w \to \cO^\times$ such that 
\[
R_{\rhobar_w}^{\eta, \tau(\chi_w), \psi_w} \ne 0.
\]
We let~$\cL_w := \cO(\chi_w)$.
\item If~$w \in S, w \nmid p$, Assumption~(3) on~$\rhobar_w$ implies that there is a unique inertial type $\tau_w$ such that $R_{\rhobar_w}^{\tau_w, \psi_w} \ne 0$.
\begin{enumerate}
    \item If~$D_w$ is split, 
    we let~$n_w$ be the conductor of any lift of~$\rhobar_w$ with determinant~$\psi_w$,
    and we let~$K_w \subset \GL_2(F_w) \cong D_w^\times$ be the subgroup of matrices congruent to 
    \[
    \fourmatrix{*}{*}{0}{1} \text{modulo~$\fm_{w}^{n_w}$}.
    \]
    We let~$\cL_w := \cO$ with trivial action of $K_w$.  
    \item If $D_w$ is not split, we let~$K_w$ be the maximal compact subgroup of~$D_w^\times$. 
    The Bernstein component of $\lbar E[D_w^\times]$ corresponding to~$\tau_w$ admits types on the maximal compact subgroup~$K_w$ of~$D_w^\times$.
    These types are not unique, but they are all conjugate under~$D_w^\times$.
    We make an arbitrary choice of such a type~$V_w$.
    By Lemma~\ref{self-duality for D*}, by replacing~$E$ with a finite unramified extension we can assume that~$V_w$ descends to~$E$.
    Then there exists a unique homothety class of $\cO[K_w]$-lattices in~$V_w$, and we let~$\cL_w$ be a representative of this class.
\end{enumerate} 
\item If~$w \not \in S, w \ne w_1$, then $K_w = \GL_2(\cO_{F_w})$ and~$\cL_w := \cO$ with trivial action of~$K_w$.
\item If~$w = w_1$, then $K_w$ is the upper-triangular pro-$p$ Iwahori subgroup of 
$(\cO_D \otimes_\bZ \cO_{F_w})^\times \cong \GL_2(\cO_{F_w})$ and $\cL_w := \cO$ with trivial action of~$K_w$.
There is a Hecke operator (depending on the choice of uniformizer~$\pi_{w_1}$)
\[
T_{w_1} := K_{w_1}\fourmatrix{\pi_{w_1}}{0}{0}{1}K_{w_1} \in \cO[K_{w_1}\backslash D_{w_1}^\times/K_{w_1}].
\]
\end{enumerate}

Finally, we define
\begin{gather*}
K^v := \prod_{w \ne v} K_w \subset (D \otimes_F \bA_f^{\infty, v})^\times\\
\cL^v := \bigotimes_{w \ne v} \cL_w.
\end{gather*}
We regard~$\cL^v$ as a smooth $\cO[K^vD_v^\times]$-module by letting~$D_v^\times$ act trivially, and we assume that this action extends to $K^vD_v^\times(\bA_F^\infty)^\times$
with $(\bA_F^\infty)^\times$ acting via~$\psi$.
(We are therefore assuming that~$\psi_w$ is unramified for all $w \not \in S \cup \{w_1\}$, and tamely ramified when~$w = w_1$.)

\subsection{Galois deformation rings.}
Because of Assumption~(1) on~$\lbar r$, the usual arguments provide us with an integer $q \geq [F:\bQ]+1-|S|$, and 
for all~$N \geq 0$, a set~$Q_N$ of~$q$ Taylor--Wiles primes,
which we take to be empty when~$N = 0$.
We let $R^{\psi}_{F, S, Q_N}$ be the universal deformation ring of $\lbar r : \Gal_{F, S \cup Q_N} \to \GL_2(\bF)$ with $\cO$-coefficients and determinant~$\psi\epsilon$, 
and we let 
\begin{equation}\label{framed deformation ring}
R^{\square, \psi}_{F, S, Q_N} \cong R^{\psi}_{F, S, Q_N}[\![z_1, \ldots, z_{4|S|-1}]\!]
\end{equation}
be the deformation ring with framings
at the places~$w \in S$.
For~$w \in S, w \nmid p$, we define $R_w^{\loc}$ to be the universal deformation ring of $\rhobar_w$ with determinant $\lbar \epsilon \lbar \psi_w$.
By Assumption~(3), it is a 
power series ring over~$\cO$.
For~$w \mid p, w \ne v$, we define $R_w^{\loc} := R_{\rhobar_w}^{\eta, \tau(\chi_w), \psi_w}$, where~$\chi_w$ is defined in part~(1) of Section~\ref{coefficients}.
Finally, we write
\[
R^\loc := R^{\square, \psi_v}_{\rhobar_v} \widehat\otimes \left (\widehat{\otimes}_{w \in S \setminus \{v\}}R_w^{\loc} \right ).
\]
By Proposition~\ref{Gorenstein}, $R^{\loc}$ is isomorphic to
\[
R^{\square, \psi_v}_{\rhobar_v}[\![x_i, y_i, u_j: 1 \leq i \leq n, 1 \leq j \leq m]\!]/(x_iy_i-p)
\]
for some integers~$n, m$.

\begin{rmk}\label{explaining R_infty}
It is often the case in the literature that $R^\loc$ is a power series ring over~$R^{\square, \psi_v}_{\rhobar_v}$ rather than the more general, possibly singular ring we have defined.
This is because it is usually assumed that $D_w^\times \cong \GL_2(F_w)$ for all~$w \mid p, w \ne v$.
In more detail, by~\cite[Proposition~3.5.1]{EGS}, for any~$\rhobar_w$ satisfying our assumptions there always exists a tame inertial type $\tau_w$, which could be principal series or cuspidal, 
such that $R^{\eta, \tau_w, \psi_w}_{\rhobar_w}$ is a power series ring over~$\cO$.
Then one work instead with $K_w := \GL_2(\cO_{F_w})$, and $\cL_w$ a $K_w$-type for the Bernstein component corresponding to~$\tau_w$.
When $D_w$ is a division algebra, this is only possible if~$\tau_w$ is cuspidal.
\end{rmk}

\subsection{Hecke algebras.}\label{Hecke algebras}
Let $H_v$ be the upper-triangular Iwahori subgroup of~$D_v^\times$ (if~$D_v$ is split) or the maximal compact subgroup of~$D_v^\times$ (if~$D_v$ is not split), 
and let $\chi : H_v \to \cO^\times$ be a smooth tame character.
Let~$K^v$ be as in Section~\ref{coefficients}, and let~$K_v \subset H_v$ be a compact open subgroup.

For all~$N \geq 1$, we introduce a subgroup $K(N)^v : = \prod_{w \ne v} K(N)_w$ by putting $K(N)_w  := K_w$ if~$w \not \in Q_N$, and for $w = q \in Q_N$, 
defining $K(N)_q \subset K_q = \GL_2(\cO_{F_q})$ as the kernel of the map
\[
I_q \to B(k_{F_q}) \to T(k_{F_q}) \xrightarrow{\diag(a, d) \mapsto ad^{-1}} k_{F_q}^\times \to k_{F_q}^\times[p^\infty]
\]
where~$I_q$ is the upper-triangular Iwahori subgroup of~$\GL_2(\cO_{F_q})$.

We define $\bT(K(N)^vK_v)$ as the $\cO$-subalgebra of $\End_{\cO}(S(K(N)^vK_v, \cL^v \otimes \chi))$ generated by $(T_w, S_w: w \not \in S \cup Q_N)$ and~$T_{\pi_{w_1}}$.
We write~$\fm_{Q_N}$ for the ideal of~$\bT(K(N)^v K_v)$ generated by 
\[
(T_{\pi_{w_1}} - \beta_{w_1}, T_w - (\bN w)\operatorname{tr}(\Frob_w),  S_w - (\bN w)\det(\Frob_w): w \not \in S \cup \{w_1\} \cup Q_N).
\]
It is a maximal ideal.
If~$q \in Q_N$, we have a Hecke operator
\[
U_{q} := K(N)_q \fourmatrix{\pi_q}{0}{0}{1} K(N)_q \in \cO[K(N)_q \backslash D_q^\times / K(N)_q],
\]
and we define $\bT(K(N)^vK_v)'$ as the $\bT(K(N)^vK_v)$-subalgebra of 
\[
\End_{\cO}\left (S(K(N)^vK_v, \cL^v \otimes \chi) \right )
\]
generated by~$U_q$ for~$q \in Q_{N}$.
We choose a maximal ideal of $\bT(K(N)^vK_v)'$ lying over~$\fm_{Q_N}$, and we denote it by~$\fm_{Q_N}'$.

Let~$\Delta_{Q_N} := \prod_{q \in Q_N}k_q^\times[p^\infty]$, and write~$S_N := \cO[\Delta_{Q_N}]$.
As usual, there is an $\cO$-linear ring homomorphism
\[
R_{F, S, Q_N}^\psi \to \bT(K(N)^vK_v)'_{\fm_{Q_N}'}
\]
such that, for all~$w \not \in S \cup Q_N \cup \{w_1\}$, the image of the characteristic polynomial of $\Frob_w$ on the universal deformation is
\[
X^2-(\bN w)^{-1}T_w X + (\bN w)^{-1}S_w.
\]
Local-global compatibility at the places in~$Q_N$ implies then that $\bT(K(N)^vK_v)'_{\fm'_{Q_N}}$ contains all the diamond operators~$\langle d \rangle$ for~$d \in \Delta_{Q_N}$, i.e. the double coset operators
\[
\langle d \rangle := K(N)_{Q_N} \fourmatrix {[d]} 0 0 1 K(N)_{Q_N}.
\]
\begin{lemma}\label{duality at finite level}
The $\cO$-algebra $\bT(K(N)^vH_v)'$ is commutative, and 
\[
S(K(N)^vH_v, \cL^v \otimes \chi)^* \cong S(K(N)^vH_v, \cL^v \otimes \chi)
\]
as $\bT(K(N)^vH_v)'$-modules.
\end{lemma}
\begin{proof}
The first claim is standard: $\bT(K(N)^v H_v)$ is the image of a polynomial ring over a spherical Hecke algebra, so it is commutative, 
and the~$U_q$ commute with each other and with $\bT(K(N)^v H_v)$. 
We now prove the second claim.
Lemma~\ref{self-duality for D*} and Remark~\ref{self-duality for tame characters} imply that $(\cL^v \otimes \chi)^* \otimes (\psi \circ \det)$ and $\cL^v \otimes \chi$ are conjugate 
under the normalizer of $H_v\prod_{w \in S}K(N)_w$ in $(D \otimes_F \bA_F^\infty)^\times$.
It follows that 
\[
S(K(N)^vH_v, \cL \otimes \chi) \cong S(K(N)^vH_v, (\cL \otimes \chi)^* \otimes (\psi \circ \det))
\] 
as $\bT(K(N)^vH_v)'$-modules. 
Hence~\eqref{duality pairing} becomes a perfect pairing
\[
    S(K(N)^vH_v, \cL \otimes \chi) \times S(K(N)^vH_v, \cL \otimes \chi) \to \cO
\]
such that $(Tx, y) = (x, \iota_w(\psi_w(\det T)T)y)$ for all~$w$ such that $K(N)_w$ acts trivially on~$\cL \otimes \chi$ and $\psi_w \circ \det$ is trivial on~$K(N)_w$, 
and all $T \in \cO[K(N)_w \backslash D_w^\times / K(N)_w]$.
Observe that $\psi_w \circ \det$ is trivial on~$K(N)_w$ for all~$w \not \in S$, because we are assuming that~$\psi_w$ is an unramified character for $s \not \in S \cup \{w_1\}$, and
a tamely ramified character when~$w = w_1$.

If~$w \not \in S \cup \{w_1\} \cup Q_N$, then $i_w(T_w) = S_w^{-1}T_w$ and~$i_w(S_w) = S_w^{-1}$.
Since~$S_w$ acts on~$S(K(N)^vH_v, \cL \otimes \chi)$ by~$\psi_w(\pi_w)$, we conclude that $T_w$ and~$S_w$ are self-adjoint with respect to the pairing.
If~$w \in Q_N$ or~$w = w_1$, let
\[
\Pi_w := \fourmatrix 0 1 {\pi_w} 0 \in \GL_2(\cO_{F_w}).
\] 
Then~$\Pi_w$ acts on $S(K(N)^vH_v, \cL \otimes \chi)$, since it normalizes $K(N)_w$,
and $\iota_w(\psi_w(\det u_w)U_w) = \Pi_w U_w \Pi_w^{-1}$ on $S(K(N)^vH_v, \cL \otimes \chi)$.
Hence the formula
\[
\langle x, y \rangle := (x, \prod\nolimits_{w \in Q_N \cup \{w_1\}}\Pi_w^{-1} y)
\]
defines a perfect duality on $S(K(N)^vH_v, \cL \otimes \chi)$ such that $\langle T x, y \rangle = \langle x, T y\rangle$ for all~$T \in \bT(K(N)^vH_v)'$.
This concludes the proof.
\end{proof}

\subsection{Patching.}
Let~$g := q - [F:\bQ] + |S| -1$, and fix a non-principal ultrafilter~$\mathfrak{F}$ on~$\bZ_{>0}$.
Define
\[
R^\psi_\infty := R^{\loc}[\![t_1, \ldots, t_g]\!].
\]
The defining properties of~$Q_N$ show that we can choose a surjection
\[
R_\infty^\psi \to R^{\square, \psi}_{F, S, Q_N}.
\]
Recall from~\eqref{framed deformation ring} that $R^{\square, \psi}_{F, S, Q_N}$ is an $\cO[\![z_1, \ldots, z_{4|S|-1}]\!]$-algebra because of the presence of framings at the places in~$S$.
We define
\[
M(K_v, N) := S(K^v(N)K_v, \cL^v)_{\fm_{Q_N}'}^* \otimes_{R^{\psi}_{F, S, Q_N}} R^{\square, \psi}_{F, S, Q_N},
\]
which is a finite free $\cO[\![z_i]\!]$-module. 
Recall the group algebra~$S_N := \cO[\Delta_{Q_N}]$ and its action on $M(K_v, N)$ via diamond operators.
Choose a surjective $\cO$-linear ring homomorphism
\[
\cO[\![y_1, \ldots, y_q]\!] \to S_N
\]
whose kernel is contained in the ideal generated by $((1+y_i)^{p^N}: 1 \leq i \leq q)$.
We give $\cO[\![y_i]\!]$ the profinite topology.
For every open ideal~$J \subset \cO[\![y_i]\!]$ 
define $\bN(J) \subset \bZ_{>0}$ as the set of~$N>0$ such that~$J$ contains the kernel of
$\cO[\![y_i]\!] \to S_N$.
If~$N \in \bN(J)$, define
\[
M(K_v, J, N) := M(K_v, N) \otimes_{\cO[\![y_i]\!]} \cO[\![y_i]\!]/J.
\]
If~$\fa \subset \cO[\![z_i]\!]$ is an open ideal, define
\[
M(K_v, J, N, \fa) := M(K_v, J, N) \otimes_{\cO[\![z_i]\!]} \cO[\![z_i]\!]/\fa.
\]
Fix~$\fa$ and~$J$ and introduce
\begin{gather*}
S_\infty := \cO[\![y_1, \ldots, y_q, z_1, \ldots, z_{4|S|-1}]\!]\\
S_{\infty, J, \fa} := \cO[\![z_i]\!]/\fa \otimes_\cO \cO[\![y_i]\!]/J\\
(S_{\infty, J, \fa})_{\bN(J)} := \prod_{N \in \bN(J)}S_{\infty, J, \fa}.
\end{gather*}
As usual, the ultrafilter~$\mathfrak{F}$ defines a prime ideal~$x$ of $(S_{\infty, J, \fa})_{\bN(J)}$, and we define
\[
M(K_v, J, \infty, \fa) := (S_{\infty, J, \fa})_{\bN(J), x} \otimes_{(S_{\infty, J, \fa})_{\bN(J)}} \prod_{N \in \bN(J)} M(K_v, J, N, \fa).
\]
This is a finite $S_{\infty, J, \fa}$-module with a diagonal action of $R^{\square, \psi}_{F, S, Q_N}$,
and it is isomorphic to $M(K_v, J, N, \fa)$ for infinitely many values of~$N$.
By e.g.~\cite[Lemma~3.4.11]{GN} given inclusions $J' \subset J, \fa' \subset \fa$ of open ideals, and a normal subgroup~$K_v' \subset K_v$, there is an isomorphism
\[
M(K_v', J', \infty, \fa') \otimes_{S_{\infty, J', \fa'}[K_v/K_v']} S_{\infty, J, \fa} \isom M(K_v, J, \infty, \fa)
\]
which makes~$M(K_v, J, \infty, \fa)$ a cofiltered system of $R_\infty^\psi \otimes_{\cO} S_\infty$-modules with surjective transition maps.
We define
\begin{equation}\label{definition of M_infty}
M_\infty := \varprojlim_{J, \fa, K_v} M(K_v, J, \infty, \fa).
\end{equation}
The inverse limit over~$K_v$ exhibits an action of $R_\infty^\psi[D_v^\times]$ on~$M_\infty$, 
and there exists an injective $\cO$-linear ring homomorphism $S_\infty \to R_\infty^\psi$ which induces
the inverse limit action of~$S_\infty$ on~$M_\infty$.
We omit the standard verification that the $R_\infty^\psi[D_v^\times]$-action is arithmetic and minimal, and we prove the self-duality property.

Let~$H_v \subset D_v^\times$ be the upper-triangular Iwahori subgroup (if~$D_v$ is split) or the maximal compact subgroup (if~$D_v$ is not split).
Let~$\chi : H_v \to \cO^\times$ be a smooth tame character. 
Since the modules in the inverse limit~\eqref{definition of M_infty} defining~$M_\infty$ have finite $\cO$-length, and $\chi$ is finitely presented over~$\cO[\![H_v]\!]$, we have that
\begin{equation}\label{presentation of M_infty(chi)}
M_\infty(\chi) := M_\infty \otimes_{\cO[\![H_v]\!]} \chi \cong \varprojlim_{J, \fa, K_v} \bigl( M(K_v, J, \infty, \fa) \otimes_{\cO[\![H_v]\!]} \chi \bigr).
\end{equation}
Since $M_\infty$ is a minimal arithmetic $R^\psi_\infty$-module, $M_\infty(\chi)$ is a Cohen--Macaulay $R_\infty^{\eta, \tau(\chi), \psi_w}$-module of generic rank one.
Since~$R_\infty^{\eta, \tau(\chi), \psi_w}$ is Gorenstein, its canonical module is trivial, and so need to prove that
\[
\Hom_{R_\infty^{\eta, \tau(\chi), \psi_w}}(M_\infty(\chi), R_\infty^{\eta, \tau(\chi), \psi_w}) \cong M_\infty(\chi)
\]
as~$R_\infty^{\eta, \tau(\chi), \psi_w}$-modules.
Since~$\dim R_\infty^{\eta, \tau(\chi), \psi_w} = \dim S_\infty$, by~\cite[Lemma~4.12]{Manning} it suffices to prove that $\Hom_{S_\infty}(M_\infty(\chi), S_\infty) \cong M_\infty(\chi)$ 
as $R_\infty^{\eta, \tau(\chi), \psi_w}$-modules, or equivalently as $R_\infty^\psi$-modules.
Since these modules are $\fm_{S_\infty}$-adically complete, it suffices to prove that the inverse systems $\Hom_{S_\infty}(M_\infty(\chi), S_\infty)/(J, \fa)$
and $M_\infty(\chi)/(J, \fa)$ (indexed by~$J$ and~$\fa$) are $R^\psi_\infty$-linearly isomorphic, and
since these modules have finite $\cO$-length, it suffices to prove that
\[
M_\infty(\chi)/(J, \fa) \cong \Hom_{S_\infty}(M_\infty(\chi), S_\infty)/(J, \fa)
\]
for all~$(J, \fa)$.
Since~$M_\infty(\chi)$ is finite free over~$S_\infty$, the natural map
\[
\Hom_{S_\infty}(M_\infty(\chi), S_\infty)/(J, \fa) \to \Hom_{S_{\infty, J, \fa}}(M_\infty(\chi)/(J, \fa), S_{\infty, J, \fa})
\]
is an $R_\infty^\psi$-linear isomorphism.
So it suffices to prove that
\[
M_\infty(\chi)/(J, \fa) \cong \Hom_{S_{\infty, J, \fa}}(M_\infty(\chi)/(J, \fa), S_{\infty, J, \fa}).
\]
By~\eqref{presentation of M_infty(chi)} we have
\[
M_\infty(\chi)/(J, \fa) \cong \varprojlim_{K_v}\bigl( M(K_v, J, \infty, \fa) \otimes_{\cO[\![H_v]\!]} \chi \bigr),
\]
and since $M_\infty(\chi)/(J, \mathfrak{a})$ has finite $\cO$-length, %
the limit at the right-hand side is attained at some~$K_v$.
Hence it suffices to prove that
\[
M(K_v, J, \infty, \fa) \otimes_{\cO[\![H_v]\!]} \chi \cong \Hom_{S_{\infty, J, \fa}}(M(K_v, J, \infty, \fa) \otimes_{\cO[\![H_v]\!]} \chi, S_{\infty, J, \fa})
\]
as $R_\infty^\psi$-modules, for all~$J, \fa, K_v$. %

By construction, there exists~$N \in \bN(J)$ such that
\begin{align*}
M(K_v, J, \infty, \fa) \otimes_{\cO[\![H_v]\!]} \chi &\cong M(K_v, J, N, \fa) \otimes_{\cO[\![H_v]\!]} \chi\\ 
&\cong \left ( \left ( S(K^v(N)K_v, \cL)^*_{\fm'_{Q_N}} \otimes_{\cO[\![H_v]\!]} \chi \right ) \otimes_{\cO} \cO[\![z_i]\!] \right ) \otimes_{S_{\infty}} S_{\infty, J, \fa}.
\end{align*}
Since~$N \in \bN(J)$, the map $S_\infty \to S_{\infty, J, \fa}$ factors through $S_N[\![z_i]\!]$.
So
\[
M(K_v, J, \infty, \fa) \otimes_{\cO[\![H_v]\!]} \chi \cong \left ( S(K^v(N)H_v, \cL\otimes_{\cO} \chi)^*_{\fm'_{Q_N}} \otimes_{\cO} \cO[\![z_i]\!] \right ) \otimes_{S_N[\![z_i]\!]} S_{\infty, J, \fa},
\]
and the $R_\infty^\psi$-action factors through the action of $\bT(K(N)^vH_v)[\![z_i]\!]$ on $S(K^v(N)H_v, \cL\otimes_{\cO} \chi)^*_{\fm'_{Q_N}}[\![z_i]\!]$.
Since $S(K^v(N)H_v, \cL\otimes_{\cO} \chi)^*_{\fm'_{Q_N}}$ is finite free over $S_N$, it thus suffices to prove that
\[
    \Hom_{S_N[\![z_i]\!]}(S(K^v(N)H_v, \cL\otimes_{\cO} \chi)^*_{\fm'_{Q_N}}[\![z_i]\!], S_N[\![z_i]\!]) \cong S(K^v(N)H_v, \cL\otimes_{\cO} \chi)^*_{\fm'_{Q_N}}[\![z_i]\!]
\]
as $\bT(K(N)^vH_v)'[\![z_i]\!]$-modules.
Again by~\cite[Lemma~4.12]{Manning} applied to $\cO[\![z_i]\!] \to S_N[\![z_i]\!]$, we reduce to proving
\[
    \Hom_{\cO}(S(K^v(N)H_v, \cL\otimes_{\cO} \chi)^*_{\fm'_{Q_N}}, \cO) \cong S(K^v(N)H_v, \cL\otimes_{\cO} \chi)^*_{\fm'_{Q_N}}
\]
as $\bT(K(N)^vH_v)'$-modules, which follows from Lemma~\ref{duality at finite level}.

%% file: formulas-arXiv.tex
\section{Formulas for Galois deformation rings.}

\subsection{$\tld z$-gauges.}\label{subsec: gauges}
If~$\tld z_j \in \Adm^\vee(3, 0)_j$, we make the following definitions:
\begin{enumerate}
\item If $\tld z_j = t_{(2, 1)_j}$ then
\[
T^{\tld z_j} := \cO[a_0, a_1, a_2^{\pm 1}, b_0, c_0, c_1, d_1^{\pm 1}, d_0]^\wedge_p, \;\;\; A^{\tld z_j} := \diag(a_2, d_1)\fourmatrix{(v+p)^2 + a_1(v+p) + a_0}{b_0}{v(c_1(v+p)+c_0)}{(v+p)+d_0}.
\]
\item If $\tld z_j = t_{(1, 2)_j}$ then
\[
T^{\tld z_j} := \cO[a_0, a_1^{\pm 1}, b_0, b_1, c_0, d_0, d_1, d_2^{\pm 1}]^\wedge_p, \;\;\; A^{\tld z_j} := \diag(a_1, d_2)\fourmatrix{(v+p)+a_0}{b_1(v+p)+b_0}{vc_0}{(v+p)^2+d_1(v+p)+d_0}.
\]
\item If $\tld z_j = t_{(1, 2)_j}s$ then
\[
T^{\tld z_j} := \cO[a_0, a_1, b_0, b_1^{\pm 1}, c_0, c_1^{\pm 1}, d_0, d_1]^\wedge_p, \;\;\; A^{\tld z_j} := \diag(b_1, c_1)\fourmatrix{a_1(v+p)+a_0}{(v+p)+b_0}{v((v+p)+c_0)}{d_1(v+p)+d_0}.
\]
\item If $\tld z_j = t_{(3, 0)_j}$ then
\[
T^{\tld z_j} := \cO[a_0, a_1, a_2, a_3^{\pm 1}, c_0, c_1, c_2, d_0^{\pm 1}]^\wedge_p, \;\;\; A^{\tld z_j} := \diag(a_3, d_0)\fourmatrix{(v+p)^3+a_2(v+p)^2+a_1(v+p)+a_0}{0}{v(c_2(v+p)^2+c_1(v+p)+c_0)}{1}.
\]
\item If $\tld z_j = t_{(0, 3)_j}$ then
\[
T^{\tld z_j} := \cO[a_0^{\pm 1}, b_0, b_1, b_2, d_0, d_1, d_2, d_3^{\pm 1}]^\wedge_p, \;\;\; A^{\tld z_j} := \diag(a_0, d_3)\fourmatrix{1}{b_2(v+p)^2+b_1(v+p)+b_0}{0}{(v+p)^3+d_2(v+p)^2+d_1(v+p)+d_0}.
\]
\item If $\tld z_j = t_{(0, 3)_j}s$ then
\[
T^{\tld z_j} := \cO[a_0, a_1, a_2, b_0^{\pm 1}, c_0, c_1, c_2^{\pm 1}, d_0]^\wedge_p, \;\;\; A^{\tld z_j} := \diag(b_0, c_2)\fourmatrix{a_2(v+p)^2+a_1(v+p)+a_0}{1}{v((v+p)^2+c_1(v+p)+c_0)}{d_0}.
\]
\item If $\tld z_j = t_{(2, 1)_j}s$ then
\[
T^{\tld z_j} := \cO[a_0, b_0, b_1, b_2^{\pm 1}, c_0^{\pm 1}, d_0, d_1, d_2]^\wedge_p, \;\;\; A^{\tld z_j} := \diag(b_2, c_0)\fourmatrix{a_0}{(v+p)^2+b_1(v+p)+b_0}{v}{d_2(v+p)^2+d_1(v+p)+d_0}.
\]
\end{enumerate}
\subsection{Single-type deformation rings.}\label{subsec: single-type approximation}
Let~$\tld z \in \Adm^\vee\lweight$, and let $T^{\tld z} = \prod_{j \in \mJ} T^{\tld z_j}$.
Let~$\tau$ be an inertial $\cO$-type with 9-deep lowest alcove presentation~$(w, \nu)$.
If~$j \in \cJ$, we define the following ideals of~$T^{\tld z_j}$, where~$\kappa^{\tau}_j$ is the structure constant in Lemma~\ref{structure constants}.
We also put $I_{\nabla_1}^{\leq (2, 1)_j, \tau} = I_{\nabla_1}^{(2, 1)_j, \tau}$.

\begin{enumerate}
\item If~$\tld z_j = t_{(2, 1)_j}$ then
\begin{align*}
I^{\leq (3, 0)_j, \tau}_{\nabla_1} :=\; & (a_1+(\kappa_j^{\tau}-2)b_0c_1, d_0-(\kappa_j^{\tau}-1)b_0c_1, \kappa_j^{\tau} c_0-(\kappa_j^{\tau}-1)(\kappa_j^{\tau}-2)b_0c_1^2, \\
& \kappa_j^{\tau} a_0-b_0c_1((\kappa_j^{\tau}-1)^2(\kappa_j^{\tau}-2)b_0c_1-p\kappa_j^{\tau}), b_0((\kappa_j^{\tau}-1)(\kappa_j^{\tau}-2)b_0c_1-2p))\\
I^{(2, 1)_j, \tau}_{\nabla_1} :=\; & (b_0, a_1, d_0, c_0, a_0)\\
I^{(3, 0)_j, \tau}_{\nabla_1} :=\; &(a_1+(\kappa_j^{\tau}-1)^{-1}2p, d_0-(\kappa_j^{\tau}-2)^{-1}2p,\kappa_j^{\tau} c_0-2pc_1, \\
& \kappa_j^{\tau}a_0-(\kappa_j^{\tau}-1)^{-1}2p^2, (\kappa_j^{\tau}-1)(\kappa_j^{\tau}-2)b_0c_1-2p).
\end{align*}
\item If~$\tld z_j = t_{(1, 2)_j}$ then
\begin{align*}
I^{\leq (3, 0)_j, \tau}_{\nabla_1} :=\; & (d_1+(\kappa_j^{\tau}-2)c_0b_1, a_0-(\kappa_j^{\tau}-1)c_0b_1, \kappa_j^{\tau} b_0-(\kappa_j^{\tau}-1)(\kappa_j^{\tau}-2)c_0b_1^2, \\
& \kappa_j^{\tau} d_0-c_0b_1((\kappa_j^{\tau}-1)^2(\kappa_j^{\tau}-2)c_0b_1-p\kappa_j^{\tau}), c_0((\kappa_j^{\tau}-1)(\kappa_j^{\tau}-2)c_0b_1-2p)).\\
I^{(2, 1)_j, \tau}_{\nabla_1} :=\; & (c_0, d_1, a_0, b_0, d_0)\\
I^{(3, 0)_j, \tau}_{\nabla_1} :=\; &(d_1+(\kappa_j^{\tau}-1)^{-1}2p, a_0-(\kappa_j^{\tau}-2)^{-1}2p, \kappa_j^{\tau} b_0-2pb_1, \\
& \kappa_j^{\tau}d_0-(\kappa_j^{\tau}-1)^{-1}2p^2, (\kappa_j^{\tau}-1)(\kappa_j^{\tau}-2)c_0b_1-2p).
\end{align*}
\item If~$\tld z_j = t_{(1, 2)_j}s$ then
\begin{align*}
I^{\leq (3, 0)_j, \tau}_{\nabla_1} :=\; & (c_0+(\kappa^{\tau}_j-1)(a_1d_1+p), b_0-\kappa^{\tau}_j(a_1d_1+p), (\kappa^{\tau}_j+1)a_0+\kappa^{\tau}_j(\kappa^{\tau}_j-1)a_1(a_1d_1+p),\\
& (\kappa^{\tau}_j-2)d_0+\kappa^{\tau}_j(\kappa^{\tau}_j-1)d_1(a_1d_1+p), (a_1d_1+p)(\kappa^{\tau}_j(\kappa^{\tau}_j-1)a_1d_1+(\kappa^{\tau}_j-2)(\kappa^{\tau}_j+1)p)) \\
I^{(2, 1)_j, \tau}_{\nabla_1} :=\; & (c_0, b_0, a_0, d_0, a_1d_1+p)\\
I^{(3, 0)_j, \tau}_{\nabla_1} :=\; & (c_0+\frac 2 {\kappa^{\tau}_j}, b_0-\frac 2{\kappa^{\tau}_j-1}, (\kappa^{\tau}_j +1) a_0+2pa_1, (\kappa^{\tau}_j-2)d_0+2pd_1, (a_1d_1+p)-\frac 2 {\kappa^{\tau}_j(\kappa^{\tau}_j-1)}p)).
\end{align*}
\item If~$\tld z_j = t_{(3, 0)_j}$ then
\begin{align*}
I^{\leq (3, 0)_j, \tau}_{\nabla_1} :=\; & (a_0, a_1, a_2, (\kappa_j^{\tau}-2)c_1+2pc_2, (\kappa_j^{\tau}-1)c_0+pc_1) \\
I^{(2, 1)_j, \tau}_{\nabla_1} :=\; & T^{\tau_j}\\
I^{(3, 0)_j, \tau}_{\nabla_1} :=\; & I^{\leq (3, 0)_j, \tau}_{\nabla_1}.
\end{align*}
\item If~$\tld z_j = t_{(0, 3)_j}$ then
\begin{align*}
I^{\leq (3, 0)_j, \tau}_{\nabla_1} :=\; & (d_0, d_1, d_2, (\kappa^{\tau}_j-2)b_1+2pb_2, (\kappa^{\tau}_j-1)b_0+pb_1 ) \\
I^{(2, 1)_j, \tau}_{\nabla_1} :=\; & T^{\tau}\\
I^{(3, 0)_j, \tau}_{\nabla_1} :=\; & I^{\leq (3, 0)_j, \tau}_{\nabla_1}.
\end{align*}

\item If~$\tld z_j = t_{(0, 3)_j}s$ then
\begin{align*}
I^{\leq (3, 0)_j, \tau}_{\nabla_1} :=\; & (d_0a_2-(c_1-p), (\kappa_j^{\tau}-1)a_1+2pa_2, pa_1+\kappa_j^{\tau} a_0, (\kappa_j^\tau-2)c_1+2p, (\kappa_j^\tau-1)c_0+pc_1) \\
I^{(2, 1)_j, \tau}_{\nabla_1} :=\; & T^{\tau_j}& \\
I^{(3, 0)_j, \tau}_{\nabla_1} :=\; & I^{\leq (3, 0), \tau}_{\nabla_1}.
\end{align*}

\item If~$\tld z_j = t_{(2, 1)_j}s$ then
\begin{align*}
I^{\leq (3, 0)_j, \tau}_{\nabla_1} :=\; & (a_0d_2-(b_1-p), (\kappa_j^{\tau}-1)d_1+2pd_2, pd_1+\kappa_j^{\tau} d_0, (\kappa_j^\tau-2)b_1+2p, (\kappa_j^\tau-1)b_0+pb_1) \\
I^{(2, 1)_j, \tau}_{\nabla_1} :=\; & T^{\tau_j}& \\
I^{(3, 0)_j, \tau}_{\nabla_1} :=\; & I^{\leq (3, 0), \tau}_{\nabla_1}. 
\end{align*}

\end{enumerate}

\subsection{Local charts for~$\PhiMod_{K, 2}^{\et}$.}\label{subsec: charts for PhiMod}
Choose $\tld u \in \Adm^\vee \sweight$.
For all~$j \in \cJ$, we are going to construct a ring $S^{\tld u_j}$ and matrices $\Psi^{\tld u_j} \in M_2(S^{\tld u_j}[v+p])$, 
and for all $\tld z_j \in \{\tld u_j, \tld u_jt_{\pm\alpha_j}\}$,
a surjection
\begin{gather*}
\pr_{\tld z_j}: S^{\tld u_j} \to T^{\tld z_j}
\end{gather*}
such that
$\pr_{\tld z_j}(\Psi^{\tld u_j}) = A^{\tld z_j} \tld z_j^{-1}$.
It will be convenient to translate these matrices by the central element $t_{(3, 3)_j} = v^3$.
Then:
\begin{enumerate}
    \item If~$\tld z_j = t_{(2, 1)_j}$ then 
    \[
    A^{\tld z_j} \tld z_j^{-1}v^3 = \diag(a_2, d_2)\fourmatrix{v((v+p)^2 + a_1(v+p) + a_0)}{v^2b_0}{v^2(c_1(v+p)+c_0)}{v^2((v+p)+d_0)}.
    \]
    \item If~$\tld z_j = t_{(1, 2)_j}$ then 
    \[
    A^{\tld z_j} \tld z_j^{-1}v^3 = \diag(a_1, d_2)\fourmatrix{v^2((v+p)+a_0)}{v(b_1(v+p)+b_0)}{v^3c_0}{v((v+p)^2+d_1(v+p)+d_0)}
    \]
    \item If~$\tld z_j = t_{(1, 2)_j}s$ then 
    \[
    A^{\tld z_j} \tld z_j^{-1}v^3 = \diag(b_1, c_1)\fourmatrix{v^2((v+p)+b_0)}{v(a_1(v+p)+a_0)}{v^2(d_1(v+p)+d_0)}{v^2((v+p)+c_0)}
    \]
    \item If~$\tld z_j = t_{(3, 0)_j}$ then 
    \[
    A^{\tld z_j} \tld z_j^{-1}v^3 = \diag(a_3, d_0)\fourmatrix{(v+p)^3+a_2(v+p)^2+a_1(v+p)+a_0}{0}{v(c_2(v+p)^2+c_1(v+p)+c_0)}{v^3}
    \]
    \item If~$\tld z_j = t_{(0, 3)_j}$ then 
    \[
    A^{\tld z_j} \tld z_j^{-1}v^3 = \diag(a_0, d_3)\fourmatrix{v^3}{b_2(v+p)^2+b_1(v+p)+b_0}{0}{(v+p)^3+d_2(v+p)^2+d_1(v+p)+d_0}
    \]
    \item If~$\tld z_j = t_{(0, 3)_j}s$ then 
    \[
    A^{\tld z_j} \tld z_j^{-1}v^3 = \diag(b_0, c_2)\fourmatrix{v^3}{a_2(v+p)^2+a_1(v+p)+a_0}{v^3d_0}{v((v+p)^2+c_1(v+p)+c_0)}
    \]
    \item If~$\tld z_j = t_{(2, 1)_j}s$ then 
    \[
    A^{\tld z_j} \tld z_j^{-1}v^3 = \diag(b_2, c_0)\fourmatrix{v((v+p)^2+b_1(v+p)+b_0)}{v^2a_0}{v(d_2(v+p)^2+d_1(v+p)+d_0)}{v^3}.
    \]
\end{enumerate}
Then~$S^{\tld u_j}$ and the matrix $\Psi^{\tld u_j} \in M_2(S^{\tld u_j})$ are defined as follows.
To shorten the discussion, we will restrict to the case $\tld u_j \in \{t_{(2, 1)_j}, t_{(1, 2)}s\}$.
The missing case~$\tld u_j = t_{(1, 2)_j}$ can then be obtained from $t_{(1, 2)_j}$ by conjugation by the Iwahori-normalizing matrix $st_{(1, 0)_j}$.

\begin{enumerate}
\item If $\tld u_j = t_{(2, 1)_j}$ then
\[S^{\tld u_j} = \cO[\alpha^{\pm 1}, \delta^{\pm 1}, \alpha_0, \alpha_1, \alpha_2, \beta_0, \beta_1, \gamma_0, \gamma_1, \gamma_2, \delta_0, \delta_1]^\wedge_p\]
and
\[\Psi^{\tld u_j}v^3 = \diag(\alpha, \delta)\fourmatrix{(v+p)^3+\alpha_2(v+p)^2+\alpha_1(v+p)+\alpha_0}{v(\beta_1(v+p)+\beta_0)}{v(\gamma_2(v+p)^2+\gamma_1(v+p)+\gamma_0)}
{v((v+p)^2+\delta_1(v+p)+\delta_0)}.\]
\item If~$\tld u_j = t_{(1, 2)}s$ then 
\[S^{\tld u_j} = \cO[\alpha^{\pm 1}, \delta^{\pm 1}, \alpha_0, \alpha_1, \beta_0, \beta_1, \beta_2, \gamma_0, \gamma_1, \gamma_2, \delta_0, \delta_1]^\wedge_p\]
and
\[\Psi^{\tld u_j}v^3 = 
\diag(\alpha, \delta)\fourmatrix{v((v+p)^2+\alpha_1(v+p)+\alpha_0)}{\beta_2(v+p)^2+\beta_1(v+p)+\beta_0}{v(\gamma_2(v+p)^2+\gamma_1(v+p)+\gamma_0)}{v((v+p)^2+\delta_1(v+p)+\delta_0)}.\]
\end{enumerate}

The surjection $\pr_{\tld z_j}$
is defined as follows:
\begin{enumerate}
    \item If~$\tld u_j = \tld z_j = t_{(2, 1)_j}$ then
    \begin{multline*}
    \pr_{\tld z_j} : (\alpha, \delta, \alpha_2, \alpha_1, \alpha_0, \beta_1, \beta_0, \gamma_2, \gamma_1, \gamma_0, \delta_1, \delta_0) \mapsto \\
    (a_2, d_2, a_1-p, a_0-pa_1, -pa_0, b_0, -pb_0, c_1, c_0-pc_1, -pc_0, d_0-p, -pd_0)
\end{multline*}
    \item If~$\tld u_j = t_{(2, 1)_j}$ and $\tld z_j = t_{(1, 2)_j}$ then
    \begin{multline*}
    \pr_{\tld z_j}: (\alpha, \delta, \alpha_2, \alpha_1, \alpha_0, \beta_1, \beta_0, \gamma_2, \gamma_1, \gamma_0, \delta_1, \delta_0) \mapsto \\
    (a_1, d_2, a_0-2p, -2pa_0+p^2, p^2a_0, b_1, b_0, c_0, -2pc_0, p^2c_0, d_1, d_0)
    \end{multline*}
    \item If $\tld u_j = t_{(2, 1)_j}$ and $\tld z_j = t_{(3, 0)_j}$ then
    \begin{equation*}
    \pr_{\tld z_j}: (\alpha, \delta, \alpha_2, \alpha_1, \alpha_0, \beta_1, \beta_0, \gamma_2, \gamma_1, \gamma_0, \delta_1, \delta_0) \mapsto
    (a_3, d_0, a_2, a_1, a_0, 0, 0, c_2, c_1, c_0, -2p, p^2)
    \end{equation*}
    \item If $\tld u_j = \tld z_j = t_{(1, 2)_j}s$ then
    \begin{multline*}
    \pr_{\tld z_j} : (\alpha, \delta, \alpha_0, \alpha_1, \beta_0, \beta_1, \beta_2, \gamma_0, \gamma_1, \gamma_2, \delta_0, \delta_1) \mapsto\\
    (b_1, c_1, -pb_0, b_0-p, -pa_0, a_0-pa_1, a_1, -pd_0, d_0-pd_1, d_1, -pc_0, c_0-p)
    \end{multline*}
    \item If $\tld u_j = t_{(1, 2)_j}s$ and~$\tld z_j = t_{(0, 3)_j}s$ then
    \[\pr_{\tld z_j} : (\alpha, \delta, \alpha_0, \alpha_1, \beta_0, \beta_1, \beta_2, \gamma_0, \gamma_1, \gamma_2, \delta_0, \delta_1) \mapsto
    (b_0, c_2, p^2, -2p, a_0, a_1, a_2, p^2d_0, -2d_0p, d_0, c_0, c_1)
    \]
    \item If $\tld u_j = t_{(1, 2)_j}s$ and~$\tld z_j = t_{(2, 1)_j}s$ then 
    \[\pr_{\tld z_j} : (\alpha, \delta, \alpha_0, \alpha_1, \beta_0, \beta_1, \beta_2, \gamma_0, \gamma_1, \gamma_2, \delta_0, \delta_1) \mapsto 
    (b_2, c_0, b_0, b_1, p^2a_0, -2pa_0, a_0, d_0, d_1, d_2, p^2, -2p).
    \]
\end{enumerate}

\subsection{Multi-type deformation rings.}\label{subsec: multi-type approximation}
Recall from the previous section the projection $\pr_{\tld z_j}: S^{\tld u_j} \to T^{\tld z_j}$, and define
\[
K^{\lambda_j, \tau}_{\nabla_1} = \pr_{\tld z_j}^{-1}(I^{\lambda_j, \tau}_{\nabla_1}) \subset S^{\tld u_j}.
\]
We write down these ideals in the cases we will need for our arguments.

\begin{enumerate}
    \item If~$\tld u_j = \tld z_j = t_{(2, 1)_j}$ then
    \begin{align*}
        K^{(2, 1)_j, \tau_{\tld z}}_{\nabla_1} = & (\alpha_0+p\alpha_1+p^2\alpha_2+p^3, \beta_0+p\beta_1, \gamma_0+p\gamma_1+p^2\gamma_2, \delta_0+p\delta_1+p^2, \\
        & \beta_1, \alpha_2+p, \delta_1+p, \gamma_1+p\gamma_2, \alpha_1)
    \end{align*}
    \item If~$\tld u_j = t_{(2, 1)_j}$ and $\tld z_j = t_{(1, 2)_j}$ then
    \begin{align*}
        K^{(3, 0)_j, \tau_{\tld z}}_{\nabla_1} = & (\alpha_0+p\alpha_1+p^2\alpha_2+p^3, \alpha_1+2p\alpha_2+3p^2, \gamma_0+p\gamma_1+p^2\gamma_2, \gamma_1+2p\gamma_2 \\
        & \gamma_2\beta_1-2p(\kappa^{\tau_{\tld z}}_j-1)^{-1}(\kappa^{\tau_{\tld z}}_j-2)^{-1}, \delta_1+(\kappa^{\tau_{\tld z}}_j-1)^{-1}2p, \\
        & \alpha_2+(1-(\kappa^{\tau_{\tld z}}_j-2)^{-1})2p, \kappa^{\tau_{\tld z}}_j\beta_0-2p\beta_1, \kappa^{\tau_{\tld z}}_j\delta_0-(\kappa^{\tau_{\tld z}}_j-1)^{-1}2p^2)
        \end{align*}
    \begin{align*}
        K^{(2, 1)_j, \tau_{\tld z}}_{\nabla_1} = & (\alpha_0+p\alpha_1+p^2\alpha_2+p^3, \alpha_1+2p\alpha_2+3p^2, \gamma_0+p\gamma_1+p^2\gamma_2, \gamma_1+2p\gamma_2, \\
        & \gamma_2, \delta_1, \alpha_2+2p, \beta_0, \delta_0)
    \end{align*}
    \item If~$\tld u_j = t_{(2, 1)_j}$ and $\tld z_j = t_{(3, 0)_j}$ then
        \[
        K^{(3, 0)_j, \tau_{\tld z}}_{\nabla_1} = (\beta_1, \beta_0,  \delta_1+2p, \delta_0-p^2, \alpha_2, \alpha_1, \alpha_0, (\kappa_j^{\tau_{\tld z}}-2)\gamma_1+2p\gamma_2, 
        (\kappa^{\tau_{\tld z}}_j-1)\gamma_0+p\gamma_1)
        \]
    \item If $\tld u_j = \tld z_j = t_{(1, 2)_j}s$ then
    \begin{align*}
    K^{(2, 1)_j, \tau_{\tld z}}_{\nabla_1} = & (\alpha_0+p\alpha_1+p^2, \beta_0+p\beta_1+p^2\beta_2, \gamma_0+p\gamma_1+p^2\gamma_2, \delta_0+p\delta_1+p^2\\
    & \delta_1+p, \alpha_1+p, \beta_1+p\beta_2, \gamma_1+p\gamma_2, \beta_2\gamma_2+p) =\\
    & (\alpha_0, \alpha_1+p, \beta_0, \beta_1+p\beta_2, \gamma_0, \gamma_1+p\gamma_2, \delta_0, \delta_1+p, \beta_2\gamma_2+p) 
    \end{align*}
    \item If $\tld u_j = t_{(1, 2)_j}s$ and~$\tld z_j = t_{(0, 3)_j}s$ then
    \begin{align*}
    K^{(3, 0)_j, \tau_{\tld z}}_{\nabla_1} = & (\alpha_0-p^2, \alpha_1+2p, \gamma_0-p^2\gamma_2, \gamma_1+2p\gamma_2, \\
    & \gamma_2\beta_2-(\delta_1-p), (\kappa^{\tau_{\tld z}}_j-1)\beta_1+2p\beta_2, p\beta_1+\kappa^{\tau_{\tld z}}_j \beta_0, (\kappa^{\tau_{\tld z}}_j-2)\delta_1 + 2p, (\kappa_j^{\tau_{\tld z}}-1)\delta_0+p\delta_1)
    \end{align*}
    \item If $\tld u_j = t_{(1, 2)_j}s$ and~$\tld z_j = t_{(2, 1)_j}s$ then
    \begin{align*}
    K^{(3, 0)_j, \tau_{\tld z}}_{\nabla_1} = & (\delta_0-p^2, \delta_1+2p, \beta_0-p^2\beta_2, \beta_1+2p\beta_2, \\
    & \beta_2\gamma_2-(\alpha_1-p), (\kappa^{\tau_{\tld z}}_j-1)\gamma_1+2p\gamma_2, p\gamma_1+\kappa^{\tau_{\tld z}}_j \gamma_0, (\kappa^{\tau_{\tld z}}_j-2)\alpha_1 + 2p, 
    (\kappa_j^{\tau_{\tld z}}-1)\alpha_0+p\alpha_1)
    \end{align*}
\end{enumerate}